\documentclass[11pt,a4paper]{article}
\usepackage[T1]{fontenc}
\usepackage[francais,english]{babel}
\usepackage[applemac]{inputenc}
\usepackage{pb-diagram}
\usepackage[vmargin=3.5cm]{geometry}%
\usepackage{color}%
\usepackage{mdwlist}%
\usepackage{graphicx}%
\usepackage{float}%
\usepackage{makeidx}%
\usepackage{amsmath}%
\usepackage{amssymb}%
\usepackage{amsthm}%
\usepackage{amsfonts}%
\usepackage{mathrsfs}%
\usepackage{bbm}%
\usepackage[all,cmtip]{xy}%
\usepackage{rotate}%
\usepackage{yfonts}%
\usepackage{array}%
\newtheorem{Prop}{Proposition}[section]%
\newtheorem{Conj}[Prop]{Conjecture}%
\newtheorem{TheoEnglish}{Theorem}[section]%
\newtheorem{DefEnglish}[Prop]{Definition}
\newtheorem{DefProp}[Prop]{Proposition-Definition}
\newtheorem{CorEnglish}[Prop]{Corollary}
\newtheorem{LemEnglish}[Prop]{Lemma}
\newtheorem{HypEnglish}[Prop]{Assumption}%
\newcommand{\A}{\mathbb A}%
\newcommand{\C}{\mathbb C}%
\newcommand{\Fp}{\mathbb F}%
\newcommand{\G}{\mathbf G}%

\newcommand{\N}{\mathbb N}%
\newcommand{\Q}{\mathbb Q}%
\newcommand{\qp}{\mathbb Q_{p}}%
\newcommand{\R}{\mathbb R}%
\newcommand{\Z}{\mathbb Z}%
\newcommand{\zp}{\mathbb Z_{p}}%
\newcommand{\Ccal}{\mathcal C}%
\newcommand{\Fcal}{\mathcal F}%
\newcommand{\Hcal}{\mathcal H}%
\newcommand{\Mcal}{\mathcal{M}}%
\newcommand{\Ncal}{\mathcal N}%
\newcommand{\Ocal}{\mathcal O}%
\newcommand{\Pcal}{\mathcal P}%
\newcommand{\Kcal}{\mathcal K}%
\newcommand{\Vcal}{\mathcal V}%
\newcommand{\Wcal}{\mathcal W}%
\newcommand{\Xcali}{\mathscr X}%
\newcommand{\aid}{\mathfrak a}%
\newcommand{\cid}{\mathfrak c}%
\newcommand{\pid}{\mathfrak p}%
\newcommand{\mgot}{\mathfrak m}%
\newcommand{\Tbar}{\overline{T}}%
\newcommand{\sldeux}{\operatorname{SL}_{2}}%
\newcommand{\gldeux}{\operatorname{GL}_{2}}%
\newcommand{\GL}{\operatorname{GL}}%

\newcommand{\Sh}{\operatorname{Sh}}
\newcommand{\et}{\operatorname{et}}
\newcommand{\ur}{\operatorname{ur}}
\newcommand{\diamant}[1]{\mathopen<\,#1\,\mathclose>}%
\newcommand{\somme}[2]{\underset{#1}{\overset{#2}\sum}}%
\newcommand{\produit}[2]{\underset{#1}{\overset{#2}\prod}}%
\newcommand{\produittenseur}[2]{\underset{#1}{\overset{#2}\bigotimes}}%
\newcommand{\sommedirecte}[2]{\underset{#1}{\overset{#2}\bigoplus}}%
\newcommand{\applicationsimple}[3]{\begin{equation}%
\nonumber%
#1 :#2\longrightarrow #3%
\end{equation}}%
\newcommand{\applicationsimplelabel}[4][]{\begin{equation}%
\label{#1}%
#2 :#3\longrightarrow #4%
\end{equation}}%
\newcommand{\suiteexacte}[5]{0\fleche#3\overset{#1}{\fleche}#4\overset{#2}{\fleche}#5\fleche0}

\newcommand{\limproj}[1]{\underset{\underset{#1}\longleftarrow}\lim}

\newcommand{\Hom}{\operatorname{Hom}}

\newcommand{\isom}{\overset{\sim}{\longrightarrow}}
\newcommand{\image}{\operatorname{Im}}

\newcommand{\plonge}{\hookrightarrow}
\newcommand{\matrice}[4]{\begin{pmatrix}#1&#2\\ #3&#4\end{pmatrix}}

\newcommand{\rank}{\operatorname{rank}}%
\newcommand{\ord}{\operatorname{ord}}
\newcommand{\fl}{\operatorname{fl}}

\newcommand{\tenseur}{\otimes}
\newcommand{\Ltenseur}{\overset{\operatorname{L}}{\tenseur}}
\newcommand{\modulo}{\operatorname{ mod }}
\newcommand{\Spec}{\operatorname{Spec}}
\newcommand{\Id}{\operatorname{Id}}%
\newcommand{\Aut}{\operatorname{Aut}}%
\newcommand{\End}{\operatorname{End}}%
\newcommand{\Frac}{\operatorname{Frac}}%
\newcommand{\Tate}{\operatorname{Ta}}%
\newcommand{\Cone}{\operatorname{Cone}}%
\newcommand{\fleche}{\longrightarrow}%
\newcommand{\croix}{^{\times}}%
\newcommand{\idele}[1]{\widehat{#1}^{\times}}%
\newcommand{\surjection}{\twoheadrightarrow}%
\newcommand{\rhobar}{\bar{\rho}}%
\newcommand{\vide}{\varnothing}%
\newcommand{\red}{\operatorname{red}}
\newcommand{\new}{\operatorname{new}}
\newcommand{\Sym}{\operatorname{Sym}}
\newcommand{\dR}{\operatorname{dR}}%
\newcommand{\Hun}{H^{1}}

\newcommand{\RGamma}{\operatorname{R}\Gamma}%
\newcommand{\Det}{\operatorname{{D}et}}%
\newcommand{\tr}{\operatorname{tr}}
\newcommand{\Fr}{\operatorname{Fr}}%
\newcommand{\Gal}{\operatorname{Gal}}

\newcommand{\Qbar}{\bar{\Q}}%
\newcommand{\s}{\sigma}%
\newcommand{\hgot}{\mathfrak h}%
\newcommand{\Hecke}{\mathbf{T}}%
\newcommand{\Eul}{\operatorname{Eul}}
\DeclareFontEncoding{OT2}{}{} 
\newcommand{\Nekovar}{Nekov\'a\v{r}}%
\newcommand{\cl}{\operatorname{cl}}
\numberwithin{equation}{subsection}%
\date{}
\begin{document}%
\title{The Equivariant Tamagawa Number Conjecture for modular motives with coefficients in the Hecke algebra}
\author{Olivier Fouquet}%
\maketitle
\begin{abstract}
We propose a formulation of the Equivariant Tamagawa Number Conjecture for modular motives with coefficients in universal deformation rings and Hecke algebras; something which seems to have been heretofore missing because the complexes of Galois cohomology required were not known to be perfect. We show that the fundamental line of this conjecture satisfies the expected compatibility property at geometric points (more precisely at the points satisfying the Weight-Monodromy conjecture) and is compatible with level-lowering and level-raising. Combining these properties with the methods of Euler and Taylor-Wiles systems, we prove a significant part of the ETNC with coefficients in Hecke algebras for motives attached to modular forms.  
\end{abstract}
%
\selectlanguage{english}%

\newcommand{\hord}{\mathfrak h^{\ord}}%
\newcommand{\hdual}{\mathfrak h^{dual}}%
\newcommand{\matricetype}{\begin{pmatrix}\ a&b\\ c&d\end{pmatrix}}%
\newcommand{\Iw}{\operatorname{Iw}}%
\newcommand{\Hi}{\operatorname{Hi}}
\newcommand{\cyc}{\operatorname{cyc}}
\newcommand{\ab}{\operatorname{ab}}
\newcommand{\can}{\operatorname{can}}%
\newcommand{\Fitt}{\operatorname{Fitt}}%
\newcommand{\Tiwa}{\mathcal T_{\operatorname{Iw}}}%
\newcommand{\Af}{\operatorname{A}}%
\newcommand{\Dunzero}{D_{1,0}}%
\newcommand{\Uun}{U_{1}}%
\newcommand{\Uzero}{U_{0}}%
\newcommand{\Uundual}{U^{1}}%
\newcommand{\Uunun}{U^{1}_{1}}%
\newcommand{\Wdual}{\Wcal^{dual}}
\newcommand{\JunNps}{J_{1,0}(\Ncal, P^{s})}%
\newcommand{\Tatepord}{\Tate_{\pid}^{ord}}%
\newcommand{\Kum}{\operatorname{Kum}}%
\newcommand{\zcid}{z(\cid)}%
\newcommand{\kgtilde}{\tilde{\kappa}}%
\newcommand{\kiwa}{\varkappa}%
\newcommand{\kiwatilde}{\tilde{\varkappa}}%
\newcommand{\Hbar}{\bar{H}}%
\newcommand{\Tred}{T/\mgot T}%
\newcommand{\Riwa}{R_{\operatorname{Iw}}}%
\newcommand{\Kiwa}{\Kcal_{\operatorname{Iw}}}%
\newcommand{\Sp}{\mathbf{Sp}}%
\newcommand{\Aiwa}{\mathcal A_{\operatorname{Iw}}}%
\newcommand{\Viwa}{\mathcal V_{\operatorname{Iw}}}%
\newcommand{\pseudiso}{\overset{\centerdot}{\isom}}%
\newcommand{\pseudisom}{\overset{\approx}{\fleche}}%
\newcommand{\carac}{\operatorname{char}}%
\newcommand{\length}{\operatorname{length}}
\newcommand{\eord}{e^{\ord}}%
\newcommand{\eordm}{e^{\ord}_{\mgot}}%
\newcommand{\hordinfini}{\hord_{\infty}}%
\newcommand{\Mordinfini}{M^{\ord}_{\infty}}%
\newcommand{\hordm}{\hord_{\mgot}}%
\newcommand{\hminm}{\hgot^{min}_{\mgot}}%
\newcommand{\Mordm}{M^{\ord}_{\mgot}}%
\newcommand{\Mtwist}{M^{tw}_{\mgot}}
\newcommand{\Xun}{X_{1}}%
\newcommand{\Xundual}{X^{1}}%
\newcommand{\Xunun}{X^{1}_{1}}%
\newcommand{\Xtw}{X^{tw}}
\newcommand{\Inert}{\mathfrak{In}}%
\newcommand{\Tsp}{T_{\Sp}}%
\newcommand{\Asp}{A_{\Sp}}%
\newcommand{\Vsp}{V_{\Sp}}%
\newcommand{\SK}{\mathscr{S}}%
\newcommand{\Rord}{R^{\ord}}%
\newcommand{\per}{\operatorname{per}}
\newcommand{\z}{\mathbf{z}}
\newcommand{\zs}{\tilde{\mathbf{z}}}
\newcommand{\Ebarbar}{\bar{\bar{E}}}
\newcommand{\Grsym}{\mathfrak S}
\newcommand{\epsi}{\varepsilon}
\newcommand{\Fun}[2]{F^{{\mathbf{#1}}}_{#2}}
\newtheorem*{TheoA}{Theorem A}%
\newtheorem*{TheoB}{Theorem B}%
\newcommand{\isocan}{\overset{\can}{\simeq}}
\bigskip{\footnotesize%
  \textsc{Départment de Mathématiques, Bâtiment 425, Faculté des sciences d'Orsay Université Paris-Sud} \par  
  \textit{E-mail address}: \texttt{olivier.fouquet@math.u-psud.fr} \par
  \textit{Telephone number}: \texttt{+33169155729} \par
  \textit{Fax number}: \texttt{+33169156019}
  }
\tableofcontents
\section{Introduction}
\subsection{Motivation}
\paragraph{Tamgawa Number Conjectures for modular motives}Let $f\in S_{k}(\Gamma_{1}(N))$ be an eigencuspform of weight $k\geq2$ with coefficients in a number field $F$. To $f$ is attached in $\cite{SchollMotivesModular}$ a Grothendieck motive $M(f)$ over $\Q$ with coefficients in $\Ocal_{F}$ whose partial $L$-function  
\begin{equation}\nonumber
L_{S}(M(f),s)=\produit{\ell\notin S}{}\Eul_{\ell}(M(f)_{\et,p},\ell^{-s})
\end{equation}
relative to a finite (possibly empty) set of finite primes $S$ is equal to the automorphic $L$-function $L_{S}(f,s)$; hence has $\ell$-adic Euler factors independent of the choice of the auxiliary prime $p$ and admits a meromorphic continuation to $\C$. The study of the values of $L_{S}(M(f),s)$ at $s\in\Z$ therefore falls under the scope of the Tamagawa Number Conjectures of \cite{BlochKato} on special values of $L$-functions of motives and, in fact, provided much of the historical motivation for their precise statements (compare for instance the rationality statements at critical points of \cite{ShimuraSpecial,DeligneFonctionsL,BeilinsonConjectureModular,BeilinsonConjecture} and the study of the exact value at critical points of \cite{KatoHodgeIwasawa,KatoEuler}). More generally, and more precisely, the so-called equivariant refinement of these conjectures given in \cite{KatoHodgeIwasawa,KatoViaBdR} predicts the equivariant special values of the $L$-function of the motive $M(f)\times_{\Q} L$ with coefficients in $\Ocal_{F}[G]$ where $L/\Q$ is a finite abelian extension with Galois group $G$.
\begin{Conj}\label{ConjTNCIntro}
Denote by $f^{*}$ the eigenform whose eigenvalues are the complex conjugates of those of $f$. For all $s\in\Z$ and all finite set $S$ containing the rational primes ramifying in $L$,  let $L^{G}_{S}(f^{*},s)$ be the element of $\C[G]$ such that $\chi(L_{S}^{G}(f^{*},s))=L_{S}(f^{*},\chi,s)$  for all character $\chi\in\widehat{G}$ extended by linearity to $\C[G]$. There exists a free one-dimensional $F[G]$-module $\Delta_{L/\Q,S}(M(f)(s))$ called the fundamental line and a motivic zeta element $\z_{L/\Q,S}(f)(s)\in\Delta_{L/\Q,S}(M(f)(s))$ satisfying the following properties.
\begin{enumerate}
\item For each complex embedding $\iota:F\plonge\C$, there exists a canonical isomorphism
\begin{equation}\nonumber
\per_{\iota,\C}:\Delta_{L/\Q,S}(M(f)(s))\tenseur_{F,\iota}\C\isocan\C[G].
\end{equation}
The image of $\z_{L/\Q,S}(f)\tenseur 1$ under $\per_{\iota,\C}$ is equal to $L_{S}^{G}(f^{*},s)$.
\item For each prime ideal $\pid\subset\Ocal_{F}$, there exists a canonical isomorphism
\begin{equation}\nonumber
\per_{\pid}:\Delta_{L/\Q,S}(M(f))\tenseur_{F}F_{\pid}\isocan\Det^{-1}_{F_{\pid}}\RGamma_{f}(G_{\Q,S},M(f)_{\et,\pid}\tenseur_{F_{\pid}}F_{\pid}[G])
\end{equation}
to the determinant of the \Nekovar-Selmer complex of $M(f)_{\et,\pid}\tenseur_{F_{\pid}}F_{\pid}[G]$. The equality 
\begin{equation}\nonumber
\per_{\pid}\left(\Ocal_{F_{\pid}}[G](\z_{L/\Q,S}(f)\tenseur 1)\right)=\Det^{-1}_{\Ocal_{F_{\pid}}[G]}\RGamma_{f}(G_{\Q,S},T(f))
\end{equation}
holds for any free $G_{\Q,S}$-stable $\Ocal_{F_{\pid}}[G]$-lattice $T(f)$ inside $M(f)_{\et,\pid}\tenseur_{F_{\pid}}F_{\pid}[G]$.
\end{enumerate}
\end{Conj}
We refer to \cite{BlochKato,KatoHodgeIwasawa,FontainePerrinRiou} for the conjectural definitions of $\Delta_{L/\Q,S}(M(f)(s))$, $\per_{\iota,\C}$ and $\per_{\pid}$ and to \cite{SelmerComplexes} (or subsection \ref{SubNeko} below) for the definition of the \Nekovar-Selmer complex.

The two statements of conjecture \ref{ConjTNCIntro} are commonly conjointly interpreted as predicting the special values of $L$-functions in terms of Galois cohomological data but reversing the perspective, as in \cite{KatoICM} or indeed as in the original work of Dirichlet on the class number formula or of Gauss in the final paragraph of the \textit{Disquisitiones}, they also provide a description of the Galois action on arithmetic invariants of $M(f)$ in terms of the special values of the $L$-function of its dual. In closer analogy with the study of $L$-functions of schemes of finite types over finite fields as in \cite[Exposé III]{DixExposes} and \cite{SGA41/2}, they can also be understood as a simultaneous description of Galois cohomology and special values of $L$-functions in terms of the single underlying element $\z_{L/\Q,S}(f)$, which then thus appears to be a global equivalent of the Frobenius morphism. This latter perspective has the additional benefit it makes clear that, as first noted in \cite{KatoHodgeIwasawa}, the system $\{\z_{L/\Q,S}(f)\}_{L,S}$ when $L$ spans finite abelian extensions of $\Q$ form an Euler system in the sense of \cite{KolyvaginEuler}; a fact whose generalization is of crucial importance in this manuscript.

Let $\Q_{\infty}$ be the $\zp$-extension of $\Q$, $\Q_{n}$ its only sub-field of degree $p^{n}$ and $\Lambda$ the completed group-algebra $\zp[[\Gal(\Q_{\infty}/\Q)]]$ of its Galois group. Putting together the collection of the Equivariant Tamagawa Number Conjectures at $p$ for the extensions $\Q_{n}/\Q$  (henceforth ETNC for $\Q_{n}/\Q$) yields a conjecture with coefficients in $\Lambda$ which we refer to as the Equivariant Tamawaga Number Conjecture for $M(f)$ with coefficients in $\Lambda$ (henceforth ETNC with coefficients in $\Lambda$). After the construction of zeta elements for modular forms in \cite{KatoEuler} and the awe-inspiring proof that they satisfy the first of the two fundamental properties of conjecture \ref{ConjTNCIntro} at critical points and, so to speak, half of the second, it seems to have been known to experts (though never published to the best of my knowledge) that the ETNC for $\Q_{n}/\Q$ for large enough $n$ at any $s$ was a consequence of the ETNC with coefficients in $\Lambda$ (a comparable statement restricted to critical value is in \cite[Section 13]{KatoEuler}). When $f$ is $p$-ordinary, that is to say when $a_{p}(f)$ is a $p$-adic unit, and under a few other  technical hypotheses \cite[Theorem 3.29]{SkinnerUrban} establishes a divisibility in the ETNC with coefficients in $\Lambda$. Together with \cite[Theorem 12.5]{KatoEuler}, this proves the ETNC with coefficients in $\Lambda$ for $M(f)$ in this setting and hence the ETNC for $\Q_{n}/\Q$ with $n$ large. The main outstanding problem thus remains the case of the special value at the central critical point when the $L$-function vanishes with high order.

\paragraph{Equivariant conjectures with coefficients in the Hecke algebra}These achievements, though spectacular, are far from being the end of the study of special values of $L$-functions of modular motives. Indeed, the motive $M(f)$ is constructed as a quotient of the Chow motive of a modular curve with weight $k$ and hence admits an action of the Hecke algebra, so that one could envision an Equivariant Tamagawa Number Conjecture with coefficients in the Hecke algebra (henceforth ETNC with coefficients in the Hecke algebra). This conjecture would be much stronger than the ETNC with coefficients in $\Lambda$ as it would encode not only the special values of the $L$-function of a single eigenform but also congruences between special values of congruent eigenforms (that such congruences could or should be true was already discussed in \cite[I.1.A]{MazurValues}). In fact, as Hecke algebras were conjectured in \cite{MazurTilouine} to be universal deformation rings in the sense of \cite{MazurDeformation} and as \cite{WilesFermat,TaylorWiles} and much subsequent works established this conjecture in many cases, the ETNC with coefficients in the Hecke algebra should be the most general possible Tamagawa Number Conjecture with commutative coefficients. In analogy with the ETNC for $L/\Q$ above, a tentative statement of the ETNC with coefficients in the Hecke algebra would be as follows.
\begin{Conj}[Tentative statement]\label{ConjTentative}
Let $M$ be the motive with weight $k$ of the modular curve $X(N)$. Let $\Hecke$ be a local quotient of the $p$-adic Hecke algebra acting on $M$ and let $Q(\Hecke)$ be its total ring of fraction. For all $s\in\Z$, there exists a free rank one $Q(\Hecke)$-module $\Delta(s)$ and a $p$-adic universal zeta element $\z(s)$ satisfying the following properties.
\begin{enumerate}
\item For all eigenform $f$ under the action of $\Hecke$ with coefficients in $F/\Q$ and all primes $\pid|p$, there exists a canonical isomorphism
\begin{equation}\nonumber
\Delta(s)\tenseur_{Q(\Hecke)}F_{\pid}\isocan_{f,\pid}\Delta(M(f)(s))\tenseur_{F}F_{\pid}
\end{equation}
sending $\z(s)$ to $\z(f)(s)$.
\item There exists a canonical isomorphism of $Q(\Hecke)$-modules
\begin{equation}\nonumber
\per_{p}:\Delta(s)\isocan\Det^{-1}_{Q(\Hecke)}\RGamma_{f}(G_{\Q,S},M_{\et,p}(s)).
\end{equation}
The equality 
\begin{equation}\nonumber
\per_{p}\left(\Hecke\cdot\z(s)\right)=\Det^{-1}_{\Hecke}\RGamma_{f}(G_{\Q,S},T)
\end{equation}
holds for any free $G_{\Q,S}$-stable $\Hecke$-lattice $T\subset M_{\et,p}(s)$.
\end{enumerate}
\end{Conj}
Unfortunately, under the present guise, these tentative statements do not form a conjecture at all. First, they do not specify which Hecke algebra exactly we are considering: does it contain Hecke operators $T(\ell)$ at $\ell|N$? Does it act faithfully on modular forms new of level $N$? Second, they do not specify how the canonical isomorphism $\isocan_{f,\lambda}$ and $\per_{p}$ are constructed. Third, it is not known whether the complex $\RGamma_{f}(G_{\Q,S},T)$ is a perfect complex of $\Hecke$-modules; the difficulty lying in proving that $\Gamma(I_{\ell},-)$ sends perfect complexes of $\Hecke$-modules to perfect complexes of $\Hecke$-modules for $\ell\neq p$. Consequently, statement 2 of conjecture \ref{ConjTentative} as it stands is in fact woefully undefined.
\subsection{Main results}
The aim of this manuscript is to give a precise formulation of the ETNC with coefficients in the Hecke algebra for modular motives and to prove a large part of it when the hypotheses of the method of Taylor-Wiles systems are satisfied. Let $f\in S_{k}(\Gamma_{1}(N))$ be an eigencuspform of weight $k\geq 2$ and let $p$ be an odd prime. Let $\Hecke_{\mgot}$ be the local factor of the $p$-adic reduced Hecke algebra attached to $f$ and let $\Fp$ be the residue field of $\Hecke_{\mgot}$. Denote by
\begin{equation}\nonumber
\rhobar_{f}:G_{\Q}\fleche\GL_{2}(\bar{\Fp})
\end{equation}
the residual $G_{\Q}$-representation attached to $f$ and by $N(\rhobar_{f})$ its Artin conductor outside $p$. Let $\Sigma$ be a finite set of finite places containing $\{\ell|N(\rhobar_{f})p\}$. The following is our main theorem (see theorem \ref{TheoMainPreuve} for a precise statement).
\begin{TheoEnglish}\label{TheoMain}
 Assume that $\rhobar_{f}$ satisfies the following hypotheses.
\begin{enumerate}
\item Let $p^{*}$ be $(-1)^{(p-1)/2}p$. The representation $\rhobar_{{f}}|_{G_{\Q(\sqrt{p^{*}})}}$ is absolutely irreducible.
\item Either the representation $\rhobar_{{f}}|_{G_{\qp}}$ is reducible but not scalar (in which case we say that $\rhobar_{f}$ is nearly ordinary) or there exists a commutative finite flat $p$-torsion group scheme $G$ over $\zp$ and a character $\bar{\mu}$ such that $\rhobar_{{f}}\tenseur\bar{\mu}^{-1}$ is isomorphic as $\bar{\Fp}[G_{\qp}]$-module to $(G\times_{\zp}\Qbar_{p})[p]$ (in which case we say that $\rhobar_{f}$ is flat).
\item\label{HypMonodromy} There exists $\ell\in\Sigma$ such that $\ell||N(\rhobar_{f})$ and such that the image of $\rhobar_{f}|_{G_{\Q_{\ell}}}$ contains a non-identity unipotent element.
\end{enumerate}
Let $R_{\Sigma,\Iw}$ be $\Hecke_{\mgot}\hat{\tenseur}\Lambda$ and let $Q(R_{\Sigma,\Iw})$ be its total ring of fractions. Let $T_{\Sigma,\Iw}$ be the $G_{\Q,\Sigma}$-representation with coefficients in $R_{\Sigma,\Iw}$ deforming $\rhobar_{f}$. Then there exists a fundamental line $\Delta_{\Sigma,\Iw}$ with coefficients in $R_{\Sigma,\Iw}$ and a universal zeta element $\z_{\Sigma,\Iw}$ which is a basis of $\Delta_{\Sigma,\Iw}$ satisfying the following properties.
\begin{enumerate}
\item For all integer $1\leq s\leq k-1$, all eigencuspform $g\in S_{k}(\Gamma_{1}(N))$ congruent to $f$ and all character $\chi$ of $\Gal(\Q_{\infty}/\Q)$ of large enough finite order, there exists a specified morphism $\per_{g,\chi,s}$ sending $\z_{\Sigma,\Iw}$ to $L_{\{p\}}(g^{*},\chi,s)$ (here, as above, $g^{*}$ denotes the eigencuspform whose eigenvalues are the complex conjugates of those of $g$).
\item There exists a specified isomorphism
\begin{equation}\nonumber
\Delta_{\Sigma,\Iw}\tenseur_{R_{\Sigma,\Iw}}Q(R_{\Sigma,\Iw})\isocan\Det^{-1}_{Q(R_{\Sigma,\Iw})}\left(\RGamma_{c}(\Spec\Z[1/\Sigma],T_{\Sigma,\Iw})\Ltenseur_{R_{\Sigma,\Iw}}Q(R_{\Sigma,\Iw})\right)
\end{equation}
such that the image of $\Delta_{\Sigma,\Iw}$ contains $\Det^{-1}_{R_{\Sigma,\Iw}}\RGamma_{c}(\Spec\Z[1/\Sigma],T_{\Sigma,\Iw})$.
\end{enumerate}
\end{TheoEnglish}
By way of exegesis, we note that the first property for $g=f$ identifies the image of $\z_{\Sigma,\Iw}$ with the zeta element of the ETNC with coefficients in $\Lambda$ for $f$, and hence that $\z_{\Sigma,\Iw}$ is an interpolation with coefficients in the Hecke algebra of the zeta elements of the ETNC with coefficients in $\Lambda$ for modular forms congruent to $f$. The ETNC with coefficients in $R_{\Sigma,\Iw}$ would then predict an equality between $\Det^{-1}_{R_{\Sigma,\Iw}}\RGamma_{c}(\Spec\Z[1/\Sigma],T_{\Sigma,\Iw})$ and the image of $\Delta_{\Sigma,\Iw}$, whereas we only state and prove an inclusion. Hence, theorem \ref{TheoMain} is a weak form of the ETNC with coefficients in the Hecke algebra. Nevertheless, this weak form is enough to entail a number of interesting results, and in particular that the full ETNC with coefficients in the Hecke algebra is often true (see theorem \ref{TheoSUPreuve} for a precise statement).
\begin{TheoEnglish}\label{TheoSU}
Let $p$ be an odd prime and $N$ such that $p\nmid N$. Let $f\in S_{k}(\Gamma_{1}(p^{r})\cap\Gamma_{0}(N))$ be an eigencuspform. Assume that $\rhobar_{f}$ satisfies the following hypotheses.
\begin{enumerate}
\item Let $p^{*}$ be $(-1)^{(p-1)/2}p$. The representation $\rhobar_{{f}}|G_{\Q(\sqrt{p^{*}})}$ is absolutely irreducible.
\item The semi-simplification of $\rhobar_{{f}}|_{G_{\qp}}$ is reducible but not scalar.
\item There exists $\ell\in\Sigma$ such that $\ell||N(\rhobar_{f})$ and such that the image of $\rhobar_{f}|_{G_{\Q_{\ell}}}$ contains a non-identity unipotent element.
\end{enumerate}
Then the ETNC with coefficients in $R_{\Sigma,\Iw}$ for the motive $M(f)$ is true at $p$.
\end{TheoEnglish}
Theorems \ref{TheoMain} and \ref{TheoSU} seem to be among the first general results on the ETNC with coefficients in Hecke algebras; if only for the somewhat tautological reason that no prior unconditional formulation of this conjecture seems to exist in the literature.\footnote{\cite[Conjecture 3.2.2]{KatoViaBdR} takes as input a smooth sheaf over $\Spec\Z[1/p]$, \cite[Conjecture 2.2]{GreenbergIwasawaMotives} requires the coefficient ring to be integrally closed.} 

We also record here some technical consequences of theorem \ref{TheoMain} which improve on the existing literature. Let $L/\qp$ be a finite extension containing all the eigenvalues of $f$, let $\Ocal$ be its ring of integers and let $\Lambda$ be $\Ocal[[\Gal(\Q_{\infty}/\Q)]]$. Let $V(f)$ be the two-dimensional $G_{\Q}$-representation attached to $f$, let $T(f)$ be a $G_{\Q}$-stable $\Ocal$-lattice in $V(f)$ an let $T(f)_{\Iw}$ be the $G_{\Q}$-representation $T(f)\tenseur\Lambda$ with action on both sides of the tensor product. Let $\z(f)$ be the zeta element of the ETNC with coefficients in $\Lambda$ for $f$. Then $\z(f)$ can be regarded as an element of $\Hun_{\et}(\Spec\Z[1/p],T(f)_{\Iw})$ and the ETNC with coefficients in $\Lambda$ for $f$ is equivalent to the equality
\begin{equation}\nonumber
\carac_{\Lambda}H^{2}_{\et}(\Spec\Z[1/p],T(f)_{\Iw})=\carac_{\Lambda}H^{1}_{\et}(\Spec\Z[1/p],T(f)_{\Iw})/\z(f).
\end{equation}
See conjecture \ref{ConjXcaliLambda} for a details.
\begin{CorEnglish}\label{CorMain}
Assume the hypotheses and notations of theorem \ref{TheoMain}. Then
\begin{equation}\label{EqDivisibility} 
\carac_{\Lambda}H^{2}_{\et}(\Spec\Z[1/p],T(f)_{\Iw})|\carac_{\Lambda}H^{1}_{\et}(\Spec\Z[1/p],T(f)_{\Iw})/\z(f).
\end{equation}
\end{CorEnglish}
In \cite[Theorem 12.5]{KatoEuler}, this divisibility is proved only possibly up to a local error term at $p$ which vanishes if $\rho_{f}|_{G_{\qp}}$ is potentially crystalline. While the difference might seem technical and unimportant, the ideas behind the proof of corollary \ref{CorMain} are actually among the most  sophisticated of the manuscript and play a crucial role in the general argument.

Theorem \ref{TheoMain} also allows us to refine known results on the compatibility between the ETNC with coefficients in the Hecke algebra and the ETNC with coefficients in $\Lambda$ (see corollary \ref{CorEPWPreuve} for a precise statement).
\begin{CorEnglish}\label{CorEPW}
Assume the hypotheses and notations of corollary \ref{CorMain}. Then the four following assertions are equivalent.
\begin{enumerate}
\item The ETNC with coefficients in $\Lambda$ for $f$ is true.
\item\label{AssertionSingle} There exists an eigencuspform $g$ of weight $k$ congruent to $f$ modulo $p$ for which the ETNC with coefficients in $\Lambda$ is true.
\item\label{AssertionAll} For all eigencuspform $g$ of weight $k$ congruent to $f$ modulo $p$, the ETNC with coefficients in $\Lambda$ is true.
\item The ETNC with coefficients in the Hecke algebra for $f$ is true.
\end{enumerate}
If moreover $\rhobar_{f}|_{G_{\qp}}$ is reducible, then the condition that $g$ has the same weight as $f$ can be removed in assertions \ref{AssertionSingle} and \ref{AssertionAll}.
\end{CorEnglish}
Results of this type were proved in \cite{EmertonPollackWeston} under the hypotheses that $\rhobar_{f}|_{G_{\qp}}$ is reducible and that the $\mu$-invariant $\mu(f)$ of $f$ is trivial. In \cite{OchiaiMainConjecture}, they were proved under the hypotheses that $\rhobar_{f}|_{G_{\qp}}$ is reducible, that $f$ belongs to $S_{k}(\Gamma_{1}(p^{r}))$ and that the ordinary Hida-Hecke algebra attached to $f$ is a regular local ring. The hypotheses on the triviality of $\mu$ and the regularity of the Hida-Hecke algebra are believed to always hold, but very few non-tautological criteria exist to establish their veracity as far as this author knows.
\subsection{Outline of the proofs}
\paragraph{The Weight-Monodromy conjecture and special values of $L$-function}Our first task is to formulate an unconditional conjecture that would coincide with the usual ETNC when the latter is well-defined. This we achieve through the following crucial observation: the severe constraints conjecturally put on the action of the inertia group on the $p$-adic étale realization of a motive by the Weight-Monodromy conjecture (henceforth WMC) allow to refine the definition of the local complexes involved in the statement of the ETNC. This process yields objects we call \emph{refined fundamental lines} which are not in general determinants of perfect complexes but rather canonical trivializations of invertible graded modules which themselves are the determinants of the sought for perfect complexes when these are known to exist. When the motive is of automorphic origin, the description of the WMC is supplemented by automorphic data coming from the Local Langlands Correspondence and our construction are in this way shown to be compatible with the action of the Hecke algebra. Indeed, the very definition of the refined fundamental line for an automorphic motive singles out a specific local factor of the Hecke algebra which coincides with the universal deformation ring subject to natural conditions. A conceptually satisfying property of the refined fundamental lines is that they are almost by construction shown to be compatible with change of rings of coefficients at motivic points; a property which generalizes the control theorem of \cite{MazurRational} (and much subsequent work) in a probably optimal way. That it is compatible with change of levels in the automorphic sense is a much deeper result which in the case of modular curves amounts to the compatibility of the refined fundamental line with specialization and a variant of Ihara's lemma.

We are then finally in position to formulate our version of the ETNC with coefficients in Hecke algebras for modular motives: see conjectures \ref{ConjXcaliRaid} and \ref{ConjXcaliHeckeSigma} for precise statements. These conjectures are equivalent to the usual trivializations of the determinants of étale cohomology with compact support when all necessary objects are known to be defined and equivalent to the usual equality of characteristic ideals when specialized to $\Lambda$  or when the Hecke algebra is known to be regular. A crucial fact is that different choices of Hecke algebras, more specifically reduced Hecke algebras and irreducible components thereof, yield different refined fundamental lines and hence different mutually compatible conjectures. This reflects the fact that the ETNC should be sensitive to changes of the action of inertia through specialization; an observation that has been conceptually understood from a conjectural point of view at least since the study in \cite{FontaineValeursSpeciales,KatoHodgeIwasawa} of partial $L$-functions and is also at the heart of \cite[Section 3.5]{EmertonPollackWeston}.

\paragraph{Euler systems and Taylor-Wiles systems}
Our proof of part of our conjectures under the hypotheses of theorem \ref{TheoMain} is then by an amplification of the method of Euler/Kolyvagin systems, where two actually quite distinct ideas are subsumed under this name. The first one, due to V.Kolyvagin in \cite{KolyvaginEuler}, is the observation that Galois cohomology classes satisfying compatibility relations in towers of extensions reminiscent of the properties of partial Euler products yield systems of classes with coefficients in principal artinian rings whose local behaviors is sufficiently constrained to establish a crude bound on the order of some Galois cohomology or Selmer groups. The second idea is a descent principle due to K.Rubin which allows under suitable assumptions to translate a collection of crude bounds for many specializations with coefficients in artinian rings in a sharp bound in the limit, that is for objects with coefficients in Iwasawa algebras. When the ring of coefficients of the limit object is not known to be normal, as is the case with Hecke algebra, this descent principle meets quite formidable challenge, as it is of course entirely possible for an invertible module to be non-integral while all its specializations to discrete valuations rings are integral in which case, no contradiction can arise by descent. For this reason, most account of the Euler/Kolyvagin systems method (\cite{PerrinRiouEuler,RubinEuler,KatoEuler,MazurRubin,HowardKolyvagin,HowardGLdeux,OchiaiEuler,FouquetRIMS} for instance) assume that the ring of coefficients is regular, or at least normal, and those which don't (\cite{KatoEulerOriginal,FouquetDihedral} for instance) typically prove weaker statement at the locus of non-normality of the coefficient ring.

Our second main novel contribution allows us to bypass this difficulty by first resolving the singularities of the Hecke algebra using the method of Taylor-Wiles of \cite{WilesFermat,TaylorWiles} systems as axiomatized in \cite{DiamondHecke,FujiwaraDeformation} before applying the descent procedure. Under the two first hypotheses of theorem \ref{TheoMain}, there exists a Taylor-Wiles system $\Delta_{Q}$ of refined fundamental lines yielding a limit object $\Delta_{\infty}$ over a regular local ring $R_{\infty}$. If the limit object $\Delta_{\infty}$ is not integral, then it has non-integral specializations to discrete valuation rings. Even though $\Delta_{\infty}$ itself has no Galois interpretation, its specializations do, so that this non-integrality contradicts Kolyvagin's bound (or more accurately the sharper results of \cite{KatoEuler}). Hence $\Delta_{\infty}$ is integral. Then so are the $\Delta_{Q}$ and in particular the fundamental line $\Delta$ we started with. This argument is by nature extremely sensitive to the existence of any error term at any step and thus relies in an essential way on the exact control property of the refined fundamental lines.

We make the following observation, which lies at the conceptual core of this manuscript: just as the conjectured compatibility of the Tamagawa Number Conjecture with the $\Gal(\Q(\zeta_{Np^{s}})/\Q)$-action coming from the covering $\Spec\Z[\zeta_{Np^{s}},1/p]\fleche\Spec\Z[1/p]$ implies that motivic zeta elements form an Euler system, the conjectured compatibility of the Tamagawa Number Conjecture with the action of the Hecke algebra coming from the covering $X_{U'}\fleche X_{U}$ of Shimura varieties implies that the refined fundamental lines form a Taylor-Wiles system. In both cases, the compatibilities we hope for the conjectures on special values of $L$-functions therefore suggest powerful tools to prove the conjectures.
\paragraph{The nearly ordinary case}
The argument outlined above establishes theorem \ref{TheoMain} and corollary \ref{CorEPW} under the hypothesis that $\rhobar_{f}|_{G_{\qp}}$ is irreducible. When $\rhobar_{f}$ is nearly ordinary, even the sharper result of \cite{KatoEuler} for the ETNC with coefficients in $\Lambda$ may contain a slight error term linked to trivial zeroes which is enough to prevent us from reaching the desired contradiction at the very end of the argument. Hence, we are forced to repeat the argument over the $p$-adic families of nearly ordinary modular forms parametrized by $\Lambda_{\Hi}\simeq\zp[[X]]$ constructed by H.Hida in \cite{HidaInventionesOrdinary,HidaAmericanCongruence,HidaNearlyOrdinaryRepresentations} (see also \cite{WilesOrdinaryLambdaAdic}). Fortunately, the definitions of our refined fundamental lines carries over to that setting and versions of the Taylor-Wiles systems machinery over $\Lambda_{\Hi}$ exist. This allows us to reduce the proof of theorem \ref{TheoMain} to the case $k>2$. This finishes the proof of theorem \ref{TheoMain} and of corollaries \ref{CorMain} and \ref{CorEPW}.
\paragraph{The ETNC with coefficients in Hecke algebras}
Under the hypotheses of theorem \ref{TheoSU}, the main results of \cite{KatoEuler} establish an inclusion in the ETNC with coefficients in $\Lambda$ for $f$ and the main results of \cite{SkinnerUrban} establish the reverse inclusion. Combined, they thus imply that the ETNC with coefficients in $\Lambda$ is true for $f$. In general, the truth of the ETNC with coefficients with $\Lambda$ is very far to imply formally the truth of the ETNC with coefficients in the universal deformation ring $R_{\Sigma,\Iw}$ but granted the full force of theorem \ref{TheoMain}, it is enough to prove theorem \ref{TheoSU} to exhibit a single modular specialization of $R_{\Sigma,\Iw}$ for which the ETNC with coefficients in $\Lambda$ is true, and so the combined results of  \cite{KatoEuler,SkinnerUrban} allow us to conclude. 

\paragraph{Discussion of the hypotheses}Here follows a brief discussion of the hypotheses of theorem \ref{TheoSU} and of their relevance. The first two numbered hypotheses are the familiar hypotheses of the Taylor-Wiles method so are used in a crucial way in the proof of theorem \ref{TheoMain}. They could probably be dispensed with at the price of inverting $p$ by an appeal to the generalization of the method of Taylor-Wiles systems introduced in \cite{KisinFlat}, the main difficulties being to show that there exists a sheaf of zeta elements on the eigencurve of \cite{ColemanMazur}. The hypotheses on $f$ and $N$ come from \cite[Corollary 3.28]{SkinnerUrban}. The last numbered hypothesis is thus presumably the most mysterious. In fact, it comes both from \cite{SkinnerUrban}, where it is assumed in order to quote results of \cite{Vatsal} on the vanishing of the anticyclotomic $\mu$-invariant, and from \cite{KatoEuler}, as a classical group-theoretic argument in the method of Euler systems proves under this hypothesis an expected bound on the cyclotomic $\mu$-argument. Hence, a single hypothesis impacts both the cyclotomic and anticylotomic $\mu$-invariants of modular forms, though through seemingly completely two different ways, and furthermore this hypothesis amounts to requiring that $\pi(f)$ admits a $p$-adically interpolatable Jacquet-Langlands switch to an indefinite quaternionic automorphic representation. This could be a coincidence or reflect a possible, but at present mysterious, unified automorphic treatment of $\mu$-invariants in the presence of an auxiliary prime with residually maximal monodromy. It might be inferred from this discussion that omitting hypothesis \ref{HypMonodromy} in theorem \ref{TheoMain} and \ref{TheoSU} would yield  similar theorems (an inclusion for theorem \ref{TheoMain} and an equality for theorem \ref{TheoSU}) outside of the prime $p$, but this is not obviously true as far as this author can see: the possibility of an error term lurking somewhere, even if it is circumscribed to the single prime $p$, might irremediably damage the descent argument.

\section{Notations}
\paragraph{General notations}Rings are assumed to be commutative. For a field $\Fp$, the category of complete local noetherian rings with residue field equal to $\Fp$ (with morphisms inducing identity on $\Fp$) is denoted by $\Ccal(\Fp)$. A representation $(T,\rho,R)$ of a topological group $G$ is a continuous morphism 
\begin{equation}\nonumber
\rho:G\fleche\Aut_{R}(T)
\end{equation}
from $G$ to the automorphisms of a free $R$-module $T$. If $K$ is a field, we write $G_{K}$ for the Galois group of a separable closure of $K$. If $K$ is a number field with ring of integers $\Ocal_{K}$ and if $S$ is a finite set of rational primes, we denote by $G_{K,S}$ the Galois group of the maximal extension of $K$ unramified outside primes of $\Ocal_{K}$ above primes in $S$. For all rational primes $\ell$, we fix an algebraic closure $\Qbar_{\ell}$ of $\Q_{\ell}$, an embedding of $\Qbar$ into $\Qbar_{\ell}$ and an identification $\iota_{\infty,\ell}:\C\simeq\Qbar_{\ell}$ extending $\Qbar\plonge\Qbar_{\ell}$. The Galois group of the unique $\zp$-extension $\Q_{\infty}/\Q$ is denoted by $\Gamma$.
\subsection{Modular curves and their cohomology}
\subsubsection{Modular curves}
Let $\G$ be the reductive group $\GL_{2}$ over $\Q$, $X$ be $\C-\R$ and $\Sh(\G,X)$ be the tower of Shimura curves attached to the Shimura datum $(\G,X)$. We consider the following compact open subgroups of $\G(\A_{\Q}^{(\infty)})$.
\begin{align}\nonumber
&U(N)=\produit{\ell}{}U(N)_{\ell}=\produit{\ell}{}\left\{g\in\gldeux(\Z_{\ell})|g\equiv\matrice{1}{0}{0}{1}\modulo \ell^{v_{\ell}(N)}\right\}\\\nonumber
&U_{1}(N)=\produit{\ell}{}U_{1}(N)_{\ell}=\produit{\ell}{}\left\{g\in\gldeux(\Z_{\ell})|g\equiv\matrice{*}{*}{0}{1}\modulo \ell^{v_{\ell}(N)}\right\}\\\nonumber
&U_{0}(N)=\produit{\ell}{}U_{0}(N)_{\ell}=\produit{\ell}{}\left\{g\in\gldeux(\Z_{\ell})|g\equiv\matrice{*}{*}{0}{*}\modulo \ell^{v_{\ell}(N)}\right\}\\\nonumber
&U(M,N)=\produit{\ell}{}U(M,N)_{\ell}=\produit{\ell}{}\left\{g\in\gldeux(\Z_{\ell})|g\equiv\matrice{1}{0}{*}{*}\modulo\ell^{v_{\ell}(M)},g\equiv\matrice{*}{*}{0}{1}\modulo\ell^{v_{\ell}(N)}\right\}.
\end{align}
The curve $Y(U)=\Sh_{U}(\G,X)$ and its compactification along cusps $j:Y(U)\plonge X(U)$ are regular schemes over $\Z$ which are smooth over $\Z_{\ell}$ if $U_{\ell}$ is maximal and $U$ is sufficiently small; \textit{e.g} $U=U(N)$ and $N\geq3$ (see \cite[p. 305]{KatzMazur}). The set of complex points of $Y(U)$ is given by the double quotient
\begin{equation}\nonumber
Y(U)(\C)\simeq\G(\Q)\backslash\left(\C-\R\times\G(\A_{\Q}^{(\infty)})/U\right)
\end{equation}
and is an algebraic variety if $U$ is sufficiently small. For $U=U_{?}(*)$ with $?=\vide,0$ or $1$ and $*=N$ or $N,M$, we write $Y_{?}(*)$ for $Y(U)$ and $X_{?}(*)$ for $X(U)$. 
\subsubsection{Hecke correspondences}Let $g$ be an element of $\G(\A_{\Q}^{(\infty)})$. Right multiplication by $g$ induces a finite flat $\Q$-morphism
\applicationsimple{[\cdot g]}{X(U\cap gUg^{-1})}{X(U\cap g^{-1}Ug)}
which defines the Hecke correspondence $T(g)=[UgU]$ on $X(U)$.
\begin{equation}\label{DiagDefHecke}
\xymatrix{
X(U\cap gUg^{-1})\ar[r]^{[\cdot g]}\ar[d]&X(U\cap g^{-1}Ug)\ar[d]\\
X(U)\ar@{-->}[r]^{[UgU]}&X(U)
}
\end{equation}
For $\ell$ a prime number and $a\in\idele{\Q}$ a finite idèle, we denote by $T(\ell)$ the Hecke correspondence $[U\matrice{1}{0}{0}{\ell}U]$ and by $\diamant{a}$ the diamond correspondence $[U\matrice{a}{0}{0}{a}U]$. The full classical Hecke algebra $\hgot(U)$ of level $U$ is the $\Z$-algebra generated by Hecke and diamond correspondences acting on $X(U)$.

\subsubsection{Cohomology}
\paragraph{Betti and étale cohomology}Let $\pi:E\fleche Y(N)$ be the universal elliptic curve over $Y(N)$ and let $\bar{\pi}:\bar{E}\fleche X(N)$ be the universal generalized elliptic curve over $X(N)$. For $k\geq2$ an integer, let $\Hcal_{k-2}$ be the local system $\Sym^{k-2}R^{1}\pi_{*}\Z$ on $Y(N)(\C)$ and let $\Fcal_{k-2}$ be $j_{*}\Hcal_{k-2}$. If $N\geq3$, let $\RGamma_{B}(X(N)(\C),\Fcal_{k-2})$ be the singular cohomology complex of the complex points of $X(N)$. If $X$ is a quotient curve $G\backslash X(N)$ with $N\geq3$ under the action of a finite group $G$, and if $A$ is a ring in which $|G|$ is invertible, we denote by $H^{i}(X(\C),\Fcal_{k-2}\tenseur_{\Z}A)$ the cohomology group $H^{i}(X(N)(\C),\Fcal_{k-2}\tenseur_{\Z}A)^{G}$ and note that it is also the cohomology of the complex $\RGamma_{B}(X(\C),\Fcal_{k-2}\tenseur_{\Z}A)$ where $X$ is seen as a Deligne-Mumford stack over $A$ (in particular $H^{i}(X(\C),\Fcal_{k-2}\tenseur_{\Z}A)$ is independent of the choice of $N$ and $G$). We denote $H^{i}(X(\C),\Fcal_{k-2})\tenseur_{\Z}\zp$ by $H^{i}_{\et}(X\times_{\Q}\Qbar,\Fcal_{k-2}\tenseur_{\Z}\zp)$ and by $\RGamma_{\et}(X\times_{\Q}\Qbar,\Fcal_{k-2}\tenseur_{\Z}\zp)$ the corresponding cohomology complex. As usual, we denote by 
\begin{equation}\nonumber
\Mcal_{k}(U(N))=H^{0}(X(N),\pi_{*}(\Omega^{1}_{\bar{E}/X(N)})^{\tenseur k})
\end{equation}
the space of holomorphic modular forms of weight $k$ and by
\begin{equation}\nonumber
S_{k}(U(N))=H^{0}(X(N),\pi_{*}(\Omega^{1}_{\bar{E}/X(N)})^{\tenseur (k-2)}\tenseur_{\Ocal(X(N))}\Omega^{1}_{X(N)/\Q})
\end{equation}
the space of holomorphic cusp forms.

\paragraph{Hecke action}The Hecke algebra acts contravariantly on cohomological realizations of $X(U)$. In particular, as the Hodge decomposition realizes the $\C$-vector space of complex cusp forms $S_{k}(U)$ as a direct summand of $\Hun(X(U)(\C),\Fcal_{k-2}\tenseur_{\Z}\C)$, the complex Hecke algebra $\hgot(U)\tenseur_{\Z}\C$ acts on $S_{k}(U)$. The $\Z$-submodule $S_{k}(U,\Z)\subset S_{k}(U)$ of cusp forms with integral $q$-expansion is stable under the action of $\hgot(U)$ thereby induced. This defines an action of $\hgot(U)\tenseur_{\Z}A$ on $S_{k}(U,\Z)\tenseur_{\Z}A$ for all ring $A$. The complex $\RGamma_{B}(X(U)(\C),\Fcal_{k-2})$ admits a representation as a bounded below (but not necessarily bounded above) complex of projective $\hgot(U)$-modules. 

An eigenform $f\in S_{k}(U)$ is an eigenvector under the action of all $T(\ell)$. The conductor $\cid(\pi(f))$ of an eigenform is the conductor of the automorphic representation $\pi(f)$ attached to $f$ (see \cite[Theorem 1]{CasselmanAtkin} for the definition of $\cid(\pi(f))$). Two eigenforms are equivalent in the sense of Atkin-Lehner if they are eigenvectors for the same eigenvalues for all $T(\ell)$ except possibly finitely many. A newform $f\in S_{k}(U)$ is an eigenform such that for all $g\in S_{k}(U')$ equivalent to $f$ in the sense of Atkin-Lehner, $\cid(\pi(f))$ divides $\cid(\pi(g))$.

Let $p$ be an odd prime. We call $\hgot(U)\tenseur_{\Z}\zp$ the $p$-adic classical Hecke algebra and denote it by $\Hecke_{\cl}(U)$. It is a semi-local ring finite and free as $\zp$-module. To an eigenform $f$ is attached a map $\lambda_{f}$ from $\Hecke_{\cl}(U)$ to $\Qbar_{p}$ such that $T(\ell)f=\lambda_{f}(T(\ell))f$ and conversely we say that a map $\lambda$ from a quotient of sub-algebra of $\Hecke_{\cl}(U)$ to a discrete valuation ring in $\Qbar_{p}$ is modular if there exists an eigenform $f$ such that $\lambda=\lambda_{f}$. Let the reduced Hecke algebra $\Hecke^{\red}(U)\subset\Hecke_{\cl}(U)$ be the sub $\zp$-algebra generated by the diamond operators and the Hecke operators $T(\ell)$ for $\ell$ such that $U_{\ell}$ is a maximal compact open subgroup. Let the new Hecke algebra $\Hecke^{\new}(U)$ be the quotient of $\Hecke_{\cl}(U)$ acting faithfully on the space of newforms of level $U$. Both $\Hecke^{\red}(U)$ and $\Hecke^{\new}(U)$ are finite flat reduced semi-local $\zp$-algebras.

\subsection{Galois representations}\label{SubGalois}
\subsubsection{Residual and rational representations}
Let $\Hecke$ be either $\Hecke^{\red}(U)$ or $\Hecke^{\new}(U)$ and let $f\in S_{k}(U)$ be an eigenform which is a newform in case $\Hecke=\Hecke^{\new}(U)$. There exists a finite extension $F_{\pid}$ of $\qp$ whose ring of integers we denote by $\Ocal$ containing the image of $\lambda_{f}$ and a maximal ideal $\mgot_{f}$ of $\Hecke$ such that $\lambda_{f}$ factors through $\Hecke_{\mgot_{f}}$. Let $\bar{\Fp}$ be the algebraic closure of the residue field of $\Hecke_{\mgot_{f}}$.

Denote by $S$ the set of finite primes $\ell$ such that $U_{\ell}$ is not a maximal compact open subgroup. Let $M_{\mgot_{f}}$ be the étale cohomology group $H^{1}_{\et}(X(U)\times_{\Q}\Qbar,\Fcal_{k-2}\tenseur_{\Z}\Qbar_{p})_{\mgot_{f}}$. The $G_{\Q,S}$-representation $(M_{\mgot_{f}},\rho_{\mgot_{f}},\Qbar_{p})$ is the unique semi-simple representation satisfying 
\begin{equation}\label{EqEichlerShimura}
\begin{cases}
\tr\rho_{\mgot_{f}}(\Fr(\ell))=T(\ell)\\
\det\rho_{\mgot_{f}}(\Fr(\ell))=\ell\diamant{\ell}
\end{cases}
\end{equation}
for all $\ell\notin S$. In \eqref{EqEichlerShimura}, Hecke operators are regarded as elements of $\Qbar_{p}$ through the injection of $\Hecke\tenseur\Qbar_{p}$ into a product of fields. The $G_{\Q}$-representation attached to $f$ is the quotient $(M(f),\rho_{f},\Qbar_{p})$ of $M_{\mgot_{f}}$ such that $\tr(\rho_{f})=\lambda_{f}$. The map $\tr(\rho_{\mgot_{f}}):G_{\Q,S}\fleche\Hecke_{\mgot_{f}}$
is a pseudocharacter of dimension 2 in the sense of \cite{WilesOrdinaryLambdaAdic,TaylorHilbert,BellaicheChenevier}. We denote by $\tr(\rhobar_{{f}}):G_{\Q,S}\fleche\bar{\Fp}$ 
its reduction modulo $\mgot_{f}$. If $\tau\in\Gal(\C/\R)$ is non-trivial, the second relation of \eqref{EqEichlerShimura} implies that $\tr(\rho_{\mgot_{f}})(\tau)=0$ hence, as $p\neq 2$, \cite[Proof of Lemma 2.2.3]{WilesOrdinaryLambdaAdic} shows that there exists a unique semi-simple residual representation
\begin{equation}\nonumber
\rhobar_{{f}}:G_{\Q,S}\fleche\Aut_{\bar{\Fp}}(\Tbar(f))
\end{equation}
whose trace is $\tr(\rhobar_{{f}})$. For $\ell$ a rational prime, let $N_{\ell}(\rhobar_{{f}})$ be the Artin conductor of $\rhobar_{{f}}|_{G_{\Q_{\ell}}}$ and let
\begin{equation}\nonumber
N(\rhobar_{{f}})=\produit{\ell\nmid p}{}N_{\ell}(\rhobar_{{f}}).
\end{equation}
be its tame global Artin conductor. Let $\Sigma\supset\{\ell|Np\}$ be a finite set of primes $\Sigma^{p}\cup\{p\}$. Denote by $N(\Sigma)$ the integer
\begin{equation}\nonumber
N(\Sigma)=N(\rhobar_{f})\produit{\ell\in\Sigma^{p}}{}\ell^{\dim_{k}(\rhobar_{f})_{I_{\ell}}}.
\end{equation}

\subsubsection{Deformations}
Henceforth, we make the following assumption.
\begin{HypEnglish}\label{HypIrr}
The $G_{\Q,S}$-representation $\rhobar_{f}$ is absolutely irreducible.
\end{HypEnglish}
Assumption \ref{HypIrr} implies by \cite[Théorème 1]{NyssenPseudo} or \cite[Théorème 4.2]{RouquierPseudoCar} that to all $\Sigma\supset\{\ell|N(\rhobar_{f})p\}$ and all pseudocharacters $\tr(\rho):G_{\Q,\Sigma}\fleche R
$ of dimension 2 with values in a henselian separated ring $R$ with maximal ideal $\mgot$ such that $\tr(\rho)\modulo\mgot$ is equal to $\tr(\rhobar_{f})$ is attached a unique semi-simple representation $(T(\rho),\rho,R)$ whose trace is equal to $\tr(\rho)$. In particular, it follows from \eqref{EqEichlerShimura} that for any discrete valuation ring $\Ocal\subset\Qbar_{p}$ containing the image of $\lambda_{f}$, there exists a unique representation $(T,\rho,\Ocal)$ with $\tr(\rho)=\lambda_{f}$ as well as a unique $(T_{\mgot_{f}},\rho_{\mgot_{f}},\Hecke_{\mgot_{f}})$ whose trace is equal to $\tr(\rho_{\mgot_{f}})$. As pointed out in \cite{CarayolRepresentationsGaloisiennes}, a choice of isomorphism $T_{\mgot_{f}}\simeq\Hecke_{\mgot_{f}}^{2}$ identifies
$$H^{1}_{\et}(X(U)\times_{\Q}\Qbar,\Fcal_{k-2}\tenseur_{\Z}\zp)_{\mgot_{f}}$$
with the square of an ideal $J\subset\Hecke_{\mgot_{f}}$. In general, $J$ is not principal nor is it known to have finite projective dimension as $\Hecke_{\mgot_{f}}$-module.

For $\Sigma\supset\{\ell|N(\rhobar_{f})p\}$, there exists a universal deformation $(T^{u}_{\Sigma},\rho_{\Sigma}^{u},R_{\Sigma}^{u}(\rhobar_{f}))$ of the $G_{\Q,\Sigma}$-representation $\rhobar_{f}$ in the sense of \cite{MazurDeformation}. The universal deformation ring $R^{u}_{\Sigma}(\rhobar_{f})$ admits quotients parametrizing deformations subjected to various supplementary conditions. We are particularly interested in the following cases.
\begin{HypEnglish}\label{HypNO}
The representation $\rhobar_{{f}}|_{G_{\qp}}$ is reducible but not scalar. Hence, it is either an extension 
\begin{equation}\label{EqOrd}
\rhobar_{{f}}|_{G_{\qp}}\simeq\matrice{\bar{\chi}_{1}}{*}{0}{\bar{\chi}_{2}}
\end{equation}
of two distinct characters $\bar{\chi}_{1}\neq\bar{\chi}_{2}$ or a non-trivial extension
\begin{equation}\label{EqOrdSplit}
\rhobar_{{f}}|_{G_{\qp}}\simeq\matrice{\bar{\chi}}{*}{0}{\bar{\chi}}
\end{equation}
of $\bar{\chi}$ by itself.
\end{HypEnglish}
When $\rhobar_{f}|_{G_{\qp}}$ satisfies assumption \ref{HypNO}, we say that $\rhobar_{f}$ is nearly ordinary distinguished. When $\rhobar_{f}|_{I_{p}}$ is moreover an extension
\begin{equation}\nonumber
0\fleche\Fp\fleche\rhobar_{f}|_{I_{p}}\fleche\Fp(-1)\fleche0,
\end{equation}
we say that $\rhobar_{f}$ is nearly ordinary finite.
\begin{DefEnglish}
Let $\rhobar_{f}$ be a nearly ordinary distinguished representation and let $A$ be an element of $\Ccal(\Fp)$. For $\mu:I_{p}\fleche A\croix$ a character, a nearly ordinary distinguished deformation $(T,\rho,A)$ of type $\mu$ is a deformation of $\rhobar_{f}$ such that there exists a short exact sequence $G_{\qp}$-representations
\begin{equation}\nonumber
\suiteexacte{}{}{\chi_{1}}{\rho|_{G_{\qp}}}{\chi_{2}}
\end{equation}
with $\chi_{1}|_{I_{p}}=\mu$. If moreover $\rhobar_{f}$ is nearly ordinary finite, a nearly ordinary finite deformation $(T,\rho,A)$ is a nearly ordinary deformation of $\rhobar_{f}$ with $\chi_{1}|_{I_{p}}=\chi_{2}(1)|_{I_{p}}$ and such that the extension
\begin{equation}\label{EqExtFlat}
0\fleche A(1)\fleche\rho|_{I_{p}}\tenseur\chi_{2}^{-1}\fleche A\fleche0
\end{equation} 
in $\Hun(I_{p},A(1))$ comes from a class in
\begin{equation}\nonumber
\limproj{n}\ \zp^{\ur,\times}/(\zp^{\ur,\times})^{p^{n}}\tenseur_{\zp}A\subset\limproj{n}\ \qp^{\ur,\times}/(\qp^{\ur,\times})^{p^{n}}\tenseur_{\zp}A\simeq\Hun(I_{p},A(1)).
\end{equation}
\end{DefEnglish}
There exists a universal deformation $(T^{\ord}_{\Sigma},\rho^{\ord}_{\Sigma},R^{\ord}_{\Sigma}(\rhobar_{f}))$ of nearly ordinary distinguished deformations of $\rhobar_{f}$. For $(x,y,z)\in\{\vide,\fl\}\times\{\vide,\mu\}\times\{\vide,\chi\}$, the ring $R^{\ord}_{\Sigma}(\rhobar_{f})$ admits quotients $R^{\ord,x}_{\Sigma,y,x}(\rhobar_{f})$ parametrizing deformations which are nearly ordinary finite if $x=\fl$, of type $\mu$ if $y=\mu$ and of determinant $\mu^{2}\chi$ if $(y,z)=(\mu,\chi)$. If $\rho$ is a nearly ordinary deformation of $\rhobar_{f}$ of type $\mu$, then $\rho\tenseur\chi$ is a nearly ordinary deformation of $\rhobar_{f}$ of type $\mu\chi$ if and only if $\chi$ is a deformation of the trivial character. Hence,  $R^{\ord}_{\Sigma}(\rhobar_{f})$ is isomorphic to $R^{\ord}_{\Sigma,\mu}(\rhobar_{f})[[\Gamma]]$ where we recall that $\Gamma$ is the Galois group of the $\zp$-extension of $\Q$. If the image of $\lambda_{f}(T(p))$ under our fixed embedding of $\C$ in $\Qbar_{p}$ is a $p$-adic unit, we say that $f$ is $p$-ordinary (a condition that depends in general on our choice of $\iota_{\infty,p}$). When $f$ is $p$-ordinary, $\rho_{f}|_{G_{\qp}}$ is an extension
\begin{equation}\nonumber
\suiteexacte{}{}{\lambda_{f}}{\rho_{f}|_{G_{\qp}}}{\lambda_{f}^{-1}\epsi\chi_{\cyc}^{1-k}}
\end{equation}
where $\lambda_{f}:G_{\qp}\fleche\Qbar_{p}\croix$ is the unramified character sending $\Fr(p)$ to $\lambda_{f}(T(p))$. Hence, $\rho_{f}$ is a nearly ordinary deformation of $\rhobar_{f}$ with trivial type and there thus exists a unique $x\in\Spec R_{\Sigma,\Id,\epsi\chi_{\cyc}^{k-1}}^{\ord}(\rhobar_{f})$ such that $\rho_{f}$ is isomorphic to $\rho_{x}=\rho^{\ord}_{\Sigma,\Id,\epsi\chi_{\cyc}^{k-1}}\modulo x$.
\begin{HypEnglish}\label{HypFlat}
There exists a commutative finite flat $p$-torsion group scheme $G$ over $\zp$ and a character $\bar{\mu}$ such that $\rhobar_{{f}}\tenseur\bar{\mu}^{-1}$ is isomorphic as $\bar{\Fp}[G_{\qp}]$-module to $(G\times_{\zp}\Qbar_{p})[p]$.
\end{HypEnglish}
When $\rhobar_{f}|_{G_{\qp}}$ satisfies assumption \ref{HypFlat}, we say that $\rhobar_{f}$ is flat. 
\begin{DefEnglish}
Let $\rhobar_{f}$ be a flat representation and let $A$ be an element of $\Ccal(\Fp)$. A flat deformation $(T,\rho,A)$ of type $\mu:G_{\qp}\fleche A\croix$ is a deformation of $\rhobar_{f}$ such that $\det(\rho\tenseur\mu^{-1})|_{I_{p}}$ is equal to $\chi_{\cyc}^{-1}$ and such that for all $n\geq1$ there exists a finite flat group scheme $G$ over $\zp$ with an $A$-action such that $\rho\tenseur\mu^{-1}\modulo\mgot_{A}^{n}$ is isomorphic to $(G\times_{\zp}\Qbar_{p})[\mgot_{A}^{n}]$.
\end{DefEnglish}
By \cite[Theorem 1.1]{RamakrishnaFlat}, there exists a universal deformation $(T^{\fl}_{\Sigma},\rho^{\fl}_{\Sigma},R^{\fl}_{\Sigma}(\rhobar_{f}))$ of flat deformations of $\rhobar_{f}$. The ring $R^{\fl}_{\Sigma}(\rhobar_{f})$ admits a quotient $R^{\fl}_{\Sigma,\mu}(\rhobar_{f})$ parametrizing flat deformations with type $\mu$ and, as above, the fact that being a flat deformation is stable by twisting by a character deforming the identity implies that  $R^{\fl}_{\Sigma}(\rhobar_{f})$ is isomorphic to $R^{\fl}_{\Sigma,\mu}(\rhobar_{f})[[\Gamma]]$. A deformation $\rho$ can be both nearly ordinary distinguished and flat, in which case it is nearly ordinary finite.

The deformation rings $R_{\Sigma,\mu}^{\ord}(\rhobar_{f})$ and $R_{\Sigma,\mu}^{\fl}(\rhobar_{f})$ are called minimal if $\Sigma$ is equal to $N(\rhobar_{f})$. For the sake of completeness, we note that $R^{\ord,\fl}_{\Sigma,\mu,\psi}(\rhobar_{f})$ (resp. $R^{\ord}_{\Sigma,\mu,\psi}(\rhobar_{f})$) is minimal if $\rhobar_{f}$ is nearly ordinary finite (resp. is nearly ordinary but not nearly ordinary finite) and $\Sigma$ is equal to $N(\rhobar_{f})$ but we will not make use of this notion.
\begin{HypEnglish}\label{HypIrrTW}
Let $p^{*}$ be $(-1)^{(p-1)/2}p$. The $G_{\Q(\sqrt{p^{*}})}$-representation $\rhobar_{f}$ is absolutely irreducible.
\end{HypEnglish}
When $\rhobar_{f}$ satisfies assumption \ref{HypIrrTW} and is either nearly ordinary distinguished or flat, the following theorem of \cite{WilesFermat,TaylorWiles} holds. 
\begin{TheoEnglish}\label{TheoTaylorWiles}
Let $f\in S_{k}(U,\chi)$ be an eigencuspform. Assume $\rhobar_{f}$ satisfies assumption \ref{HypIrrTW} and either assumption \ref{HypNO} or \ref{HypFlat}. Let $\Hecke$ be the Hecke ring $\Hecke^{\red}(U)_{\mgot_{f}}$ acting on $S_{k}(U,\chi)$ if $\rhobar_{f}$ satisfies assumption \ref{HypFlat} and let it be the ring generated by $\Hecke^{\red}(U)_{\mgot_{f}}$ acting on $S_{k}(U,\chi)$ along with $T(p)$ if $\rhobar_{f}$ satisfies assumption \ref{HypNO}. Then $\Hecke$ is a complete intersection ring of dimension 1 isomorphic to $R^{\ord}_{\Sigma,\Id,\chi}(\rhobar_{f})$ or $R_{\Sigma,\Id}^{\fl}(\rhobar_{f})$ depending on whether $\rhobar_{f}$ is nearly ordinary distinguished or flat and $\Hun_{\et}(X(U)\times_{\Q}\Qbar,\Fcal_{k-2}\tenseur_{\Z}\Ocal)_{\mgot_{f}}$ is free of rank 2 as $\Hecke$-module.
\end{TheoEnglish}
\begin{proof}
This is part of output of the method of Tayor-Wiles systems in this setting; see \cite{WilesFermat,TaylorWiles} for the original argument and \cite{DiamondHecke,FujiwaraDeformation} for a compact statement of the results needed here. 
\end{proof}
While the isomorphism between Hecke rings and universal deformation ring is frequently considered the deepest statement of theorem \ref{TheoTaylorWiles}, it is in fact the other two which are crucial in this manuscript. Under the hypotheses of theorem \ref{TheoTaylorWiles}, it follows from the discussion above that $R_{\Sigma}^{\fl}(\rhobar_{f})$ is isomorphic to $\Hecke[[\Gamma]]$ and that both these rings are complete intersection of dimension 2.
\subsection{Motives attached to modular forms}\label{SubMotives}
Let $\bar{E}^{(k-2)}$ be the $(k-2)$-fold fiber product of $\bar{E}$ with itself over $X(N)$. Let $KS_{k}$ be the canonical desingularization of $\bar{E}^{(k-2)}$ constructed in \cite{DeligneModulaires} (see also \cite[Section 3]{SchollMotivesModular}). The symmetric group $\Grsym_{k-2}$ acts on $\bar{E}^{(k-2)}$ by permutations, the $(k-2)$-th power of $(\Z/N\Z)^{2}$ acts by translation and $\mu_{2}^{k-2}$ acts by inversion in the fibers. Let $\tilde{G}_{k-2}$ be the wreath product of $((\Z/N\Z)^{2}\rtimes\mu_{2})^{k-2}$ with $\Grsym_{k-2}$. Then $\tilde{G}_{k-2}$ acts by automorphisms on $\bar{E}^{(k-2)}$ and thus on $KS_{k}$. Let $\epsi$ be the character of $\tilde{G}_{k-2}$ which is trivial on $(\Z/N\Z)^{2(k-2)}$, the product map on $\mu_{2}^{k-2}$ and signature on $\Grsym_{k-2}$. Let $\Pi_{\epsi}\in\Z[\frac{1}{2Nk!}][\tilde{G}_{k-2}]$ be the projector attached to $\epsi$. 

The category $CH(\Q)$ of Chow motives is the pseudo-abelian envelope of the category of proper smooth schemes over $\Q$ with degree zero correspondences modulo rational equivalence as morphisms. A Chow motive is thus a pair $(X,e)$ with $X/\Q$ proper and smooth and $e$ a projector of $CH^{\dim X}(X\times X)_{\Q}$. The pair $(KS_{k},\Pi_{\epsi})$ constructed above is thus a Chow motive. We denote it by $\Wcal_{N}^{k-2}$ and its Betti (resp. étale) realization by $^{B}\Wcal_{N}^{k-2}$ (resp. by $^{\et}\Wcal_{N}^{k-2}$). By \cite[Theorem 1.2.1]{SchollMotivesModular}, there is a canonical isomorphism of $\Q[\Gal(\C/\R)]$-modules
 \begin{align}
 ^{B}\Wcal_{N}^{k-2}=H^{k-1}(KS_{k}(\C),\Q)(\epsi)\isocan\Hun(X(N)(\C),\Fcal_{k-2}\tenseur_{\Z}\Q)
 \end{align}
as well as a canonical isomorphism of $\qp[\Gal(\Qbar/\Q)]$-modules
 \begin{align}
 ^{\et}\Wcal_{N}^{k-2}=H^{k-1}_{\et}(KS_{k}\times_{\Q}\Qbar,\qp)(\epsi)\isocan\Hun_{\et}(X(N)\times_{\Q}\Qbar,\Fcal_{k-2}\tenseur_{\Z}\qp).
 \end{align}
For a number field $L$, a Grothendieck motive over $\Q$ with coefficients in $L$ is an object in the category of motives over $\Q$ in which $\Hom(h(X),h(Y))$ is the group of algebraic cycles on $X\times Y$ of codimension $\dim Y$ tensored over $\Q$ with $L$ modulo homological equivalence. Fix a number field $F$ containing all the eigenvalues of Hecke operators acting on eigenforms in $S_{k}(U(N))$. The image of  $\Wcal_{N}^{k-2}$ in the category of Grothendieck motive over $\Q$ with coefficients in $F$ decomposes under the action of the Hecke correspondences.  Let $f\in S_{k}(U_{1}(N))$ be a newform and denote as before by $\lambda_{f}$ the map sending a Hecke operator to the corresponding eigenvalues. Let $\Wcal(f)$ be the largest Grothendieck sub-motive of $\Wcal_{N}^{k-2}$ over $\Q$ with coefficients in $F$ on which $\Hecke^{\red}(N)$ acts through $\lambda_{f}$. We denote by $\Wcal(f)_{B}$ (resp. $\Wcal(f)_{\dR}$, resp. $\Wcal(f)_{\et,p}$) the Betti (resp. de Rham, resp. $p$-adic étale) realization of $\Wcal(f)$. The $\Qbar_{p}[G_{\Q}]$-module $\Wcal(f)_{\et,p}$ is isomorphic to $M(f)$.

\subsection{Hida theory}
Assume in this sub-section $f\in S_{k}(U_{1}(N))$ to be $p$-ordinary and let $\Ocal\subset\Qbar_{p}$ be a discrete valuation ring containing the image of $\lambda_{f}$. The diamond correspondences $\diamant{a}$ with $a\equiv 1\modulo p$ and $a$ locally trivial outside $p$ act on the tower of modular curves
$$X_{1}(Np^{\infty})=\limproj{s}\ X_{1}(Np^{s}).$$
Let $\Lambda_{\Hi}=\Ocal[[\Gamma_{\Hi}]]\simeq\Ocal[[1+p\zp]]$ be the completed group $\Ocal$-algebra of these correspondences. It is a complete local regular ring of dimension 2. Let $\gamma$ be a topological generator of $\Gamma_{\Hi}$. For $k\geq2$ an integer and $\epsi$ a finite order character of $\Gamma_{\Hi}$ factoring through $1+p^{s+1}\zp$, an arithmetic point of weight $k$, level $s$ and character $\epsi$ of $\Lambda_{\Hi}$ is an $\Ocal$-algebra morphism
\begin{align}\nonumber
\phi:\Lambda_{\Hi}&\fleche\Qbar_{p}\\\nonumber
\gamma&\longmapsto\epsi(\gamma)\chi_{\cyc}^{k-2}(\gamma)
\end{align}
Here, $\gamma$ is considered as an element of $G_{\Q}$ via the identification of $\Gamma_{\Hi}$ with the Galois group of the unique $\zp$-extension of $\Q$. If $A$ is a finite $\Lambda_{\Hi}$-algebra, an arithmetic point $\psi\in\Hom(A,\Qbar_{p})$ of $A$ is an $\Ocal$-algebra morphism inducing an arithmetic point on $\Lambda_{\Hi}$. If $\phi$ is an $\Ocal$-algebra map from $\Lambda_{\Hi}$ to $\Qbar_{p}$, let $\Ocal_{\phi}$ be the smallest discrete valuation ring containing the image of $\phi$. If $M$ is a $\Lambda_{\Hi}$-module, we denote by $M[\phi]$ the quotient of $M$ on which $\Lambda$ acts through $\phi$.

Let $\Hecke_{\cl}^{\ord}(N)$ be the inverse limit of ordinary Hecke algebras
\begin{equation}\label{EqHord}
\Hecke_{\cl}^{\ord}(N)=\limproj{s}\ \eord\hgot(U_{1}(Np^{s}))\tenseur_{\Z}\Ocal
\end{equation}
where $\eord$ is Hida's projector 
\begin{equation}\nonumber
\eord=\underset{n\rightarrow\infty}{\lim} T(p)^{n!}.
\end{equation}
If $M$ is a finite $\Hecke_{\cl}(Np^{s})$-module, then we denote by $M^{\ord}$ the $\Hecke^{\ord}_{\cl}(N)$-module $\eord M$. Let the ordinary reduced Hecke algebra $\Hecke^{\red,\ord}(Np^{s})\subset\Hecke_{\cl}^{\ord}(Np^{s})$ be the sub $\Ocal$-algebra generated by the diamond operators, the Hecke operators $T(\ell)$ for $\ell$ such that $\ell\nmid Np$ and the Hecke operator $T(p)$. Let the ordinary new Hecke algebra be $\Hecke^{\new,\ord}(Np^{s})$. The Hecke algebras $\Hecke^{\red,\ord}(N)$ and $\Hecke^{\new,\ord}(N)$ are the inverse limits of the $\Hecke^{\red,\ord}(Np^{s})$ and $\Hecke^{\new,\ord}(Np^{s})$.
\begin{align}\label{EqRedOrd}
\Hecke^{\red,\ord}(N)&=\limproj{s}\ \eord\Hecke^{\red}(U_{1}(Np^{s}))\\\label{EqNewOrd}
\Hecke^{\new,\ord}(N)&=\limproj{s}\ \eord\Hecke^{\new}(U_{1}(Np^{s}))
\end{align}
All these algebras are finitely generated as $\Lambda_{\Hi}$-modules. 


Consider the complex
\begin{equation}\nonumber
\RGamma_{\et}(X_{1}(Np^{\infty})\times_{\Q}\Qbar,\Ocal)^{\ord}=\limproj{s}\RGamma_{\et}(X_{1}(Np^{s})\times_{\Q}\Qbar,\Ocal)\tenseur_{\Hecke_{\cl}(U_{1}(Np^{s}))}\Hecke_{\cl}^{\ord}(N).
\end{equation}
As the action on $H^{i}(X_{1}(Np^{\infty})\times_{\Q}\Qbar,\Ocal)^{\ord}$ for $i=0,2$ is by multiplication by $p$ and is invertible, the only non-zero cohomology module of $\RGamma_{\et}(X_{1}(Np^{\infty})\times_{\Q}\Qbar,\Ocal)^{\ord}$ is $M^{\ord}=H^{1}(X_{1}(Np^{\infty})\times_{\Q}\Qbar,\Ocal)^{\ord}$. As
\begin{equation}\nonumber
\RGamma_{\et}(X_{1}(Np^{\infty})\times_{\Q}\Qbar,\Ocal)^{\ord}\Ltenseur_{\Lambda_{\Hi},\phi}\Ocal_{\phi}\simeq\RGamma_{\et}(X_{1}(Np^{s})\times_{\Q}\Qbar,\Ocal)^{\ord}[\phi]
\end{equation}
for $\phi$ an arithmetic point of $\Lambda_{\Hi}$ of weight $2$ and level $s$ with values in $\Ocal_{\phi}$, the $\Lambda_{\Hi}$-module $M^{\ord}$ is free of finite rank and satisfies
\begin{equation}\nonumber
M^{\ord}\tenseur_{\Lambda_{\Hi},\phi}\Ocal_{\phi}\simeq\Hun_{\et}(X_{1}(Np^{s})\times_{\Q}\Qbar,\Ocal_{\phi})^{\ord}[\phi].
\end{equation} 
From this and the contraction isomorphism
\begin{equation}\nonumber
\RGamma_{\et}(X_{1}(Np^{\infty})\times_{\Q}\Qbar,\Ocal)^{\ord}\simeq\RGamma_{\et}(X_{1}(Np^{\infty})\times_{\Q}\Qbar,\Fcal_{k-2}\tenseur_{\Z}\Ocal)^{\ord}
\end{equation}
for $k\geq2$, it follows that 
\begin{equation}\label{EqControlHecke}
\Hecke\tenseur_{\Lambda_{\Hi},\phi}\Ocal_{\phi}\simeq\Hecke_{k}(Np^{s})[\phi].
\end{equation}
for $\Hecke=\Hecke_{\cl}^{\ord}(N),\Hecke=\Hecke^{\red,\ord}(N)$ or $\Hecke=\Hecke^{\new,\ord}(N)$ and $\phi$ an arithmetic point of weight $k\geq2$ and level $s$. Moreover, if $\lambda$ is an arithmetic prime of $\Hecke_{\cl}^{\ord}(N)$ above an arithmetic prime of $\Lambda_{\Hi}$ of weight $k$ and level $s$, there exists a unique eigencuspform $g\in S_{k}(U_{1}(Np^{s}))$ such that $\lambda_{g}$ extended to $\Hecke_{\cl}^{\ord}(N)$ is equal to $\lambda$ and hence such that 
$$M^{\ord}\tenseur_{\Hecke_{\cl}^{\ord}(N),\lambda}\Qbar_{p}\simeq M(g)$$
as $\Qbar_{p}[G_{\Q}]$-modules.
 
Let $\Hecke$ be either $\Hecke^{\red,\ord}(N)$ or $\Hecke^{\new,\ord}(N)$. Then there exists a unique maximal ideal $\mgot$ of $\Hecke$ such that $\lambda_{f}$ factors through $\Hecke_{\mgot}$. The complex
 \begin{equation}\nonumber
M^{\ord}_{\mgot}=\limproj{s}\RGamma_{\et}(X_{1}(Np^{s})\times_{\Q}\Qbar,\Ocal)\tenseur_{\hgot(U_{1}(Np^{s}))}\Hecke_{\mgot}.
\end{equation}
is concentrated in degree 1. There exists a pseudo-character $\tr(\rho_{\mgot}):G_{\Q}\fleche\Hecke_{\mgot}
$
of dimension 2 of $G_{\Q}$ such that the composition of $\tr(\rho_{\mgot})$ with an arithmetic point $\lambda_{g}$ is $\tr(\rho_{g})$ as defined in subsection \ref{SubGalois}. Recall that $\rhobar_{f}$ satisfies the assumption \ref{HypIrr}. There thus exists a $G_{\Q}$-representation $(T_{\Hi},\rho_{\mgot},\Hecke_{\mgot})$ unique up to isomorphism whose trace is $\tr(\rho_{\mgot})$ and which is characterized by
\begin{equation}\label{EqEichlerShimuraHida}
\begin{cases}
\tr\rho_{\mgot}(\Fr(\ell))=T(\ell)\\
\det\rho_{\mgot}(\Fr(\ell))=\ell\diamant{\ell}
\end{cases}
\end{equation}
for all $\ell\notin S$. By \cite{WilesOrdinaryLambdaAdic,HidaNearlyOrdinary}, the $G_{\qp}$-representation $T_{\Hi}$ is reducible. If moreover $\rhobar_{f}$ satisfies assumption \ref{HypNO}, so if it is nearly ordinary distinguished, then $T_{\Hi}$ fits in a short exact sequence
\begin{equation}\nonumber
\suiteexacte{}{}{T_{\Hi}^{+}}{T_{\Hi}}{T_{\Hi}^{-}}
\end{equation}
of $\Hecke_{\mgot}[G_{\qp}]$-modules free of positive ranks as $\Hecke$-modules and $M^{\ord}_{\mgot}$ is isomorphic to $T_{\Hi}$ (and so is in particular free of rank 2 as $\Hecke_{\mgot}$-module).
\begin{Prop}\label{PropTW}
Assume that $\rhobar_{f}$ satisfies assumptions \ref{HypIrrTW} and \ref{HypNO}. Let $\Sigma\supset\{\ell|Np\}$ be a finite set of primes and let $N(\Sigma)$ denote the integer of section \ref{SubGalois}. Then $\Hecke=\Hecke^{\red,\ord}(N(\Sigma))_{\mgot}[[\Gamma]]$ is a complete intersection ring of dimension 3 isomorphic to $R^{\ord}_{\Sigma}(\rhobar_{f})$ and $M^{\ord}_{\mgot}\tenseur_{\Hecke}\Hecke[[\Gamma]]$ is a free $\Hecke[[\Gamma]]$-module of rank 2.
\end{Prop}
\begin{proof}
Granted theorem \ref{TheoTaylorWiles}, this follows from equation \eqref{EqControlHecke}. See also \cite[Theorem 4.1]{BockleDensity} and \cite[Corollary 11.5]{FujiwaraDeformation}.
\end{proof}
\section{The ETNC for modular motives}\label{SectionETNC}
\subsection{\Nekovar-Selmer complexes, étale cohomology and the determinant functor}\label{SubNeko}
\subsubsection{Review of the determinant functor}
Let $R$ be a commutative ring. A graded invertible module $(P,r)$ is a pair formed with a projective $R$-module $P$ of rank one and a locally constant map $r$ from $\Spec R$ to $\Z$. If $(P_{1},r)$ and $(P_{2},r)$ are graded invertible module with the same $r$, the statement that they are isomorphic is tautologically true. Consequently, we insist in this manuscript that any isomorphism between graded invertible modules be completely specified, and ideally canonical, that is to say independent of any choice beyond those incorporated in the definitions of $(P_{1},r)$ and $(P_{2},r)$. Nevertheless, it is often the case that we can make this specification only up to a choice of a unit in $R$, in which case we say that $(P_{1},r)$ and $(P_{2},r)$ are isomorphic up to a choice of a unit. 

The determinant functor $\Det_{R}$ of \cite{MumfordKnudsen,DeligneDeterminant} is the functor
\begin{equation}\nonumber
\Det_{R}P=\left(\underset{R}{\overset{{\rank_{R}P}}{\bigwedge}}P,\rank_{R}P\right)
\end{equation}
from the category of finite projective $R$-modules to the category of graded invertible $R$-modules (with morphisms restricted to isomorphisms). 
A perfect complex $C$ of $R$-modules is an object in the derived category of $R$-modules represented by a bounded complex of projective $R$-modules of finite ranks. The determinant functor extends to a functor from the category of perfect complexes of $R$-modules with morphisms restricted to quasi-isomorphisms to the category of graded invertible $R$-modules by setting
\begin{equation}\label{EqDefDet}
\Det_{R}C=\underset{i\in\Z}{\bigotimes}{}\Det_{R}^{(-1)^{i}}C^{i}
\end{equation}
for any representation of $C$ such that the $C^{i}$ are projective of finite ranks. The determinant functor commutes with derived tensor product and there is a canonical isomorphism between $\Det_{R}(0)$ and $(R,0)$.
\subsubsection{\Nekovar-Selmer complexes and étale cohomology}
Let $\Q\subset K\subset \Qbar$ be an extension of $\Q$ with ring of integers $\Ocal_{K}$. Let $S_{p}$ be the set of primes of $\Ocal_{K}$ over $p$. Let $U=\Spec\Ocal_{K}[1/p]$ be the open subset of $\Spec\Ocal_{K}$ defined by $\Spec\Ocal_{K}-S_{p}$. Let $M$ be a finite $p$-torsion module with a continuous action of $G_{K}$ and let $S\supset S_{p}$ be a finite set of finite primes of $\Ocal_{K}$ such that $M$ is a representation of $G_{K,S}$. Then $M$ defines a locally constant étale sheaf $M_{\et}$ on $V=\Spec\Ocal_{K}-S$. 

A local condition at $v\in S$ is a pair $(C^{\bullet}_{?}(G_{K_{v}},M),i_{v})$ where $C^{\bullet}_{?}(G_{K_{v}},M)$ is a bounded complex and $i_{v}:C^{\bullet}_{?}(G_{K_{v}},M)\fleche C^{\bullet}(G_{K_{v}},M)$ is a morphism of complexes. Denote also by $$i:C^{\bullet}(G_{K},M)\fleche \sommedirecte{v\in S}{}C^{\bullet}(G_{K_{v}},M)$$ the direct sum of the localization maps at $S$ and by $\iota$ the map
\begin{equation}\nonumber
i-\sommedirecte{v\in S}{}i_{v}.
\end{equation}
The \Nekovar-Selmer complex $\RGamma_{?}(G_{K,S},M)$ of $M$ (see \cite{SelmerComplexes}) attached to the local conditions $(C^{\bullet}_{?}(G_{K_{v}},M),i_{v})$ for $v\in S$ is the complex
\begin{equation}\label{EqDefSelmer}
\Cone\left(C^{\bullet}(G_{K,S},M)\oplus\sommedirecte{v\in S}{}C^{\bullet}_{?}(G_{K_{v}},M)\overset{\iota}{\fleche}\sommedirecte{v\in S}{}C^{\bullet}(G_{K_{v}},M)\right)[-1]
\end{equation}
seen in the derived category. In a slight abuse of notations, we henceforth do not distinguish complexes and their images in the derived category so that we write $\RGamma_{?}(G_{K_{v}},M)$ for $C^{\bullet}_{?}(G_{K_{v}},M)$ and likewise in all similar situations. Henceforth, we also systematically assume that $(C^{\bullet}_{?}(G_{K_{v}},M),i_{v})$ is equal to $(C^{\bullet}(G_{K_{v}},M),\Id_{v})$ for all $v\in S_{p}$.

When $\RGamma_{?}(G_{K_{v}},T)$ is the zero complex for all $v\in S-S_{p}$, the attached \Nekovar-Selmer complex is  the complex of cohomology with compact support outside $p$
\begin{equation}\nonumber
\RGamma_{c}(G_{K,S},M)=\Cone\left(\RGamma(G_{K,S},M){\fleche}\sommedirecte{v\in S\backslash S_{p}}{}\RGamma(G_{K_{v}},M)\right)[-1].
\end{equation}
In the opposite direction, when $\RGamma_{?}(G_{K_{v}},M)$ is equal to $\RGamma(G_{K_{v}},T)$ and $i_{v}$ is the identity for all $v\in S$, the \Nekovar-Selmer complex is the complex $\RGamma(G_{K,S},M)$ of continuous cochains with values in $M$. Particularly important in this manuscript is the \Nekovar-Selmer complex attached to the unramified condition $\RGamma(G_{K_{v}},T^{I_{v}})$ at $v\nmid p$ with its natural map to $\RGamma(G_{K_{v}},M^{I_{v}})$ and to the relaxed condition $\RGamma(G_{K_{v}},T)$ at $v|p$. Explicitly, this is the complex:
\begin{equation}\nonumber
\Cone\left(\RGamma(G_{K,S},M)\oplus\sommedirecte{v\in S\backslash S_{p}}{}\RGamma(G_{K_{v}}/I_{v},M^{I_{v}})\fleche\sommedirecte{v\in S\backslash S_{p}}{}\RGamma(G_{K_{v}},M)\right)[-1]
\end{equation}
We denote it by $\RGamma_{f}(G_{K,S},M)$. The following lemma is well-known.
\begin{LemEnglish}\label{LemEtaleSelmer}
Let $i$ be the inclusion $V=U-\{x\in S\}\plonge U$. There is a canonical isomorphism between $\RGamma_{f}(G_{K,S},M)$ and $\RGamma_{\et}(U,i_{*}M_{\et})$.
\end{LemEnglish}
In the following, we need to consider étale sheaves of $R$-modules with $R$ possibly of large Krull dimension. Though it is certainly well-known that the formalism of étale cohomology carries over to these rings (by taking inverse limits on $n$ of truncated projective resolutions over $R/\mgot^{n}$ and using the fact that $\RGamma_{\et}(X,-)$ is a triangulated way-out functor) and thus that lemma \ref{LemEtaleSelmer} identifies $\RGamma_{f}(G_{K,S},M)$ with $\RGamma_{\et}(\Spec\Ocal_{K}[1/p],i_{*}M_{\et})$ for all $G_{K,S}$-representation $M$ over $R$, this author found a published reference hard to find. By contrast, all necessary results for Galois cohomology with coefficients in admissible modules can be found in \cite{SelmerComplexes}. For this reason, the objects intervening in the ETNC are described in this manuscript using Galois cohomology and the careful reader may wish to consider the notation $\RGamma_{\et}(\Spec\Ocal_{K}[1/p],i_{*}M_{\et})$, which we abbreviate as $\RGamma_{\et}(\Ocal_{K}[1/p],M)$, as a placeholder for $\RGamma_{f}(G_{K,S},M)$ if deemed necessary.

If in addition to being a $G_{K,S}$-representation, $M$ is a perfect complex of $R$-modules, then so are $\RGamma(G_{K,S},M)$, $\RGamma_{c}(G_{K,S},M)$ and $\RGamma(G_{K_{v}},M)$ for all $v$. If $M^{I_{v}}$ is moreover a perfect complex of $R$-modules for all $v\in S$, then $\RGamma_{f}(G_{K,S},M)$ is a perfect complex. 

\subsection{Integral lattices in the cohomology of modular curves}\label{SubLattices}
\subsubsection{Integral lattices}
In this sub-section is a local integral domain. Let $(T,\rho,R)$ be a $G_{K,S}$-representation of rank 2 and let $(V,\rho,\Frac(R))$ be the representation obtained by tensor product with $\Frac(R)$. Let $v\nmid p$ be a finite place of $\Ocal_{K}$. If $T^{I_{v}}$ is of rank one, assume that $\Fr(v)-1$ acts on $V^{I_{v}}$ by multiplication by an element of $R$ (this is of course always true if $R$ is integrally closed or if $T^{I_{v}}$ can be completed in a basis of $T$).
\begin{DefEnglish}\label{DefXcaliv}
The graded invertible module $\Xcali_{v}(T)$ is defined as follows.
\begin{equation}\nonumber
\Xcali_{v}(T)=\begin{cases}
\Det_{R}\RGamma(G_{K_{v}}/I_{v},T^{I_{v}})\textrm{ if $\rank_{R}T^{I_{v}}\neq1$.}\\
\Det_{R}[R\overset{\Fr(v)-1}{\fleche} R]\textrm{ if $\rank_{R}T^{I_{v}}=1$.}
\end{cases}
\end{equation}
Here the complex $[R\overset{\Fr(v)-1}{\fleche} R]$ is placed in degree $0,1$.
\end{DefEnglish}
The module $\Xcali_{v}(T)$ recovers the determinant of the unramified cohomology of $T$ when both are defined and is compatible with change of rings provided the rank of inertia invariants remains constant in the sense of the following lemma.
\begin{LemEnglish}\label{LemXcaliBienDef}
If $T^{I_{v}}$ is a perfect complex of $R$-modules, then there is a canonical isomorphism
\begin{equation}\nonumber
\Xcali_{v}(T)\overset{\can}{\simeq}\Det_{R}\RGamma(G_{K_{v}}/I_{v},T^{I_{v}}).
\end{equation}
If $R\fleche R'$ is a local morphism of integral domains such that
$$\rank_{R'}(T\tenseur_{R}R')^{I_{v}}=\rank_{R}T^{I_{v}}$$
then $\Xcali_{v}(T)\tenseur_{R}R'$ is canonically isomorphic to $\Xcali_{v}(T\tenseur_{R}R')$.
\end{LemEnglish} 
\begin{proof}
If $T^{I_{v}}$ is a perfect complex of $R$-modules, then so is $\RGamma(G_{K_{v}}/I_{v},T^{I_{v}})$. Hence $\Det_{R}\RGamma(G_{K_{v}}/I_{v},T^{I_{v}})$ is well-defined. The first assertion of the lemma is non-tautological only if $\rank_{R}T^{I_{v}}=1$. In that case, a finite projective resolution of $T^{I_{v}}$ yields a projective resolution of $(\Fr(v)-1)T^{I_{v}}$ and computing $\Det_{R}(\Fr(v)-1)T^{I_{v}}\tenseur_{R}\Det^{-1}T^{I_{v}}$ using these resolutions yields the desired result. If $\rank_{R} T^{I_{v}}=0$, then both $\Xcali_{v}(T)\tenseur_{R}R'$ and $\Xcali_{v}(T\tenseur_{R}R')$ are canonically isomorphic to $(R',0)$. If $\rank_{R}T^{I_{v}}=1$, they are both canonically isomorphic to $\Det_{R'}[R'\overset{\Fr(v)-1}{\fleche}R']$. If $\rank_{R}T^{I_{v}}=2$, then both $T$ and $T\tenseur_{R}R'$ are unramified so the canonical isomorphism
\begin{equation}\nonumber
\RGamma(G_{K_{v}}/I_{v},T^{I_{v}})\Ltenseur_{R}R'\isocan\RGamma(G_{K_{v}}/I_{v},T\tenseur_{R}R')
\end{equation}
yields the result after taking determinant. The second assertion is thus true.
\end{proof}
Let $T$ be a $G_{K,S}$-representation such that $\Xcali_{v}$ is defined for all $v\nmid p$. 
\begin{DefEnglish}\label{DefXcali}
The graded invertible $R$-module $\Xcali(T)$ is defined to be:
\begin{equation}\nonumber
\Det_{R}\RGamma_{c}(G_{K,S},T)\tenseur_{R}\produittenseur{v\in S\backslash S_{p}}{}\Xcali_{v}(T)
\end{equation}
\end{DefEnglish}
We recall that the subscript $c$ denotes cohomology compactly supported  outside $p$. Though $\Xcali(T)$ has \textit{a priori} no special relevance for an arbitrary $T$, note that by construction there are canonical isomorphisms 
\begin{align}\label{EqIsoXcaliSel}
\Xcali(T)\overset{\can}{\simeq}\Det^{-1}_{R}\RGamma_{f}(G_{K,S},T)\overset{\can}{\simeq}
\Det^{-1}_{R}\RGamma_{\et}(\Ocal_{K}[1/p],T)
\end{align}
whenever all the objects appearing in \eqref{EqIsoXcaliSel} are well defined.
\subsubsection{The Weight-Monodromy conjecture for modular motives}
Let $f\in S_{k}(U)$ be a newform and denote by $\Hecke$ the new Hecke algebra $\Hecke^{\new}(U)$. Let $\aid\in\Spec\Hecke_{\mgot_{f}}$ be a minimal prime ideal. Let $R(\aid)$ be the domain $\Hecke_{\mgot_{f}}/\aid$ and $\Frac(R(\aid))$ its field of fraction. The pseudo-character $\tr(\rho_{\mgot_{f}})$ modulo $\aid$ has values in $R(\aid)$ so there exists a $G_{\Q,S}$-representation $(V,\rho_{\mgot_{f}},\Frac(R(\aid)))$ whose trace is $\tr(\rho_{\mgot_{f}})\modulo\aid$, and hence a $G_{K,S}$-representation with the same properties by restriction.
\begin{Prop}\label{PropWeightMonodromy}
Let $\ell\nmid p$ be a finite place. Let $T\subset V$ be a sub-representation with coefficients in $R(\aid)$. Then $\Xcali_{\ell}(T)$ is well defined and there is a canonical isomorphism
\begin{equation}\label{EqCanIsoXcali}
\Xcali_{\ell}(T)\tenseur_{R(\aid),\lambda}\Ocal\overset{\can}{\simeq}\Xcali_{\ell}(T\tenseur_{R(\aid),\lambda}\Ocal)
\end{equation}
for all modular map $\lambda:R(\aid)\fleche\Ocal$.
\end{Prop}
\begin{proof}
When $V^{I_{\ell}}$ is one-dimensional, the compatibility between the local and global Langlands correspondence at $\ell$ implies that $\det(1-\Fr(\ell)X|V^{I_{\ell}})=1-T(\ell)X$. So the eigenvalue $\alpha_{\ell}$ of $\Fr(\ell)-1$ on $V^{I_{\ell}}$ is an element of $R(\aid)$ and $\Xcali_{\ell}(T)$ is well-defined. 

Let $\lambda:R(\aid)\fleche\Ocal$ be a modular map (so $\Ocal$ is a discrete valuation ring in $\Qbar_{p}$) and let $T_{\Ocal}$ and $V_{\Ocal}$ denote respectively $T\tenseur_{R(\aid),\lambda}\Ocal$ and $T_{\Ocal}\tenseur\Frac(\Ocal)$. By the second assertion of lemma \ref{LemXcaliBienDef}, it is enough to prove that $\rank_{R(\aid)} T^{I_{\ell}}$ is larger than $\rank_{\Ocal} T_{\Ocal}^{I_{\ell}}$.

Non-zero elements of $\Qbar_{p}$ are not in the kernel of $\lambda$ so if $\s\in I_{\ell}$ acts on $V$ non-trivially through a finite quotient, then its action is also non-trivial on $T_{\Ocal}$. It is thus enough to prove that $\rank_{R(\aid)} T^{U}$ is larger than $\rank_{\Ocal} T_{\Ocal}^{U}$ for $U$ a finite index subgroup of $I_{\ell}$. By Grothendieck's monodromy theorem \cite[Page 515]{SerreTate}, we can choose $U$ such that $V^{U}$ is quasi-unipotent, in which case $\rank_{R(\aid)}T^{U}$ is at least 1 and is exactly 1 if the monodromy operator is of rank 1. Because the representation $V_{\Ocal}$ is a pure $G_{\Q_{\ell}}$-module by Ramanujan's conjecture (proved for modular forms in \cite[Théorème A]{CarayolHilbert}), the eigenvalues of a lift $\s$ of $\Fr(\ell)$ acting on $V^{U}$ or $V_{\Ocal}^{U}$ are all non zero. If the action of $U$ on $T$ is trivial, the action on $T_{\Ocal}$ is also trivial and we are done. If monodromy acts non trivially, the quotient of the eigenvalues of $\s$ acting on $V^{U}$ is well-defined and equal to $\ell^{\pm1}$. The eigenvalues of $\s$ on $V_{\Ocal}^{U}$ then have different Weil weights so there is a non-trivial, hence necessarily rank 1, monodromy operator acting on $V_{\Ocal}^{U}$. The rank of $T_{\Ocal}^{U}$ is then at most 1, and so is less than $\rank_{R(\aid)}T^{U}$.
\end{proof}
\begin{CorEnglish}\label{CorWeightMonodromy}
Let $K/\Q$ be a finite extension and let $K_{\infty}$ be a $\zp^{d}$-extension of $K$ with Galois group $\Gamma$. Let $R$ be the completed group algebra $R(\aid)[[\Gamma]]$. Any $G_{K,S}$-representation $(T,\rho,R(\aid))$ gives rise to a representation $(T\tenseur\chi,\rho\tenseur\chi,R)$ with the $G_{K,S}$-action on $R$ given by the character $\chi$ equal to the projection onto $\Gamma$ followed by inclusion in $R$. Let $\lambda:R(\aid)\fleche S$ be a modular specialization. Then $\lambda$ extends as a map of flat $\Ocal$-algebras from $R$ to $S$ by sending $\Gamma$ to $\{1\}$ through the trivial morphism. Let $\phi:R\fleche A$, $\psi:A\fleche B$ and $\pi:B\fleche S$ be morphisms of flat $\Ocal$-algebras between domains making the diagram 
\begin{equation}\nonumber
\xymatrix{
R\ar[rr]^{\lambda}\ar[d]_{\phi}&&S\\
A\ar[rr]^{\psi}&&B\ar[u]^{\pi}
}
\end{equation}
commute. For $x\in\{\phi,\psi\circ\phi\}$, let $T_{x}$ be the representation whose trace is $x(\tr(T\tenseur\chi))$. Then there is a canonical isomorphism
\begin{equation}\nonumber
\Xcali(T_{\phi})\tenseur_{A,\psi}B\overset{\can}{\simeq}\Xcali(T_{\psi\circ\phi}).
\end{equation}
\end{CorEnglish}
This corollary applies in particular to $\phi$ or $\psi$ equal to the identity.
\begin{proof}
This reduces to the existence of a canonical isomorphism
\begin{equation}\nonumber
\Xcali_{\ell}(T_{\phi})\tenseur_{A,\psi}B\overset{\can}{\simeq}\Xcali_{\ell}(T_{\psi\circ\phi})
\end{equation}
for all $\ell\in S-S_{p}$ and hence, by lemma \ref{LemXcaliBienDef}, to the statement that $\rank_{A}T_{\phi}^{I_{\ell}}$ is equal to $\rank_{B}T_{\psi\circ\phi}^{I_{\ell}}$ . As these ranks are both greater than $\rank_{R}(T\tenseur\chi)^{I_{\ell}}$ and smaller than $\rank_{S}T^{I_{\ell}}$, it is enough to prove that $\rank_{R}(T\tenseur\chi)^{I_{\ell}}=\rank_{S}T^{I_{\ell}}$. Because $K_{\infty}/K$ is unramified outside $p$ by \cite[Theorem 1]{IwasawaZl}, $(T\tenseur\chi)^{I_{\ell}}$ is equal to $T^{I_{\ell}}\tenseur\chi$ and so $\rank_{R}(T\tenseur\chi)^{I_{\ell}}$ is equal to $\rank_{R(\aid)}T^{I_{\ell}}$ and thus to $\rank_{S}T^{I_{\ell}}$ by proposition \ref{PropWeightMonodromy}.
\end{proof}
\subsection{Review of the ETNC with coefficients in $\Lambda$}
\subsubsection{$\Lambda$-adic representation}
Let $f\in S_{k}(U_{1}(N))$ be a newform whose eigenvalues are contained in a number field $F$. Fix an integer $1\leq s\leq k-1$ and let $M$ be the motive with coefficients in $F$ equal to the Tate twist $\Wcal(f)(s)$ of the motive $\Wcal(f)$ of subsection \ref{SubMotives}. We denote respectively by $M_{B}$, $M_{dR}$ and $M_{\et,p}$ the Betti, de Rham and $p$-adic étale realizations of $M$. Let $\pid|p$ be a finite place of $F$, let $\Ocal$ be the ring of integers of $F_{\pid}$ and $\Fp$ its residue field. For $\Sigma\supset\{\ell|Np\}$, let $(V(f),\rho_{f},F_{\pid})$ be the $G_{\Q,\Sigma}$-representation given by $M_{\et,p}\tenseur_{\qp}F_{\pid}$. Let $(T(f),\rho_{f},\Ocal)$ be a $G_{\Q,\Sigma}$-stable $\Ocal$-lattice inside $V(f)$ and $(\Tbar(f),\rhobar_{f},\Fp)$ be the residual representation attached to $T(f)$. 

For $m\in\N$, let $\Q_{m}$ be the sub-extension of $\Q(\zeta_{p^{m+1}})/\Q$ with Galois group $G_{m}$ isomorphic to $\Z/p^{m}\Z$. Recall that $\Q_{\infty}/\Q$ is the the unique $\zp$-extension of $\Q$, hence the union of the $\Q_{m}$ for all $m$, and that $\Gamma$ is $\Gal(\Q_{\infty}/\Q)$. Let $\Gamma_{m}$ be $\Gal(\Q_{\infty}/\Q_{m})$. Let $\Lambda$ be the completed group algebra $\Ocal[[\Gamma]]$, a complete regular local ring of dimension 2. The canonical surjection of $G_{\Q,\Sigma}$ onto $\Gal(\Q_{\infty}/\Q)$ followed by injection in $\Lambda\croix$ defines a $G_{\Q,\Sigma}$-representation $(\Lambda,\chi_{\Gamma},\Lambda)$ which we also denote by $\Lambda$ in a slight abuse of notation and which is the universal deformation of the trivial $\Fp$-representation unramified outside $p$. For $R$ a complete local noetherian $\Ocal$-algebra, let $\Riwa$ be $R[[\Gamma]]$. If $(T,\rho,R)$ is a $G_{\Q,\Sigma}$-representation, let $(T_{\Iw},\rho\tenseur\chi_{\Gamma},\Riwa)$ be the $G_{\Q,\Sigma}$-representation $T\tenseur_{R}R[[\Gamma]]$ with $G_{\Q,\Sigma}$-action on both sides of the tensor product. More generally, the $R[G_{m}]$-module $T\tenseur_{R}R[G_{m}]$ is always understood to have an action of $G_{\Q,\Sigma}$ on both sides of the tensor product whenever it is regarded as a $G_{\Q,\Sigma}$-representation. In a slight abuse of notation, we denote by $(V(f)_{\Iw},\rho\tenseur\chi_{\Gamma},\Lambda[1/p])$ the representation $T(f)\tenseur_{\Ocal}\Lambda[1/p]$.

The étale cohomology complex $\RGamma_{\et}(\Z[1/p],T(f)_{\Iw})$ is a complex of finite $\Lambda$-modules, necessarily perfect as $\Lambda$ is a regular local ring, whose cohomology is concentrated in $[0,3]$. Let $S$ be a flat $\Ocal$-algebra and let $\phi:\Lambda\fleche S$ be a morphism of $\Ocal$-algebras. Functoriality of cochain complexes and the fact that $\Lambda$ is unramified outside $p$ imply that there are canonical isomorphisms
\begin{align}\label{EqDescenteLambdaCont}
&\RGamma(G_{\Q_{\ell}},T(f)_{\Iw})\Ltenseur_{\Lambda,\phi}S\simeq\RGamma(G_{\Q_{\ell}},T(f)\tenseur_{\Ocal}S)\textrm{ for all $\ell\nmid\infty$,}\\\nonumber
&\RGamma(G_{\Q,\Sigma},T(f)_{\Iw})\Ltenseur_{\Lambda,\phi}S\simeq\RGamma(G_{\Q,\Sigma},T(f)\tenseur_{\Ocal}S),\\\nonumber
&\RGamma(G_{\Q_{\ell}}/I_{\ell},T(f)_{\Iw}^{I_{\ell}})\Ltenseur_{\Lambda,\phi}S\simeq\RGamma(G_{\Q_{\ell}}/I_{\ell},T(f)^{I_{\ell}})
\end{align}
which together yield a canonical isomorphism
\begin{equation}\nonumber
\RGamma_{\et}(\Z[1/p],T(f)_{\Iw})\Ltenseur_{\Lambda,\phi}S\simeq\RGamma_{\et}(\Z[1/p],T(f)\tenseur_{\Ocal}S)
\end{equation}
of perfect complexes of $S$-modules. In particular, the projection $\Gamma\fleche\Gamma/\Gamma_{m}$ induces a canonical isomorphism
\begin{equation}\nonumber
\RGamma_{\et}(\Z[1/p],T(f)_{\Iw})\Ltenseur_{\Lambda}\Ocal[G_{m}]\simeq\RGamma_{\et}(\Z[1/p],T(f)\tenseur_{\Ocal}\Ocal[G_{m}]).
\end{equation}
for all integer $m\in\N$.
\subsubsection{Review of the construction of Kato's Euler system}\label{SubEulerSystem}
We briefly review the construction and fundamental properties of several elements in the cohomology and $K$-theory of modular curves which were constructed in \cite{KatoEuler}. As we follow closely \cite{KatoEuler}, the reader might find it convenient to keep a copy of this article at hand while reading \ref{SubEulerSystem}, \ref{SubEulerLambda} and \ref{SubZetaLambda}.
\paragraph{Eisenstein Euler systems}First, analytic elements 
\begin{equation}\nonumber
_{c,d}\z_{M,N}(k,r)\in \Mcal_{k}(U(M,L))
\end{equation}
are constructed from Einsenstein series in \cite[Section 4]{KatoEuler} (they are denoted there $_{c,d}\z_{M,N}(k,r,r')$ and our $_{c,d}\z_{M,N}(k,r)$ is $_{c,d}\z_{M,N}(k,r,k-1)$). Here:
\begin{itemize}
\item $1\leq r\leq k-1$ and if $r=k-2$ then $M\geq 2$.
\item $(c,M)=(d,L)=1$.
\end{itemize}
The crucial characteristic property of these elements is that they are the evaluation on $U(M,N)$ of a unique algebraic distribution $\z_{\operatorname{Eis}}(k,r)$ on $M_{2}(\A_{\Q}^{(\infty)})$ with values in $\Mcal_{k}(U(M,N))$ (see \cite{ColmezBSD}). Choose integers $m\geq1$, $M$ and $L$ such that $m|M$, $M|L$ and $N|L$ and consider the morphisms of schemes
\begin{equation}\label{EqCovering}
Y(L)\fleche Y(M,L)\fleche Y_{1}(N)\tenseur\Q(\zeta_{m}).
\end{equation}
In \cite[Section 5.2]{KatoEuler}, elements $_{c,d}\z_{1,N,m}(k,r,\xi,S)$ are defined by taking the images of  $_{c,d}\z_{M,N}(k,r)$ under twisted trace maps from $\Mcal_{k}(U(M,L))$ to $\Mcal_{k}(U_{1}(N))\tenseur_{\Q}\Q(\zeta_{m})$. Here, $S$ denotes the set of primes dividing $L$. Importantly, these elements are independent of the choice of $L$ in \eqref{EqCovering}. As the elements $_{c,d}\z_{1,N,m}(k,r,\xi,S)$ are linear combinations of Eisenstein series, the Rankin-Selberg method of \cite{ShimuraSpecial,ShimuraHilbert} relates them to critical special values of the universal $L$-function of the modular curve; see \cite[Theorem 5.6]{KatoEuler}. The essential property for our purpose is that the $_{c,d}\z_{1,N,m}(k,r,\xi,S)$ are related with special values of the universal $L$-function with Euler factors at the primes in $S$ removed.
\paragraph{$p$-adic Euler systems}On the other hand, $p$-adic elements denoted
\begin{equation}
_{c,d}\z_{p^{n}}^{(p)}(f,k,j,\alpha,\operatorname{prime}(pN))\in\Hun(\Z[1/p,\zeta_{p^{n}}],V(f)(k-2s))
\end{equation}
are constructed from Siegel units in \cite[Section 2 and section 8]{KatoEuler}. Here:
\begin{itemize}
\item $f$ is a newform in $S_{k}(U_{1}(N))$.
\item $(c,d)$ are integers different from $\pm 1$, congruent to 1 modulo $N$ and such that $cd$ is prime to $6p$.
\item $1\leq j\leq k-1$.
\item $\alpha$ belongs to $\sldeux(\Z)$.
\end{itemize}
As in the case of the $_{c,d}\z_{M,N}(k,r)$, these elements are related to evaluations of algebraic distributions but in a much more complex way involving the Chern class map. The elements $_{c,d}\z_{p^{n}}^{(p)}(f,k,j,\alpha,\operatorname{prime}(pN))$ for varying $n$ then form a projective system for corestriction; see \cite[Section 8]{KatoEuler}. 

\paragraph{Relations between analytic and $p$-adic Euler systems}Let $Y$ be $Y_{1}(N)\tenseur\Q(\zeta_{m})$ and let $X$ be the smooth compactification of $Y$. Let $M_{Y,et}$ be the $p$-adic étale cohomology group:
$$M_{Y,et}=\Hun_{\et}(Y\times_{\Q}\Qbar,\Fcal_{k-2}\tenseur_{\Z}\qp)(k-r)$$
As recalled below, there exists a dual exponential map
\begin{equation}\nonumber
\exp^{*}:\Hun(G_{\qp},M_{Y,et})\fleche\Mcal_{k}(X)\tenseur_{\Q}\qp.
\end{equation}
Localizing cohomology at $p$ on th elect-hand side yields a map
\begin{equation}
\exp^{*}_{Y}:\Hun(\Z[1/p,\zeta_{m}],M_{Y,et})\fleche\Mcal_{k}(X_{1}(N))\tenseur_{\Q}\Q(\zeta_{m})\tenseur_{\Q}\qp
\end{equation}
which sends $_{c,d}\z^{(p)}_{1,N,m}(k,r,r',\xi,S)$ to $_{c,d}\z_{1,N,m}(k,r,\xi,S)\in\Mcal_{k}(X_{1}(N))\tenseur_{\Q}\Q(\zeta_{m})$ by \cite[Theorem 9.6]{KatoEuler}. In particular, this image, which is $\qp$-rational by construction, is actually $\Q$-rational.

\subsubsection{Euler systems with coefficients in $\Lambda$}\label{SubEulerLambda}
We review briefly the construction and important properties of a remarkable non-zero $\Lambda[1/p]$-linear morphism
\begin{equation}\label{EqZf}
Z(f):M_{B}\tenseur_{\Z}\Lambda[1/p]\fleche\Hun_{\et}(\Z[1/p],V(f)_{\Iw})
\end{equation}
whose existence is asserted in \cite[Theorem 12.5]{KatoEuler}. 
For suitable choices of $j_{1},j_{2}$ and $\alpha_{1},\alpha_{2}$ as above, a particular basis
$$(\delta_{1},\delta_{2})=(\delta(f,j_{1},\alpha_{1})^{+},\delta(f,j_{2},\alpha_{2})^{-})\in M_{B}^{2}$$
of $M_{B}$ is defined in \cite[Section 4.7 and 13.9]{KatoEuler}. Let
\begin{equation}\nonumber
\gamma=b_{1}\delta(f,j_{1},\alpha_{1})^{+}+b_{2}\delta(f,j_{2},\alpha_{2})^{-}
\end{equation}
be an element of $M_{B}\tenseur_{\Z}\Lambda[1/p]$. The element
\begin{align}
Z(f)(\gamma)=\somme{i=1}{2}\mu_{i}^{-1}b_{i}\left(_{c,d}\z_{p^{n}}^{(p)}(f,k,j,\alpha,\operatorname{prime}(pN))\right)_{n\geq1}^{(-1)^{i}}
\end{align}
is then a linear combination of the $\z_{p^{n}}$ with coefficients in $\Frac(\Lambda)$. The coefficients $\mu_{i}$ involve the inverse in $\Frac(\Lambda)$ of the Euler factors of the dual newform $f^{*}$ at primes $\ell\nmid p$ dividing $N$ (see \cite[Page 229]{KatoEuler} for the precise definition). It is shown in \cite[Section 13.9,13.12]{KatoEuler} that $Z(f)(\gamma)$, which \textit{a priori} depends  on the choices of $c,d,j_{1},j_{2},\alpha_{1},\alpha_{2}$ and has coefficients in $\Frac(\Lambda)$, is independent of all choices and belongs to $\Hun_{\et}(\Z[1/p],V(f)_{\Iw})$. 

We restrict $Z(f)$ to $(M_{B}\tenseur_{\Z}\Lambda[1/p])^{+}$. The source of $Z(f)$ is then a complex of $\Lambda[1/p]$-modules concentrated in degree 0 and its target is the first cohomology group of $\RGamma_{\et}(\Z[1/p],V(f)_{\Iw})$. Lifting the image re-interprets $Z(f)$ as a morphism of complexes
\begin{equation}\label{EqZfComplexes}
Z(f):(M_{B}\tenseur_{\Z}\Lambda[1/p])^{+}[-1]\fleche\RGamma_{\et}(\Z[1/p],V(f)_{\Iw}).
\end{equation} 
Under assumption \ref{HypIrr}, \cite[Theorem 12.5 (4)]{KatoEuler} states that the image of $Z(f)$ actually lies in $\Hun(\Z[1/p],T(f)_{\Iw})$. Hence, there is a $\Lambda$-adic version of 
$Z(f)$:
\begin{equation}\label{EqMorEuler}
Z(f):(M_{B}\tenseur_{\Z}\Lambda)^{+}[-1]\fleche\RGamma_{\et}(\Z[1/p],T(f)_{\Iw}).
\end{equation}
More accurately, the statement about the image of $Z(f)$ is proved in \cite[Theorem 12.5 (4)]{KatoEuler} under the slightly different assumptions that $\sldeux(\zp)$ is included in the image of $\rho_{f}$; as this stronger statement is only used in the proof of \cite[Theorem 12.5 (4)]{KatoEuler} given in \cite[13.14]{KatoEuler} to show that all lattices inside $V(f)$ are isomorphic, assumption \ref{HypIrr} is also sufficient to deduce the result. 

\begin{DefEnglish}\label{DefDeltaIwasawa}
Let $\Delta_{\Lambda[1/p]}(V(f)_{\Iw})$ and $\Delta_{\Lambda}(T(f)_{\Iw})$ be respectively the graded invertible $\Lambda[1/p]$-module $\Det_{\Lambda[1/p]}\Cone Z(f)$ and the graded invertible $\Lambda$-module $\Det_{\Lambda}\Cone Z(f)$ under assumption \ref{HypIrr}. Let $\z(f)$ denote a $\Lambda[1/p]$-basis of $\Delta_{\Lambda[1/p]}(V(f)_{\Iw})$ or a $\Lambda$-basis of $\Delta_{\Lambda}(T(f)_{\Iw})$ under assumption \ref{HypIrr}.
\end{DefEnglish}
The exact definition of $\z(f)$ above involves a specific choice of unit in $\Lambda[1/p]$. Yet, for convenience, we sometimes refer to $\z(f)$ without mentioning explicitly this choice in the following. The exact choice of unit, though unimportant for our purpose, is made in proposition \ref{PropKatoZeta}.
\paragraph{}
Let $Z$ denote the non-zero $\Lambda$-submodule of $\Hun_{\et}(\Z[1/p],T(f)_{\Iw})$ equal to the image of $Z(f)$. As $(M_{B}\tenseur_{\Z}\Lambda)^{+}$ is a free $\Lambda$-module of rank 1, $Z$ is also free of rank 1. By \cite[Theorem 12.4]{KatoEuler}, the $\Lambda$-modules
$$H^{2}_{\operatorname{\et}}(\Z[1/p],T(f)_{\Iw}),\ \Hun_{\et}(\Z[1/p],T(f)_{\Iw})/Z$$ are torsion. The complex $(\Cone Z(f))\tenseur_{\Lambda}\Frac(\Lambda)$ is thus acyclic and there are a canonical isomorphisms
\begin{equation}\label{EqIsoCan}
\Delta_{\Lambda}(T(f)_{\Iw})\tenseur_{\Lambda}\Frac(\Lambda)\overset{\can}{\simeq}\Det_{\Frac(\Lambda)}(0)\overset{\can}{\simeq}\Frac(\Lambda).
\end{equation}
It follows that $\Delta_{\Lambda}(T(f)_{\Iw})\tenseur_{\Lambda}\Frac(\Lambda)$ comes with two specified $\Lambda$-submodules. The first one is the pre-image of $\Lambda\subset\Frac(\Lambda)$, or equivalently of $\Det_{\Lambda}(0)\subset\Det_{\Frac(\Lambda)}(0)$,  under the isomorphisms of \eqref{EqIsoCan}. The second one is $\Delta_{\Lambda}(T(f)_{\Iw})$. Localizing at grade 1 primes and using the structure theorem for modules over discrete valuation rings shows that the image of $\Delta_{\Lambda}(T(f)_{\Iw})$ in $\Frac(\Lambda)$ through the isomorphisms of \eqref{EqIsoCan} is the characteristic ideal 
\begin{equation}\label{EqCaracDet}
\carac^{-1}_{\Lambda}H^{2}_{\et}(\Z[1/p],T(f)_{\Iw})\tenseur_{\Lambda}\carac_{\Lambda}\Hun_{\et}(\Z[1/p],T(f)_{\Iw})/Z.
\end{equation}
\subsubsection{Zeta elements for $M\times_{\Q}\Q_{m}$}\label{SubZetaLambda}
Let $N$ be a motive over $\Q$. As recalled in the introduction, the ETNC at $p$ with coefficients in $\Ocal[G_{m}]$ or $\Lambda$ of \cite[Conjecture 4.9]{KatoHodgeIwasawa} and \cite[Conjecture 3.2.2]{KatoViaBdR} are far-reaching conjectures predicting the existence for all $m$ of specific $\zp[G_{m}]$-bases of $\Det_{\qp[G_{m}]}\RGamma(\Z[1/p],(N\times_{\Q}\Q_{m})_{\et,p})$, called zeta elements, which are intimately linked with the $p$-adic valuations of the special values of the $L$-function of $N\times_{\Q}\Q_{m}$ together with its natural action of $G_{m}$ as well as a universal zeta element, that is to say a $\Lambda$-basis of $\Det_{\Lambda[1/p]}\RGamma(\Z[1/p],(N\times_{\Q}\Q_{\infty})_{\et,p})$ interpolating the zeta elements for finite $m$.

In that degree of generality, the existence of most of the objects necessary to even state the conjecture is itself already conjectural. In the case of the motive $M\times_{\Q}\Q_{m}$, however, all the necessary objects are known to exist unconditionally. Nevertheless, even in that case, there is an inherent tension in the presentation of the material, as the logical order of exposition is quite different from the logical order of proof of the known results. Experts will know, for instance, that the precise definition of the zeta elements requires first the construction of families of almost zeta elements, then showing that they form Euler systems, then using the method of Euler systems to show the finiteness of some cohomology groups and only then exploiting this extra knowledge to exactly pinpoint the actual zeta elements. In the following, we proceed as if all theorems were known to hold simultaneously, so that the bibliographic references we give are strictly speakin logically incoherent, and explain in which sense the element $\z(f)$ of the previous subsection is compatible with the statement of the ETNC for the motive $M\times_{\Q}\Q_{m}$ for all $m\in\N$ when $s\neq k/2$ and for almost all $m$ when $s=k/2$.  

Recall that $f\in S_{k}(U_{1}(N))$ and that $f^{*}$ is the eigenform whose eigenvalues are the complex conjugate of those of $f$. Let $\epsi$ be the finite order character of $(\Z/N\Z)\croix$ such that $\diamant{a}f=\epsi(a)f$. Let $m\geq1$ be an integer. For $\s\in\Gal(\Q_{m}/\Q)$, let $\Pcal_{\s}$ be the set of rational primes $\ell\nmid p$ such that 
\begin{equation}\nonumber
\left(\frac{\Q_{m}/\Q}{\ell}\right)=\s
\end{equation}
where $\left(\frac{\Q_{m}/\Q}{\cdot}\right)$ is the Artin reciprocity map. The $\s$-partial $L$-value of $f^{*}$ is the evaluation at $s$ of the meromorphic continuation to $\C$ of the Euler product
\begin{equation}\nonumber
L^{G_{m}}_{\{p\}}(f^{*},\s,z)=\produit{\ell\in\Pcal_{\s}}{}\frac{1}{1-\bar{a}_{\ell}(f)\ell^{-z}+\bar{\epsi}(\ell)\ell^{1-2z}}
\end{equation}
and the $G_{m}$-equivariant $L$-value $L_{\{p\}}^{G_{m}}(f^{*},s)$ is the sum.
\begin{equation}\nonumber
L_{\{p\}}^{G_{m}}(f^{*},\cdot)=\left(\frac{1}{2\pi i}\right)^{s+1-k}\somme{\s\in G_{m}}{}L^{G_{m}}_{\{p\}}(f^{*},\s,s)\s\in\C[G_{m}].
\end{equation}
It is the unique element of $\C[G_{m}]$ such that, for all $\chi\in\widehat{G}_{m}$
\begin{equation}\nonumber
\chi(L_{\{p\}}^{G_{m}}(f^{*},s))=\left(\frac{1}{2\pi i}\right)^{s+1-k}L_{\{p\}}(f^{*},\chi,s).
\end{equation} 
 According to \cite[Theorem]{JacquetShalika} (resp. \cite[Theorem I]{RohrlichNonVanishing}), the element $L_{\{p\}}^{G_{m}}(f^{*},s)$ is non-zero if $s\neq k/2$ (resp. is non-zero if $s=k/2$ except possibly for a finite number of $m$). For $m$ outside the finite or empty set such that $L_{\{p\}}^{G_{m}}(f,s)$ vanishes, denote by $\Delta_{F_{\pid}[G_{m}]}(V(f)_{m})$ the $F_{\pid}[G_{m}]$-module $\Delta_{\Lambda[1/p]}(V(f)_{\Iw})\tenseur_{\Lambda[1/p]}F_{\pid}[G_{m}]$. Let $M_{m}^{+}$ be the Betti cohomology group $(M_{B}\tenseur_{\Z}F_{\pid}[G_{m}])^{+}$. There is then a canonical isomorphism of $F_{\pid}[G_{m}]$-modules
\begin{equation}\label{EqDeltaComplexes}
\xymatrix{
\Delta_{F_{\pid}[G_{m}]}(V(f)_{m})\ar[d]^{\isocan}
\\
\Det_{F_{\pid}[G_{m}]}\RGamma_{\et}(\Z[1/p],V(f)\tenseur_{F_{\pid}}F_{\pid}[G_{m}])\tenseur_{F_{\pid}[G]}\left(\Det_{F_{\pid}[G_{m}]}M_{m}^{+}\right).
} 
\end{equation}
Let $\z(f,G_{m})$
be the image of $\z(f)$ inside $\Delta_{F_{\pid}[G_{m}]}(V(f)_{m})$. According to \cite[Theorem 12.5 (1)]{KatoEuler}, the element $\z(f,G_{m})$ is non-zero. Moreover, by \cite[Theorem 14.5 (1)]{KatoEuler},  the cohomology of $\RGamma_{\et}(\Z[1/p],V(f)\tenseur_{F_{\pid}}F_{\pid}[G_{m}])$ is then concentrated in degree 1 and $H^{1}_{\et}(\Z[1/p],V(f)\tenseur_{F_{\pid}}F_{\pid}[G_{m}])$ is of rank 1 over $F_{\pid}[G_{m}]$. Consequently:
\begin{equation}\nonumber
\Det_{F_{\pid}[G_{m}]}\RGamma_{\et}(\Z[1/p],V(f)\tenseur_{F_{\pid}}F_{\pid}[G_{m}])=\Det^{-1}_{F_{\pid}[G_{m}]}\Hun_{\et}(\Z[1/p],V(f)\tenseur_{F_{\pid}}F_{\pid}[G_{m}])
\end{equation}
Composing this isomorphism with localization at $p$ 
\begin{equation}\nonumber
\Hun_{\et}(\Z[1/p],V(f)\tenseur_{F_{\pid}}F_{\pid}[G_{m}])\fleche\Hun(G_{\qp},V(f)\tenseur_{F_{\pid}}F_{\pid}[G_{m}])
\end{equation}
and by the natural map to $\Hun(G_{\qp(\zeta_{p^{m}})},V(f))$ given by Shapiro's lemma yields a canonical isomorphism
\begin{equation}\label{EqDetShapiro}
\Det_{F_{\pid}[G_{m}]}\RGamma_{\et}(\Z[1/p],V(f)\tenseur_{F_{\pid}}F_{\pid}[G_{m}])\isom\Det^{-1}_{F_{\pid}[G_{m}]}\Hun(G_{\qp(\zeta_{p^{m}})},V(f)).
\end{equation}
For $K$ a finite extension of $\qp$ and $V$ a $p$-adic representation of $G_{K}$, let $D_{\dR}^{0}(V)$ be $H^{0}(G_{K},B_{\dR}^{0}\tenseur_{\qp}V)$. There exists a canonical map
\applicationsimplelabel[EqExpDuale]{\exp^{*}}{\Hun(G_{K},V)}{D_{\dR}^{0}(V)}
called the dual exponential map from $\Hun(G_{K},V)$ to $D_{\dR}^{0}(V)$. When $K=\qp(\zeta_{p^{m}})$ and $V=V(f)$, the dual exponential map yields a map from $\Hun(G_{\qp(\zeta_{p^{m}})},V(f))$ to 
\begin{equation}\nonumber
D_{\dR}^{0}(V(f))=D_{\dR}^{s}(M(f))=S_{k}(U_{1}(N))(f)\tenseur_{F}F_{\pid}[G_{m}].
\end{equation}
See for instance \cite[Section 11]{KatoEuler} for the last equality above. Composing \eqref{EqDetShapiro} with \eqref{EqExpDuale} thus yields a map
\begin{equation}\label{EqDetCusp}
\Det_{F_{\pid}[G_{m}]}\RGamma_{\et}(\Z[1/p],V(f)\tenseur_{F_{\pid}}F_{\pid}[G_{m}])\fleche (\Det^{-1}_{F}S_{k}(U_{1}(N))(f))\tenseur_{F}F_{\pid}[G_{m}].
\end{equation}
According to \cite[Theorem 12.5]{KatoEuler}, the image of $\z(f,G_{m})$ in 
\begin{equation}\nonumber
(\Det^{-1}_{F}S_{k}(U_{1}(N))(f))\tenseur_{F}F_{\pid}[G_{m}]\tenseur \left(\Det_{F_{\pid}[G_{m}]}M^{+}_{m}\right)
\end{equation}
through \eqref{EqDetCusp} actually belongs to the $\Q$-rational subspace
\begin{equation}\nonumber
\left(\Det^{-1}_{\Q[G_{m}]}S_{k}(U_{1}(N))(f)\tenseur_{\Q}\Q[G_{m}]\right)\tenseur\left(\Det_{\Q[G_{m}]}(M_{B}\tenseur_{\Z}\Q[G_{m}])^{+}\right).
\end{equation}
This fundamental rationality property is the algebraic equivalent of the $\Q_{m}$-equivariant rationality of special values of $L$-functions as in \cite{DeligneFonctionsL}. In this setting, it is a consequence of the rationality property noted at the end of section \ref{SubEulerSystem}. There is a canonical isomorphism of $\C[G_{m}]$-modules
\applicationsimple{\per_{\C[G_{m}]}}{S_{k}(X_{1}(N))(f)\tenseur_{\Q}\C[G_{m}]}{(\Hun(X_{1}(N)(\C),\Fcal_{k-2})(f)\tenseur_{\Z}\C[G_{m}])^{+}}
as well as an isomorphism of $\C$-vector spaces
\begin{equation}\label{EqBetti}
[(2\pi i)^{s}\Hun(X_{1}(N)(\C),\Fcal_{k-2})(f)]^{+}\simeq M_{B}^{+}.
\end{equation}
Composing \eqref{EqDetCusp} with tensor product with $\C$, the isomorphism \eqref{EqBetti} and finally with the period map thus yields maps
\begin{equation}\label{EqPeriodMap}
\xymatrix{
&\Z[G_{m}]\z(f,G_{m})\ar[d]&\\
&\left(\Det^{-1}_{\Q[G_{m}]}S_{k}(U_{1}(N))(f)\tenseur_{\Q}\Q[G_{m}]\right)\tenseur\Det_{\Q[G_{m}]}(M_{B}\tenseur_{\Z}\Q[G_{m}])^{+}\ar[d]^{-\tenseur_{\Q}\C}&\\
&\left(\Det^{-1}_{\C[G_{m}]}S_{k}(U_{1}(N))(f)\tenseur_{\Q}\C[G_{m}]\right)\tenseur\Det_{\C[G_{m}]}(M_{B}\tenseur_{\Z}\C[G_{m}])^{+}\ar[d]^{\per_{\C[G_{m}]}}&\\
&\C[G_{m}]&
}
\end{equation}
and thus defines a $\Z[G_{m}]$-lattice inside $\C[G_{m}]$. The basis $\z(f)$ is characterized by the following fundamental property.
\begin{Prop}\label{PropKatoZeta}
There exists a choice of unit of $\Lambda$ and a corresponding choice of $\z(f)$ in definition \ref{DefDeltaIwasawa} such that the image of $\z(f)$ inside $\C[G_{m}]$ through \eqref{EqDeltaComplexes}, \eqref{EqDetCusp} and \eqref{EqPeriodMap} is the $G_{m}$-equivariant special $L$-value $L_{\{p\}}^{G_{m}}(f^{*},s)$.
\end{Prop}
\begin{proof}
This is \cite[Theorem 12.5 (1)]{KatoEuler}. See especially \cite[Section 13.12]{KatoEuler} for the proof.
\end{proof}
As we remarked already, the full strength of proposition \ref{PropKatoZeta} is not used in this article and it would have been enough for our purpose to choose $\z(f)$ up to a unit in $\Lambda$. Philosophically speaking, this stems from the fact that we are only interested in the ETNC at $p$, whereas the full ETNC actually predicts the existence of motivic zeta elements whose image in the $\ell$-adic étale cohomology realization provide $\ell$-adic zeta elements for all primes $\ell$.

\subsubsection{The ETNC for $M\times_{\Q}\Q_{\infty}$ at $p$} 
In the previous two subsections, we have seen that there are two canonical $\Lambda$-lattices inside $\Delta_{\Lambda}(T(f)_{\Iw})\tenseur_{\Lambda}\Frac(\Lambda)$: the lattice $\Det_{\Lambda}(0)$ coming from functoriality of determinants and the lattice $\Delta_{\Lambda}(T(f)_{\Iw})$ which is characterized (though not defined) as the pre-image of the special values of the $L$-function of the dual of $M$. One possible formulation of the ETNC for $M\times_{\Q}\Q_{\infty}$ at $p$ is then that these two lattices coincide.
\begin{Conj}\label{ConjXcaliLambda}
There is an identity of $\Lambda$-lattices
\begin{align}\label{EqDeltaLambda}
\Delta_{\Lambda}(T(f)_{\Iw})=\Det_{\Lambda}(0)
\end{align}
inside $\Delta_{\Lambda}(T(f)_{\Iw})\tenseur_{\Lambda}\Frac(\Lambda)\overset{\can}{\simeq}\Frac(\Lambda)$.\end{Conj}
Equivalently, the two natural $\Lambda$-bases of $\Delta_{\Lambda}(T(f)_{\Iw})\tenseur_{\Lambda}\Frac(\Lambda)$ described in sub-section \ref{SubEulerLambda} coincide.
\begin{DefEnglish}\label{DefDelta}
Let $\phi:\Lambda\fleche S$ be a local morphism from $\Lambda$ to one of its $\Ocal$-flat quotient such that the image of $Z(f)$ under $\phi_{*}$ is non-zero. Define $\Delta_{S}(T(f)_{\Iw}\tenseur_{\Lambda}S)$ to be the graded invertible $S$-module $\Det_{S}\Cone Z(f)$ where
\begin{equation}\nonumber
Z(f):(M_{B}\tenseur_{\Z}S)^{+}[-1]\fleche\RGamma_{\et}(\Z[1/p],T(f)\tenseur_{\Ocal}S)
\end{equation}
is viewed as a morphism of $S$-modules. Define $\Delta_{S[1/p]}(V(f)_{\Iw}\tenseur_{\Lambda}S)$ to be the graded invertible $S[1/p]$-module $\Det_{S}\Cone Z(f)$ where
\begin{equation}\nonumber
Z(f):(M_{B}\tenseur_{\Z}S[1/p])^{+}[-1]\fleche\RGamma_{\et}(\Z[1/p],V(f)\tenseur_{\Ocal}S)
\end{equation}
is viewed as a morphism of $S[1/p]$-modules. 
%
\end{DefEnglish}
Just as the equality \eqref{EqDeltaLambda} is a possible formulation for the ETNC for $M\times_{\Q}\Q_{\infty}$, a possible formulation of the ETNC for $M\times_{\Q}\Q_{n}$ at $p$ is as follows.
\begin{Conj}\label{ConjXcaliGm}
Assume $s\neq k/2$. For all $m\geq1$, there is an identity of $\Ocal[G_{m}]$-lattices
\begin{equation}\label{EqDeltaGm}
\Delta_{\Ocal[G_{m}]}(T(f)\tenseur_{\Ocal}\Ocal[G_{m}])=\Det_{\Ocal[G_{m}]}(0)
\end{equation}
inside $\Delta_{F_{\lambda}[G_{m}]}(V(f)_{m})$. If $s=k/2$, then the identity \eqref{EqDeltaGm} is true for all $m\geq1$ except possibly finitely many. More generally, there is an identity of $S$-lattices
\begin{equation}\label{EqDeltaS}
\Delta_{S}(T(f)\tenseur_{\Ocal}S)=\Det_{S}(0)
\end{equation}
inside $\Delta_{S[1/p]}(V(f)\tenseur_{F_{\lambda}}S[1/p])$ for all morphisms $\phi:\Lambda\fleche S$ as in definition \ref{DefDelta}. 
\end{Conj}
For clarity of reference, we note that conjecture \ref{ConjXcaliGm} for $m\geq1$ is equivalent to the $p$-part of \cite[Conjecture 4.9]{KatoHodgeIwasawa} for the motive $M$ and the abelian Galois extension $\Q_{m}/\Q$. Conjecture \ref{ConjXcaliLambda}  (resp. \ref{ConjXcaliGm}) is equivalent to \cite[Conjecture 3.2.2 part (v)]{KatoViaBdR} for the étale sheaf of perfect complexes of $\Lambda$-modules $T(f)_{\Iw}$ (resp. of $S$-modules $T(f)\tenseur_{\Ocal}S$) on $\Spec\Z[1/p]$. The computation of equation \eqref{EqCaracDet} also shows that conjecture \ref{ConjXcaliLambda} is equivalent to \cite[Conjecture 12.10]{KatoEuler}.

By descent as in \eqref{EqDescenteLambdaCont}, conjecture \ref{ConjXcaliLambda} is seen to imply conjecture \ref{ConjXcaliGm}.
\begin{Prop}
Assume conjecture \ref{ConjXcaliLambda}. Then conjecture \ref{ConjXcaliGm} is true for all morphism $\phi:\Lambda\fleche S$ as in definition \ref{DefDelta}.
\end{Prop}
\begin{proof}
Let $\phi:\Lambda\fleche S$ be a morphism as in definition \ref{DefDelta}. According to equation \eqref{EqDescenteLambdaCont}, there is an equality of $S$-lattices 
\begin{equation}\nonumber
\Delta_{\Lambda}(T(f)_{\Iw})\tenseur_{\Lambda,\phi}S=\Delta_{S}(T(f)_{\Iw}\tenseur_{\Lambda,\phi}S)
\end{equation}
inside $\Delta_{S[1/p]}(V(f)_{\Iw}\tenseur_{\Lambda[1/p],\phi}S[1/p])$. The equality \eqref{EqDeltaLambda} thus implies that
\begin{equation}\nonumber
\Delta_{S}(T(f)_{\Iw}\tenseur_{\Lambda,\phi}S)=\Det_{\Lambda}(0)\tenseur_{\Lambda}S=\Det_{S}(0)
\end{equation}
and hence the statement \eqref{EqDeltaS}.
\end{proof}
\subsection{The ETNC with coefficients in $\Hecke^{\new}$ and $\Hecke^{\red}$}
We keep the notational convention that $f$ is an eigencuspform in $S_{k}(U_{1}(N))$. Henceforth, the representation $\rhobar_{f}$ is assumed to satisfy assumption \ref{HypIrrTW} and either assumption \ref{HypNO} or assumption \ref{HypFlat}. If $f$ is $p$-ordinary then $\rhobar_{f}$ satisfies assumption \ref{HypNO} and all Hecke algebras written below are assumed to contain the operator $T(p)$. Let $\Sigma\supset\{\ell|Np\}$ be a finite set of primes and let $N(\Sigma)$ be the integer
\begin{equation}\nonumber
N(\Sigma)=N(\rhobar_{f})\produit{\ell\in\Sigma^{p}}{}\ell^{\dim_{k}(\rhobar_{f})_{I_{\ell}}}
\end{equation}
as in sub-section \ref{SubGalois}.

To $f$ is attached a unique maximal ideal $\mgot_{f}$ of $\Hecke^{\red}(N(\Sigma))$. Let $\aid^{\red}$ be a minimal prime ideal of $\Hecke^{\red}(N(\Sigma))_{\mgot_{f}}$ such that $\lambda_{f}$ factors through $\Hecke^{\red}(N(\Sigma))_{\mgot_{f}}/\aid^{\red}$. Because an eigenform is a newform for some unique level, there is an injective morphism
\begin{equation}\nonumber
\Hecke^{\red}(N(\Sigma))_{\mgot_{f}}[1/p]\plonge\produit{M|N(\Sigma)}{}\Hecke^{\new}(M)[1/p].
\end{equation}
Hence, to $\aid^{\red}$ is attached a unique $M|N(\Sigma)$ and a unique minimal ideal of $\Hecke^{\new}(M)$ such that $\aid^{\red}$ is the image of $\aid\in\Spec\Hecke^{\new}(M)$. Thus, if $f$ is new of level $U$ and if $\lambda_{f}$ factors through $R(\aid)=\Hecke^{\new}(U)/\aid$ for $\aid$ a minimal prime ideal of $\Hecke^{\new}(U)$, then there is a map from $\Hecke^{\red}(N(\Sigma))_{\mgot_{f}}$ to $R(\aid)$ which factors through an injective map from $\Hecke^{\red}(N(\Sigma))_{\mgot_{f}}/\aid^{\red}$ to $R(\aid)$.

By \cite[Proposition 2.15]{WilesFermat} (and its proof), there is a unique maximal ideal $\mgot$ of $\Hecke(N(\Sigma))$ such that $R_{\Sigma}=\Hecke(N(\Sigma))_{\mgot}$ is isomorphic to $\Hecke^{\red}(N(\Sigma))_{\mgot_{f}}$. There is thus a morphism
\begin{equation}\label{EqMorSigma}
\phi(\aid):R_{\Sigma}\fleche R(\aid)
\end{equation}
of local $\Ocal$-algebras obtained as the composition
\begin{equation}\nonumber
R_{\Sigma}=\Hecke(N(\Sigma))_{\mgot}\simeq \Hecke^{\red}(N(\Sigma))_{\mgot_{f}}\surjection\Hecke^{\red}(N(\Sigma))_{\mgot_{f}}/\aid^{\red}\plonge \Hecke^{\new}(M)/\aid=R(\aid).
\end{equation}
Theorem \ref{TheoTaylorWiles} implies that $R_{\Sigma}$ is the universal deformation ring parametrizing nearly ordinary or flat deformations of $\rhobar_{f}$ with trivial type which are unramified outside $\Sigma$ and whose determinant is of weight $k-1$. Moreover, $\Hun_{\et}(X_{1}(N(\Sigma))(\C),\Fcal_{k-2}\tenseur_{\Z}\zp)_{\mgot_{f}}$ is a free $R_{\Sigma}$-module of rank 2 and $\Hun_{\et}(X_{1}(M)(\C),\Fcal_{k-2}\tenseur_{\Z}\zp)_{\mgot_{f}}$ is a free $R(\aid)$-module of rank 2. Denote by $(T_{\Sigma},\rho_{\Sigma},R_{\Sigma})$ and $(T(\aid),\rho(\aid),R(\aid))$ the corresponding $G_{\Q,\Sigma}$-deformations of $\rhobar_{f}$.

\subsubsection{The ETNC with coefficients in $\Hecke^{\new}$}
Fix a minimal prime  $\aid$ of $\Hecke^{\new}(U)$ through which $\lambda_{f}$ factors and let $N(\aid)$ be the level of modular points factoring through $R(\aid)$. Let $\Kcal(\aid)$ be the fraction field of $R(\aid)_{\Iw}$ and let $\Vcal(\aid)_{\Iw}$ be the $G_{\Q}$-representation $T(\aid)_{\Iw}\tenseur_{R(\aid)_{\Iw}}\Kcal(\aid)$.

As recalled in section \ref{SubEulerLambda}, the element $Z(f)(\gamma)$ is a linear combination of
$$_{c,d}\z_{p^{n}}^{(p)}(f,k,j,\alpha_{i},\operatorname{prime}(pN))$$
with coefficients involving the inverse of the Euler factors of the dual newform $f^{*}$.  According to \cite[Section 5]{KatoEuler} and to \cite[Proposition 8.10, Theorem 9.5]{KatoEuler}, the classes $_{c,d}\z_{p^{n}}^{(p)}(f,k,j,\alpha,\operatorname{prime}(pN))$ are the images of classes $_{c,d}\z_{1,N,m}^{(p)}(k,r,r',\xi,S)$ with coefficients in $\Hecke^{\red}$ through the projection to $\Hecke^{\new}$ composed with $\lambda_{f}$. Hence, mimicking the proof given in \cite[Section 13.9]{KatoEuler} with $\bar{a}_{\ell}$ replaced everywhere by $T(\ell)$ (the seemingly extraneous complex conjugation comes from the fact that $T(\ell)\in\Hecke^{\new}$ is in that context acting on $f^{*}$), we obtain an  $R(\aid)_{\Iw}$-linear morphism
\begin{equation}\label{EqZfAid}
Z(\aid):M_{B}\tenseur_{\Z}\Lambda\fleche\Hun_{\et}(\Z[1/p],T(\aid)_{\Iw})
\end{equation}
which we view as a morphism of complexes of $R(\aid)$-modules
\begin{equation}\label{EqZfComplexesAid}
Z(\aid):(M_{B}\tenseur_{\Z}\Lambda)^{+}[-1]\fleche\RGamma_{\et}(\Z[1/p],T(\aid)_{\Iw}).
\end{equation}
The same construction can also be performed with $p$ inverted. Denote by $\image Z(\aid)$ the image of $Z(\aid)$ inside $\Hun_{\et}(\Z[1/p],T(\aid)_{\Iw})$. Then $Z(\aid)$ is non-zero and hence a free $R(\aid)$-module of rank 1.
\begin{DefEnglish}\label{DefDeltaAid}
Let $\Delta_{R(\aid)_{\Iw}}(T(\aid)_{\Iw})$ be the graded invertible $R(\aid)_{\Iw}$-module
$$\Xcali(T(\aid)_{\Iw})^{-1}\tenseur\Det_{R(\aid)_{\Iw}}\image Z(\aid)$$
where $\image Z(\aid)$ is the sub-module generated by the image of $Z(\aid)$ inside $\Hun_{\et}(\Z[1/p],T(\aid)_{\Iw})$. Let $\Delta_{\Kcal(\aid)}(\Vcal(\aid)_{\Iw})$ be the graded invertible $\Kcal(\aid)$-module $\Delta_{R(\aid)_{\Iw}}(T(\aid)_{\Iw})\tenseur_{R(\aid)_{\Iw}}\Kcal(\aid)$.
\end{DefEnglish}
As $R(\aid)[1/p]$ is finite étale over $\qp$, the \Nekovar-Selmer complex $\RGamma_{\et}(\Z[1/p],V(\aid)_{\Iw})$ is a perfect complex of $K(\aid)$-modules. Hence, so is $\Cone Z(\aid)\tenseur1$. After inverting $p$, $\Delta_{R(\aid)_{\Iw}}(T(\aid)_{\Iw})$ thus becomes canonically isomorphic to the determinant of the cone of a morphism of complexes towards the \Nekovar-Selmer complex of a Galois representation. However, $\Delta_{R(\aid)_{\Iw}}(T(\aid)_{\Iw})$ itself does not obviously arise in the same way; which explains the resort to the the set-up of section \ref{SubLattices}. 

By construction, $\Delta_{R(\aid)_{\Iw}}(T(\aid)_{\Iw})$ comes with a canonical $R(\aid)_{\Iw}$-basis $\z(\aid)$ which is sent to $\z(g)\in\Delta_{\Lambda}(T(g)_{\Iw})$ for all eigenforms $g$ such that $\lambda_{g}$ factors through $R(\aid)$. Beside, $\Det_{R(\aid)_{\Iw}}(T(\aid)_{\Iw})\tenseur_{R(\aid)_{\Iw}}\Kcal(\aid)_{\Iw}$ is canonically isomorphic to $\Det_{\Kcal(\aid)_{\Iw}}\Cone (Z(\aid)\tenseur1)$ which is an acyclic complex and hence canonically isomorphic to $\Kcal(\aid)_{\Iw}$. Hence, there is a second canonical $R(\aid)_{\Iw}$-basis in $\Det_{R(\aid)_{\Iw}}(T(\aid)_{\Iw})\tenseur_{R(\aid)_{\Iw}}\Kcal(\aid)_{\Iw}$ given by the pre-image of $R(\aid)_{\Iw}\subset\Kcal(\aid)_{\Iw}$ through the isomorphisms above. This suggests the following conjecture.
\begin{Conj}\label{ConjXcaliRaid}
There is an identity of $R(\aid)_{\Iw}$-lattices
\begin{align}\label{EqDeltaRaid}
\Delta_{R(\aid)_{\Iw}}(T(\aid)_{\Iw})=\Det_{R(\aid)_{\Iw}}(0)
\end{align}
inside $\Delta_{R(\aid)_{\Iw}}(T(\aid)_{\Iw})\tenseur_{R(\aid)_{\Iw}}\Kcal(\aid)_{\Iw}\overset{\can}{\simeq}\Kcal(\aid)_{\Iw}$.
\end{Conj}
Conjecture \ref{ConjXcaliRaid} is compatible with modular specializations of $R(\aid)$ in the sense of the following proposition.
\begin{Prop}\label{PropCompatibleSpec}
Let $\lambda_{g}$ be a modular specialization of $R(\aid)_{\Iw}$ and let $\phi:\Lambda\fleche S$ be a morphism as in definition \ref{DefDelta}. Assume conjecture \ref{ConjXcaliRaid}. Then there is an identity of $S$-lattices
\begin{equation}\nonumber
\Delta_{S}(T(g)\tenseur_{\Ocal}S)=\Det_{S}(0)
\end{equation}
inside $\Delta_{S[1/p]}(V(g)\tenseur_{F_{\lambda}}S[1/p])$.
\end{Prop}
\begin{proof}
It is enough to show that $\Delta_{R(\aid)_{\Iw}}(T(\aid)_{\Iw})\tenseur_{R(\aid)_{\Iw},\lambda_{g},\phi}S$ is equal to $\Delta_{S}(T(g)\tenseur_{\Ocal}S)$ and that $\Det_{R(\aid)_{\Iw}}(0)\tenseur_{R(\aid)_{\Iw},\lambda_{g},\phi}S$ is equal to $\Det_{S}(0)$. The latter assertion is part of the functoriality properties of the determinant functor, so we show the first. Because the morphism $Z(g)$ is by construction a specialization of the morphism $Z(\aid)$, it is enough to show the equalities
\begin{equation}\label{EqBettiRaidS}
\left(\Det_{R(\aid)_{\Iw}}M_{B}\tenseur_{\Z}\Lambda\right)\tenseur_{R(\aid)_{\Iw},\lambda_{g},\phi}S=\Det_{S}\left(M_{B}\tenseur_{\Z}S\right)
\end{equation}
and
\begin{equation}\label{EqXcaliRaidS}
\Xcali(T(\aid)_{\Iw})\tenseur_{R(\aid)_{\Iw},\lambda_{g},\phi}S=\Xcali(T(g)\tenseur_{\Ocal}S).
\end{equation}
The equality \eqref{EqBettiRaidS} holds by definition of $\phi$. To prove \eqref{EqXcaliRaidS}, it is enough to show the comparable statement
\begin{equation}
\Xcali_{\ell}(T(\aid)_{\Iw})\tenseur_{R(\aid)_{\Iw},\lambda_{g},\phi}S=\Xcali_{\ell}(T(g)\tenseur_{\Ocal}S)
\end{equation}
for all finite $\ell\nmid p$ dividing $N(\aid)$. This holds by corollary \ref{CorWeightMonodromy}.
\end{proof}
\subsubsection{The ETNC with coefficients in $\Hecke^{\red}$}
Denote by $M_{\aid}$ the $R(\aid)_{\Iw}$-module
$$M_{\aid}=\Hun_{\et}(X(N(\aid))(\C),\Fcal_{k-2}\tenseur_{\Z}\zp)\widehat{\tenseur}_{\Hecke}R(\aid)_{\Iw}$$
and by $M_{\Sigma}$ the $R_{\Sigma,\Iw}$-module
$$M_{\Sigma}=\Hun_{\et}(X(N(\Sigma))(\C),\Fcal_{k-2}\tenseur_{\Z}\zp)\widehat{\tenseur}_{\Hecke}R_{\Sigma,\Iw}.$$
We wish to relate $M_{\Sigma}\tenseur_{R_{\Sigma,\Iw}}R(\aid)_{\Iw}$ with $M_{\aid}$ by way of a well-chosen cohomological level-lowering map $\pi_{\Sigma,\aid}$. For $d_{2}|d_{1}|N|N(\Sigma)$, there is a geometric degeneracy map
\begin{eqnarray}\label{EqProjD}%
\pi_{N,d_{1},d_{2}}:&X_{1}(N)&\fleche X_{1}(N/d_{1})\\
\nonumber%
&[z,g]&\longmapsto [z,g\matrice{1}{0}{0}{d_{2}}]%
\end{eqnarray}
between modular curves. The map $\pi$ is constructed from the cohomological realizations of these maps for rational primes $\ell\nmid p$ dividing $N(\Sigma)/N(\aid)$. Let $\ell^{e_{\ell}}$ be a power of a rational prime dividing $N_{\Sigma}/N(\aid)$. For $N(\aid)\ell^{e_{\ell}}|N|N_{\Sigma}$, denote by $\pi_{N,\ell}$
the map
\begin{equation}\nonumber
\pi_{N,\ell}=\begin{cases}
1&\textrm{ if $e_{\ell}=0$}\\
\pi_{N,\ell,1*}-\ell^{-s}T(\ell)\pi_{N,\ell,\ell*}&\textrm{ if $e_{\ell}=1$}\\
\pi_{N,\ell^{2},1*}-\ell^{-s}T(\ell)\pi_{N,\ell^{2},\ell*}+\ell^{-s-2}\diamant{\ell}\pi_{\ell^{2},\ell^{2}*}&\textrm{ if $e_{\ell}=2$}
\end{cases}
\end{equation}
from $\Hun(X_{1}(N)(\C),\Fcal_{k-2}\tenseur_{\Z}\zp)$ to $\Hun(X_{1}(N/\ell^{e_{\ell}})(\C),\Fcal_{k-2}\tenseur_{\Z}\zp)$. For $(N_{i})_{1\leq i\leq n}$ a list of integers satisfying $N_{1}=N(\aid)$, $N_{n}=N(\Sigma)$ and such that for all $1\leq i\leq n$ there exists a prime $\ell_{i+1}\nmid p$ such that $N_{i+1}/N_{i}=\ell_{i+1}^{e_{\ell_{i+1}}}$, we denote by $\pi_{i}$ the map 
\applicationsimple{\pi_{N_{i+1},\ell_{i+1}}}{\Hun(X_{1}(N_{i+1})(\C),\Fcal_{k-2}\tenseur_{\Z}\zp)}{\Hun(X_{1}(N_{i})(\C),\Fcal_{k-2}\tenseur_{\Z}\zp)}
and by $\pi_{\Sigma,\aid}$ the composition
\begin{equation}\label{EqPiSigma}
\pi_{\Sigma,\aid}=(\pi_{1}\circ\cdots\circ\pi_{n-1})
\end{equation}
from $\Hun(X_{1}(N(\Sigma))(\C),\Fcal_{k-2}\tenseur_{\Z}\zp)$ to $\Hun(X_{1}(N(\aid))(\C),\Fcal_{k-2}\tenseur_{\Z}\zp)$.
\begin{DefEnglish}
Define
\begin{equation}
\Eul_{\ell}(\aid)=\det(\Id-\Fr(\ell)|\Vcal(\aid)^{I_{\ell}}),\ \Eul_{\Sigma}(\aid)=\produit{\ell\in\Sigma^{p}}{}\Eul_{\ell}(\aid).
\end{equation}
\end{DefEnglish}
Note that as $M_{\Sigma}^{+}$ is a free $R_{\Sigma,\Iw}$-module of rank 1 by theorem \ref{TheoTaylorWiles}, $M_{\Sigma}^{+}\tenseur_{R_{\Sigma,\phi(\aid)}}R(\aid)$ is a free $R(\aid)_{\Iw}$-module of rank 1 and hence $\image\pi_{\Sigma,\aid}^{\Iw}$ is free of rank 1.
\begin{Prop}\label{PropIhara}
The map $\pi_{\Sigma,\aid}$ induces a morphism of $R(\aid)_{\Iw}$-modules
\begin{equation}\label{EqMorIhara}
\pi^{\Iw}_{\Sigma,\aid}:M^{+}_{\Sigma}\tenseur_{R_{\Sigma,\phi(\aid)}}R(\aid)\fleche M^{+}_{\aid}
\end{equation}
such that
\begin{equation}\label{EqComparaison}
\image\pi^{\Iw}_{\Sigma,\aid}=\Eul_{\Sigma}(\aid)M^{+}_{\aid}.
\end{equation}
\end{Prop}
\begin{proof}
Up to two modifications, this is \cite[Theorem 3.6.2]{EmertonPollackWeston}. The first modification is that the theorem stated there concerns nearly ordinary Hida families of eigenforms under the assumption \ref{HypNO}. One can easily check that this later hypothesis is used in \cite[Section 3.8]{EmertonPollackWeston}, where the result is proved, only at the very onset of the proof to specialize to a classical form in the sense of Hida theory; a step we do not need here. The second is that the Euler factor in \cite[Definition 3.6.1]{EmertonPollackWeston} is evaluated at $\diamant{\ell^{-1}}$ whereas ours is evaluated at 1. The reason is that the analytic $p$-adic $L$-function constructed there interpolates the special value at 0 of the motive $L(M(f)(1),0)$ whereas our algebraic object is related to $L(M(f)(s),0)$ for $1\leq s\leq k-1$ so we have incorporated the Tate twist in the Galois representation $\Vcal(\aid)$.
\end{proof}
Let $\Sigma'\supset\Sigma\supset\{Np\}$ be two finite sets of finite primes. Then restricting the action of $R_{\Sigma'}$ to forms of level $N(\Sigma)$ realizes $R_{\Sigma}$ as a quotient of $R_{\Sigma'}$. In this context, proposition \ref{PropIhara} admits an easier variant relating $M_{\Sigma'}$ and $M_{\Sigma}$. Define as above
\begin{equation}\nonumber
\pi_{\Sigma,\Sigma'}:\Hun(X_{1}(N(\Sigma'))(\C),\Fcal_{k-2}\tenseur_{\Z}\zp)\fleche\Hun(X_{1}(N(\Sigma))(\C),\Fcal_{k-2}\tenseur_{\Z}\zp)
\end{equation}
to be the composition of the geometric degeneracy maps from $X_{1}(N(\Sigma'))$ to $X_{1}(N(\Sigma))$.
\begin{Prop}\label{PropIharaSigma}
The map $\pi_{\Sigma',\Sigma}$ induces a morphism of $R_{\Sigma,\Iw}$-modules
\begin{equation}\label{EqMorIharaSigma}
\pi^{\Iw}_{\Sigma',\Sigma}:M^{+}_{\Sigma'}\tenseur_{R_{\Sigma'}}R_{\Sigma}\fleche M^{+}_{\Sigma}
\end{equation}
such that
\begin{equation}\label{EqComparaisonSigmaPrime}
\image\pi^{\Iw}_{\Sigma',\Sigma}=M^{+}_{\Sigma}\produit{\ell\in\Sigma'\backslash\Sigma}{}\det(1-\Fr(\ell)|T_{\Sigma}).
\end{equation}
\end{Prop}
\begin{proof}
This is \cite[Proposition 2.6,2.7]{WilesFermat}.
\end{proof}
\begin{Prop}\label{PropIhara2}
The graded invertible $R(\aid)_{\Iw}$-modules
\begin{equation}\nonumber
\Det_{R(\aid)_{\Iw}}\left(\RGamma_{c}(\Z[1/\Sigma],T_{\Sigma,\Iw})\Ltenseur_{R_{\Sigma,\Iw}}R(\aid)_{\Iw}\right)
\end{equation}
and
\begin{equation}\nonumber
\Xcali(T(\aid)_{\Iw})\tenseur_{R(\aid)_{\Iw}}\produittenseur{\ell\in\Sigma^{p}}{}\Xcali^{-1}_{\ell}(T(\aid)_{\Iw})
\end{equation}
are canonically isomorphic.
\end{Prop}
\begin{proof}
The complex $\RGamma_{c}(\Z[1/\Sigma],T_{\Sigma,\Iw})$ is a perfect complex of $R_{\Sigma,\Iw}$-modules so its determinant is well-defined. For simplicity of notation, in this proof only we denote it by $\RGamma_{c}(T_{\Sigma,\Iw})$. By the base-change property of étale cohomology (or chain) complexes of \cite[Théorème 4.3.1]{SGA4} or \cite[Section 4.12]{SGA41/2}, there is a canonical isomorphism
\begin{equation}\nonumber
\RGamma_{c}(T_{\Sigma,\Iw})\Ltenseur_{R_{\Sigma,\Iw}}R(\aid)_{\Iw}=\RGamma_{c}(\Z[1/\Sigma],T(\aid)_{\Iw}).
\end{equation}
By the definition of $\Xcali(T(\aid)_{\Iw})$ and the remark following equation \eqref{EqIsoXcaliSel}, there is thus a  canonical isomorphism
\begin{equation}\nonumber
\Det_{R(\aid)_{\Iw}}\left(\RGamma_{c}(T_{\Sigma,\Iw})\Ltenseur_{R_{\Sigma,\Iw}}R(\aid)_{\Iw}\right)\overset{\can}{\simeq}\Xcali(T(\aid)_{\Iw})\tenseur_{R(\aid)_{\Iw}}\produittenseur{\ell\in\Sigma^{p}}{}\Xcali^{-1}_{\ell}(T(\aid)_{\Iw}).
\end{equation}
\end{proof}
\begin{DefEnglish}\label{DefDeltaSigma}
Let $\Delta_{\Sigma}(T_{\Sigma,\Iw})$ be the graded invertible $R_{\Sigma,\Iw}$-module
\begin{equation}\label{EqDefDeltaSigma}
\Det_{R_{\Sigma,\Iw}}\RGamma_{c}(\Z[1/\Sigma],T_{\Sigma,\Iw})\tenseur_{R_{\Sigma,\Iw}}\left(\Det_{R_{\Sigma,\Iw}}M_{\Sigma}^{+}\right)
\end{equation}
and let $\z_{\Sigma,\Iw}$ be a basis of $\Delta_{\Sigma,\Iw}$.
\end{DefEnglish}
As was the case with $\z(f)$, the element $\z_{\Sigma,\Iw}$ is defined here only up to a choice of unit of $R_{\Sigma,\Iw}$. The exact choice of unit, though immaterial for our concerns, is pinned down after proposition-definition \ref{DefPropPiDeltaSigma}. We show that $\Delta_{\Sigma}(T_{\Sigma,\Iw})\tenseur_{R_{\Sigma,\Iw}}Q(R_{\Sigma,\Iw})$ comes with two specified $R_{\Sigma,\Iw}$-structures.
\begin{Prop}\label{PropIsoCompactBetti}
There exists an isomorphism
\begin{equation}\label{EqIsoCompactBetti}
M_{\Sigma}^{+}\tenseur_{R_{\Sigma,\Iw}}Q(R_{\Sigma,\Iw})\simeq H^{1}_{c}(\Z[1/\Sigma],T_{\Sigma,\Iw})\tenseur_{R_{\Sigma,\Iw}}Q(R_{\Sigma,\Iw})
\end{equation}
 sending a $\Lambda$-basis of $M_{\Sigma}^{+}$ to a $\Lambda$-basis of $H^{1}_{c}(\Z[1/\Sigma],T_{\Sigma,\Iw})$. The identification of these two $Q(R_{\Sigma,\Iw})$-modules by \eqref{EqIsoCompactBetti} and acyclicity induce a specified isomorphism
\begin{equation}\label{EqSpecAcyc}
\left(\Det_{R_{\Sigma,\Iw}}\RGamma_{c}(\Z[1/\Sigma],T_{\Sigma,\Iw})\tenseur_{R_{\Sigma,\Iw}}\left(\Det_{R_{\Sigma,\Iw}}M_{\Sigma}^{+}\right)\right)\tenseur_{R_{\Sigma,\Iw}}Q(R_{\Sigma,\Iw})\isocan Q(R_{\Sigma,\Iw}).
\end{equation}
\end{Prop}
\begin{proof}
For all $\ell\nmid p$, there exists a twist of $T(f)$ such that the eigenvalues of $\Fr(\ell)$ acting on $T(f)^{I_{\ell}}$ are of non-zero weights. Hence $H^{0}(G_{\Q_{\ell}},T_{\Sigma,\Iw})\tenseur_{\Lambda}\Frac(\Lambda)$ vanishes for all $\ell$ and the three complexes
\begin{equation}\nonumber
\RGamma_{c}(\Z[1/\Sigma],T_{\Sigma,\Iw}),\RGamma_{et}(\Z[1/\Sigma],T_{\Sigma,\Iw}),\RGamma_{\et}(\Z[1/p],T_{\Sigma,\Iw})
\end{equation}
become equal after tensor product with $\Frac(\Lambda)$. Hence $H^{i}_{c}(\Z[1/\Sigma],T_{\Sigma,\Iw})\tenseur_{R_{\Sigma,\Iw}}Q(R_{\Sigma,\Iw})$ for $i\neq1$ vanishes because this is already the case for $H^{i}_{\et}(\Z[1/\Sigma],T_{\Sigma,\Iw})\tenseur_{\Lambda}\Frac(\Lambda)$.

From 
\begin{equation}\nonumber
H^{1}_{c}(\Z[1/\Sigma],T_{\Sigma,\Iw})\tenseur_{\Lambda}\Frac(\Lambda)\simeq H^{1}_{\et}(\Z[1/p],T_{\Sigma,\Iw})\tenseur_{\Lambda}\Frac(\Lambda)
\end{equation}
and \cite[Theorem 14.5 (1)]{KatoEuler}, we deduce that $\Hun_{c}(\Z[1/\Sigma],T_{\Sigma,\Iw})$ is a $\Lambda$-module of rank 1. Let $x,y$ be a regular sequence in $\Lambda$. The isomorphisms
\begin{align}\nonumber
&\RGamma_{c}(\Z[1/\Sigma],T_{\Sigma,\Iw})\Ltenseur_{\Lambda}\Lambda/x\simeq\RGamma_{c}(\Z[1/\Sigma],T_{\Sigma,\Iw}/x)\\\nonumber 
&\RGamma_{c}(\Z[1/\Sigma],T_{\Sigma,\Iw}/x)\Ltenseur_{\Lambda/x}\Lambda/(x,y)\simeq\RGamma_{c}(\Z[1/\Sigma],T_{\Sigma,\Iw}/(x,y))
\end{align}
show that $H^{1}_{c}(\Z[1/\Sigma],T_{\Sigma,\Iw})[x]$ and $H^{1}_{c}(\Z[1/\Sigma],T_{\Sigma,\Iw}/x)[y]$ are zero and that
\begin{equation}\nonumber
H^{1}_{c}(\Z[1/\Sigma],T_{\Sigma,\Iw})/x\plonge H^{1}_{c}(\Z[1/\Sigma],T_{\Sigma,\Iw}/x)
\end{equation}
and hence is torsion-free. The depth of $\Hun_{c}(\Z[1/\Sigma],T_{\Sigma,\Iw})$ as $\Lambda$-module is thus at least 2 and so $H^{1}_{c}(\Z[1/\Sigma],T_{\Sigma,\Iw})$ is a free $\Lambda$-module of rank 1. It is thus isomorphic to $M_{\Sigma}^{+}$ as $\Lambda$-module.

All the assertions of the proposition then follow.
\end{proof}
Hence, there exists an $R_{\Sigma,\Iw}$-basis of $\Delta_{\Sigma}(T_{\Sigma,\Iw})\tenseur_{R_{\Sigma,\Iw}}Q(R_{\Sigma,\Iw})$ given by the inverse image of $R_{\Sigma,\Iw}$ through the isomorphism \eqref{EqSpecAcyc}. 
\begin{DefProp}\label{DefPropPiDeltaSigma}
Let $\pi^{\Delta}_{\Sigma,\aid}$ be the map
\begin{equation}\label{EqDefPiDelta}
\pi_{\Sigma,\aid}^{\Delta}:\Delta_{\Sigma,\Iw}(T_{\Sigma,\Iw})\fleche\Delta_{R(\aid)_{\Iw}}(T(\aid)_{\Iw})\tenseur_{R(\aid)_{\Iw}}\Kcal(\aid)
\end{equation}
equal to $-\tenseur_{R_{\Sigma,\Iw}}R(\aid)_{\Iw}$ on $\Det_{R_{\Sigma,\Iw}}\RGamma_{\et}(\Z[1/\Sigma],T_{\Sigma,\Iw})$ and to the determinant of \eqref{EqMorIhara} on $\Det_{R_{\Sigma,\Iw}}M_{\Sigma}^{+}$. Then $\pi^{\Delta}_{\Sigma,\aid}(\Delta_{\Sigma}(T_{\Sigma,\Iw}))$ is equal to $\Delta_{R(\aid)_{\Iw}}(T(\aid)_{\Iw})$.
\end{DefProp}
\begin{proof}
Note that as $M_{\Sigma}^{+}$ and $M_{\aid}^{+}$ are free of rank 1, the definition of the determinant of \eqref{EqMorIhara} poses no problem. It is enough to show that the $R_{\Sigma,\Iw}$-basis $\z_{\Sigma,\Iw}$ of $\Delta_{\Sigma}(T_{\Sigma,\Iw})$ is sent to an $R(\aid)_{\Iw}$-basis of $\Delta_{R(\aid)_{\Iw}}(T(\aid)_{\Iw})$. But combining propositions \ref{PropIhara} and \ref{PropIhara2} shows that $\z_{\Sigma,\Iw}$ is sent to a basis of $\Delta_{R(\aid)_{\Iw}}(T(\aid)_{\Iw})$ multiplied by
\begin{equation}\nonumber
\Eul_{\Sigma}(\aid)\produittenseur{\ell}{}\Xcali_{\ell}^{-1}(T(\aid)_{\Iw}).
\end{equation}
In the canonical trivialization of $\Xcali_{\ell}^{-1}(T(\aid)_{\Iw})$ given by tensor product with $\Kcal(\aid)$ and identification of $R(\aid)_{\Iw}\subset\Kcal(\aid)$ with $\Det_{R(\aid)_{\Iw}}(0)\subset\Xcali_{\ell}^{-1}(T(\aid)_{\Iw})\tenseur\Kcal(\aid)$, the module $\Xcali_{\ell}^{-1}(T(\aid)_{\Iw})$ is sent to $\Eul_{\ell}(\aid)^{-1}R(\aid)_{\Iw}$. Hence, the image of $\z_{\Sigma,\Iw}$ is indeed a basis of $\Delta_{R(\aid)_{\Iw}}(T(\aid)_{\Iw})$.
\end{proof}
Hence, the $R_{\Sigma,\Iw}$-module $R_{\Sigma,\Iw}\z_{\Sigma,\Iw}$ is sent to $R(\aid)_{\Iw}\z(\aid)$ for all $M|N(\Sigma)$ and all minimal prime ideals $\aid\in\Spec\Hecke^{\new}(M)$.

An explicit construction of $\z_{\Sigma,\Iw}$ can be given in terms of the universal elements of \cite{KatoEuler}. In order to do so, it is enough to construct an element $\z_{\Sigma,\Iw}\in\Delta_{\Sigma}(T_{\Sigma,\Iw})$ such that for all modular specializations $\lambda_{f}$ (resp. $\lambda_{g}$) with values in a discrete valuation ring factoring through $R(\aid)$ (resp. $R(\aid')$) and for all $m\in\N$ sufficiently large, the image $\z_{\Sigma}(f,G_{m})$ (resp. $\z_{\Sigma}(g,G_{m})$) of $\z_{\Sigma,\Iw}$ through $\lambda_{f}$ (resp. $\lambda_{g}$) composed with the surjection from $\Lambda$ to $\Ocal[G_{m}]$ is equal to $\z(f,G_{m})$ (resp. $\z(g,G_{m})$). That in turns amount to showing that $\z_{\Sigma}(f,G_{m})$ (resp. $\z_{\Sigma}(g,G_{m})$)  is sent to $L_{\{p\}}^{G_{m}}(f,s)$ (resp. $L_{\{p\}}^{G_{m}}(g,s)$) through the map of proposition \ref{PropKatoZeta}.

As recalled in \ref{SubEulerSystem}, by the independence on the choice of the covering in the construction the analytic Euler system of \cite{KatoEuler}, the elements $\z(f,G_{m})$ and $\z(g,G_{m})$ are linear combinations of images in the relevant spaces of the same element, namely  the $_{c,d}\z_{1,N,p^{m}}(k,r,\xi,S)$, in $\Hun(X(N(\Sigma))(\C),\Fcal_{k-2})\tenseur\Q(\zeta_{p^{m}})$. Furthermore, by \cite[Theorem 5.6]{KatoEuler}, there exists a linear combination $\z_{m}$ of the
$$_{c,d}\z_{1,N,p^{m}}(k,r,\xi,S)\in\Hun(X(N(\Sigma))(\C),\Fcal_{k-2})\tenseur\Q(\zeta_{p^{m}})$$
such that the image of $\z_{m}$ through the period map of \cite[Theorem 5.6]{KatoEuler} is equal to the special value of the universal $L$-function with Euler factors removed (attentive readers of \cite{KatoEuler} know that this linear combination involves $c$, $d$ and the diamond operators $\diamant{d}$ but its exact expression is unimportant to us). The element $\z_{m}$ is independent of all choices, and especially of the choices of $\aid$ and $\aid'$, yet is sent by universality to $L_{\Sigma(\aid)}^{G_{m}}(f,s)$ or $L_{\Sigma(\aid')}^{G_{m}}(g,s)$ after projection to the relevant eigenspaces. Moreover, the elements $\z_{m}$ for a projective system for the norm map, as can be seen either directly from their construction in terms of Siegel units as in \cite[Proposition 2.3]{KatoEuler} or from their characteristic property as $L_{\Sigma(\aid)}^{G_{m}}(f,s)$ or $L_{\Sigma(\aid')}^{G_{m}}(g,s)$ satisfy this property. Let $\z_{\Sigma,\Iw}$ be the the invariant part of the inverse limit on $m$ of the $\z_{m}$ under complex conjugation. Then the image of $\z$ through $\pi^{\Delta}_{\Sigma,\aid}$ composed with projection to the eigenspace corresponding to $f$ and with projection from $\Lambda$ to $\Ocal[G_{m}]$ is equal to $L_{\Sigma(\aid)}^{G_{m}}(f,s)$. By universality of $\z_{m}$, the same statement holds for $\aid'$ and $\lambda_{g}$ so $\z_{\Sigma,\Iw}$ satisfies the expected properties.

From this point of view, the previous proposition can be seen as a conceptual reinterpretation of the computations of \cite[Theorem 5.6]{KatoEuler}. Alternatively, propositions \ref{PropIhara}, \ref{PropIhara2} and \ref{DefPropPiDeltaSigma} express the statement that there exists two \textit{a priori} equally valid ways to associate a $p$-adic measure to an eigenform: one computing the special values of the $L$-function with all Euler factors at places of bad reduction removed and one with only the $p$-adic Euler factor removed. Interestingly, but not surprisingly within the conjectural framework of the ETNC, these two measure come from the very same universal cohomological object through two not quite identical routes. As such, these three propositions are close algebraic counterparts to \cite[Theorem 3.6.2]{EmertonPollackWeston}.

\begin{DefProp}\label{DefPropPiDeltaSigmaPrime}
Let $\pi^{\Delta}_{\Sigma',\Sigma}$ be the map
\begin{equation}\label{EqDefPiDeltaSigma}
\pi_{\Sigma',\Sigma}^{\Delta}:\Delta_{\Sigma'}(T_{\Sigma',\Iw})\fleche\Delta_{\Sigma}(T_{\Sigma,\Iw})\tenseur_{R_{\Sigma,\Iw}}Q(R_{\Sigma,\Iw})
\end{equation}
equal to $-\tenseur_{R_{\Sigma',\Iw}}R_{\Sigma,\Iw}$ on $\Det_{R_{\Sigma,\Iw}}\RGamma_{c}(\Z[1/\Sigma'],T_{\Sigma',\Iw})$ and to the determinant of \eqref{EqMorIharaSigma} on $\Det_{R_{\Sigma',\Iw}}M_{\Sigma'}^{+}$. Then $\pi^{\Delta}_{\Sigma',\Sigma}(\Delta_{\Sigma'}(T_{\Sigma',\Iw}))$ is equal to $\Delta_{\Sigma}(T_{\Sigma,\Iw})$.
\end{DefProp}
\begin{proof}
Once noted that 
\begin{equation}\nonumber
\RGamma_{c}(\Z[1/\Sigma],T_{\Sigma',\Iw})\Ltenseur_{R_{\Sigma',\Iw}}R_{\Sigma,\Iw}=\RGamma_{c}(\Z[1/\Sigma'],T_{\Sigma,\Iw})
\end{equation}
and that
$$\Det^{-1}_{R_{\Sigma,\Iw}}\RGamma(G_{\Q_{\ell}},T_{\Sigma,\Iw})=\Det^{-1}_{R_{\Sigma,\Iw}}[T_{\Sigma,\Iw}\overset{1-\Fr(\ell)}{\fleche}T_{\Sigma,\Iw}]$$ is canonically identified with $\det^{-1}(1-{\Fr(\ell)}|T_{\Sigma,\Iw})R_{\Sigma,\Iw}$ after tensor product with $Q(R_{\Sigma,\Iw})$ for all $\ell\in\Sigma'-\Sigma$, the proof becomes similar (but easier) to that of proposition \ref{DefPropPiDeltaSigma} using proposition \ref{PropIharaSigma} in place of proposition \ref{PropIhara}.
\end{proof}
We are now in position to state a universal ETNC.
\begin{Conj}\label{ConjXcaliHeckeSigma}
There is an identity of $R_{\Sigma,\Iw}$-lattices
\begin{equation}\nonumber
\Delta_{R_{\Sigma,\Iw}}(T_{\Sigma,\Iw})=\Det_{R_{\Sigma,\Iw}}(0)
\end{equation}
in $\Delta_{R_{\Sigma,\Iw}}(T_{\Sigma,\Iw})\tenseur_{R_{\Sigma,\Iw}}Q(R_{\Sigma,\Iw})\overset{\can}{\simeq}Q(R_{\Sigma,\Iw})$.
\end{Conj}
Conjecture \ref{ConjXcaliHeckeSigma} is compatible with modular specializations and with change of levels in the sense of the following proposition.
\begin{Prop}\label{PropCompLevel}
Conjecture \ref{ConjXcaliHeckeSigma} implies conjecture \ref{ConjXcaliRaid} for all $M|N(\Sigma)$ and all minimal prime ideals $\aid\in\Spec\Hecke^{\new}(U(M))$ as well as conjecture \ref{ConjXcaliGm} for all modular specializations $\lambda_{g}$ of $R_{\Sigma}$ and for all morphisms $\phi:\Lambda\fleche S$ as in definition \ref{DefDelta}. Conjecture \ref{ConjXcaliHeckeSigma} for $R_{\Sigma}$ is equivalent to conjecture \ref{ConjXcaliHeckeSigma} for $R_{\Sigma'}$ for all $\Sigma'\supset\Sigma$.
\end{Prop}
\begin{proof}
Let $M|N(\Sigma)$ and $\aid\in\Spec\Hecke^{\new}(U(M))$ be a minimal prime. According to propositions \ref{PropIhara}, \ref{PropIhara2} and \ref{DefPropPiDeltaSigma}, $\Delta_{R_{\Sigma,\Iw}}(T_{\Sigma,\Iw})$ and $\Det_{R_{\Sigma,\Iw}}(0)$ are sent through the map $\pi^{\Delta}_{\Sigma,\aid}$ to $\Delta_{R(\aid)_{\Iw}}(T(\aid)_{\Iw})$ and $\Det_{R(\aid)_{\Iw}}(0)$ respectively. Hence, conjecture \ref{ConjXcaliHeckeSigma} implies conjecture \ref{ConjXcaliRaid} and thus conjecture \ref{ConjXcaliGm} for all modular specializations $\lambda_{g}$ of $R_{\Sigma}$ factoring through $R(\aid)$ and for all morphism $\phi:\Lambda\fleche S$ as in definition \ref{DefDelta} by proposition \ref{PropCompatibleSpec}. The last assertion follows from proposition \ref{DefPropPiDeltaSigmaPrime} and functoriality of the determinant.
\end{proof}

\section{Proofs of the main results}
\subsection{A lemma about Euler systems for modular forms}
Henceforth, we consistently assume the following.
\begin{HypEnglish}\label{HypTamagawa}
There exists $\ell\in\Sigma$ such that $\ell||N(\rhobar_{f})$ and such that the image of $\rhobar_{f}|_{G_{\Q_{\ell}}}$ contains a non-identity unipotent element.
\end{HypEnglish}
As discussed in the introduction, the ultimate mathematical meaning of assumption \ref{HypTamagawa} remains quite mysterious; its proximate function, on the other hand, is provided by the following lemma.
\begin{LemEnglish}\label{LemUnipotent}
Let $g$ be a newform with coefficients in $\Ocal'$ congruent to $f$ modulo $p$. There exists $\s\in\Gal(\Qbar/\Q(\zeta_{p^{\infty}}))$ such that the cokernel of $\rho_{g}(\s)-1$ is an $\Ocal'$-module free of rank 1.
\end{LemEnglish}
\begin{proof}
As $\Q(\zeta_{p^{\infty}})$ is unramified at $\ell$, it is enough to show that there exists $\s\in I_{\ell}$ such that $\rho_{g}(\s)-1\neq0=(\rho_{g}(\s)-1)^{2}$ and such that the cokernel of $\rhobar_{f}(\s)-1$ is of dimension 1. By assumption \ref{HypTamagawa}, the $I_{\ell}$-representation $\rhobar_{g}=\rhobar_{f}$ is ramified and has a non-identity unipotent element $\rhobar_{f}(\s)$ in its image so is not the direct sum of two characters. Thus $\pi(g)_{\ell}$ is a Steinberg representation; see for instance \cite[Section 1]{CarayolNiveau} or \cite[Page 1]{DiamondTaylor}. After restriction to a subgroup $U$ of prime to $p$ finite index in $I_{\ell}$, $\rho_{g}|_{U}$ is then unipotent but non-trivial. Hence, a suitable prime to $p$ power of $\rho_{g}(\s)$ is a unipotent element mapping to a non-identity unipotent element.
\end{proof}
The following proposition, due to K.Kato, plays a crucial role in the reduction to the ETNC with coefficients in the Hecke algebra to known results about Iwasawa theory of modular forms.
\begin{Prop}\label{PropContradiction}
Assume that $k>2$ or that $\rhobar_{f}$ is not nearly ordinary. Let $g\in S_{k}(U)$ be a newform with coefficients in $\Ocal'$ congruent to $f$. Let $T(g)_{\Iw}$ be the Galois representation with coefficients in $A=\Lambda\tenseur_{\Ocal}\Ocal'$ attached to $g$. As in sub-section \ref{SubEulerLambda}, we denote by $Z$ the image of $Z(g)$ in $\Hun_{\et}(\Z[1/p],T(g)_{\Iw})$. Then:
\begin{equation}\nonumber
\carac_{A}H^{2}_{\et}(\Z[1/p],T(g)_{\Iw})\mid\carac_{A}\Hun_{\et}(\Z[1/p],T(g)_{\Iw})/Z
\end{equation}
\end{Prop}
\begin{proof}
Thanks to assumption \ref{HypTamagawa} and lemma \ref{LemUnipotent}, the hypothesis \cite[Section 0.6 ($ii_{str}$)]{KatoEulerOriginal} is satisfied. By \cite[12.5 (4)]{KatoEuler}, the inequality of lengths 
\begin{equation}\nonumber
\length_{A_{\pid}}H^{2}_{\et}(\Z[1/p],T)_{\pid}\leq \length_{A_{\pid}}\Hun_{\et}(\Z[1/p],T)_{\pid}/Z_{\pid}+\length_{A_{\pid}}H^{2}(G_{\qp},T(g)_{\Iw})_{\pid}.
\end{equation}
thus holds for all $\pid$ of grade 1 in $A$. Fix such a $\pid$. Then $H^{2}(G_{\qp},T(g)_{\Iw})$ does not vanish after localization at $\pid$ only if it is infinite. Following \cite[Section 13.13]{KatoEuler}, we note that this happens only if $\rho_{g}|G_{\qp}$ is reducible and not potentially crystalline, and hence only if $\pi(g)_{p}$ is an ordinary Steinberg representation of weight 2 by \cite[Theorem]{SaitoHodge}.
\end{proof}

\subsection{Weak forms of the ETNC conjectures}
We formulate weakened version of our conjectures \ref{ConjXcaliLambda}, \ref{ConjXcaliGm}, \ref{ConjXcaliRaid} and \ref{ConjXcaliHeckeSigma}.
\begin{Conj}\label{ConjXcaliGmWeak}
If $s\neq k/2$, there is an inclusion of $\Ocal[G_{m}]$-lattices
\begin{equation}\label{EqDeltaGmWeak}
\Delta_{\Ocal[G_{m}]}(T(f)\tenseur_{\Ocal}\Ocal[G_{m}])\subset\Det_{\Ocal[G_{m}]}(0)
\end{equation}
inside $\Delta_{F_{\lambda}[G_{m}]}(V(f)\tenseur_{F_{\lambda}}F_{\lambda}[G_{m}])$ for all $m\geq1$. If $s=k/2$, this is true for all $m\geq1$ except possibly finitely many. More generally, there is an inclusion of $S$-lattices
\begin{equation}\label{EqDeltaSWeak}
\Delta_{S}(T(f)\tenseur_{\Ocal}S)\subset\Det_{S}(0)
\end{equation}
inside $\Delta_{S[1/p]}(V(f)\tenseur_{F_{\lambda}}S[1/p])$ for all morphism $\phi:\Lambda\fleche S$ as in definition \ref{DefDelta}. 
\end{Conj}
\begin{Conj}\label{ConjXcaliLambdaWeak}
There is an inclusion of $\Lambda$-lattices
\begin{align}\label{EqDeltaLambdaWeak}
\Delta_{\Lambda}(T(f)_{\Iw})\subset\Det_{\Lambda}(0)
\end{align}
inside $\Delta_{\Lambda}(T(f)_{\Iw})\tenseur_{\Lambda}\Frac(\Lambda)\overset{\can}{\simeq}\Frac(\Lambda)$.
\end{Conj}
\begin{Conj}\label{ConjXcaliRaidWeak}
There is an inclusion of $R(\aid)_{\Iw}$-lattices
\begin{align}\nonumber
\Delta_{R(\aid)_{\Iw}}(T(\aid)_{\Iw})\subset\Det_{R(\aid)_{\Iw}}(0)
\end{align}
inside $\Delta_{\Kcal(\aid)}(\Vcal(\aid)_{\Iw})\overset{\can}{\simeq}\Kcal(\aid)$.
\end{Conj}
\begin{Conj}\label{ConjXcaliHeckeSigmaWeak}
There is an inclusion of $R_{\Sigma}$-lattices
\begin{equation}\nonumber
\Delta_{R_{\Sigma,\Iw}}(T_{\Sigma,\Iw})\subset\Det_{R_{\Sigma,\Iw}}(0)
\end{equation}
in $\Delta_{R_{\Sigma,\Iw}}(T_{\Sigma,\Iw})\tenseur_{R_{\Sigma,\Iw}}Q(R_{\Sigma,\Iw})\overset{\can}{\simeq}Q(R_{\Sigma,\Iw})$.
\end{Conj}
Propositions \ref{PropCompatibleSpec} and \ref{PropCompLevel} have the following counterpart, whose proof is similar but easier, and therefore omitted.
\begin{Prop}\label{PropCompWeak}
Conjecture \ref{ConjXcaliHeckeSigmaWeak} for $\Sigma$ is equivalent to conjecture \ref{ConjXcaliHeckeSigmaWeak} for all $\Sigma'\supset\Sigma$. Conjecture \ref{ConjXcaliHeckeSigmaWeak} implies conjecture \ref{ConjXcaliRaidWeak} for all $M|N(\Sigma)$ and all minimal prime ideals $\aid\in\Spec\Hecke^{\new}(U(M))$. Conjecture \ref{ConjXcaliRaidWeak} for $\aid$ implies conjecture \ref{ConjXcaliLambdaWeak} for a modular specialization $\lambda_{f}$ of $R(\aid)$. Conjecture \ref{ConjXcaliLambdaWeak} for a modular specialization $\lambda_{f}$ implies conjecture \ref{ConjXcaliGmWeak} for $\lambda_{f}$ and $\phi:\Lambda\fleche S$ a morphism as in definition \ref{DefDelta}.
\end{Prop}

The aim of the next two sub-sections is to prove conjecture \ref{ConjXcaliHeckeSigmaWeak} under on one hand assumption \ref{HypTamagawa} and on the other either assumption \ref{HypFlat} or assumption \ref{HypNO}.

\subsection{Proof of conjecture \ref{ConjXcaliHeckeSigmaWeak} under assumption \ref{HypLocIrr}}
In addition to our ongoing assumptions, we assume in this sub-section the following.
\begin{HypEnglish}\label{HypLocIrr}
The local representation $\rhobar_{f}|_{G_{\qp}}$ is irreducible.
\end{HypEnglish}
Because $\rhobar_{f}$ then does not satisfy assumption \ref{HypNO}, it has to satisfy assumption \ref{HypFlat}. By proposition \ref{PropCompWeak}, there is no loss of generality in assuming furthermore that $f$ is minimally ramified outside $p$, or in other words that $f$ is new of level $N(\rhobar_{f})p^{s}$ for some $s$, in order to prove conjecture \ref{ConjXcaliHeckeSigmaWeak}. Thanks to assumption \ref{HypLocIrr}, $\rho_{f}$ is attached to a point of the minimal universal deformation ring  $R^{\fl}_{\Sigma,\Id}(\rhobar_{f})$.

In all this subsection, we identify a graded invertible module with grade equal to zero to the invertible module equal to its first component. For simplicity of notations, we also sometimes abbreviate $\Delta_{\Sigma}(T_{\Sigma,\Iw})$ in $\Delta_{\Sigma}$ and $Q(R_{\Sigma,\Iw})$ in $Q_{\Sigma,\Iw}$. 

\subsubsection{Trivialization of the fundamental lines}
If $(u,v)$ is a pair of regular elements of $R_{\Sigma,\Iw}$, we say that an ideal $J$ is adequate with respect to $(u,v)$ if $xy\notin J$. Given such a pair $(u,v)$ in $R_{\Sigma,\Iw}$, $J$ is adequate with respect to $(u,v)$ if it is contained in a large enough power of the maximal ideal of $R_{\Sigma,\Iw}$ and the subset of prime ideals which are not adequate with respect to $(u,v)$ is of large codimension in $\Spec R_{\Sigma,\Iw}$. 

For all $\Sigma\supset\{\ell|Np\}$, let 
\begin{equation}\nonumber
\psi_{\Sigma}:\Delta_{\Sigma}(T_{\Sigma,\Iw})\tenseur Q(R_{\Sigma,\Iw})\isom Q(R_{\Sigma,\Iw})
\end{equation}
be the specified isomorphism of equation \eqref{EqSpecAcyc}. Then there exists regular elements $(x,y)\in R_{\Sigma,\Iw}^{2}$ such that the following diagram commutes.
\begin{equation}
\xymatrix{
\Delta_{\Sigma}\ar[r]^{\psi_{\Sigma}}\ar[d]&\frac{x}{y}R_{\Sigma,\Iw}\ar[d]\\
\Delta_{\Sigma}\tenseur_{R_{\Sigma,\Iw}}Q_{\Sigma,\Iw}\ar[r]^(0.6){\psi_{\Sigma}}&Q_{\Sigma,\Iw}
}
\end{equation}
Equivalently, $\psi_{\Sigma}$ induces an isomorphism 
\begin{equation}\label{EqEquationXY}
y\Delta_{\Sigma}\overset{\psi_{\Sigma}}{\simeq} xR_{\Sigma,\Iw}.
\end{equation}
Let $J$ be an ideal of $R_{\Sigma,\Iw}$ adequate for $(x,y)$ and let $R$ be the quotient $R_{\Sigma,\Iw}/J$. Then $\psi_{\Sigma}$ induces an isomorphism of non-zero $R$-modules
\begin{equation}\label{EqEquationModJ}
\bar{y}\Delta_{\Sigma}(T_{\Sigma,\Iw})/J\overset{\psi_{\Sigma}}{\simeq} \bar{x}R.
\end{equation}

Let $\Sigma,\Sigma'$ two finite sets of finite primes containing $\{\ell|Np\}$. Choose elements $(x,y)$ and $(x',y')$ as in equation \eqref{EqEquationXY} for $\Delta_{\Sigma}(T_{\Sigma,\Iw})$ and $\Delta_{\Sigma'}(T_{\Sigma',\Iw})$ respectively and let $J$ (resp. $J'$) be an ideal of $R_{\Sigma,\Iw}$ (resp. $R_{\Sigma',\Iw}$) adequate with respect to $(x,y)$ (resp. $(x',y')$). If $R=R_{\Sigma,\Iw}/J$ is isomorphic to $R'=R_{\Sigma',\Iw}/J'$, then let $\Sigma''$ be $\Sigma\cup\Sigma'$. Let $\phi_{\Sigma}$ be the isomorphism between $\Delta_{\Sigma}\tenseur Q_{\Sigma,\Iw}$ and $Q_{\Sigma,\Iw}$ of proposition \eqref{EqSpecAcyc} but normalized so that the image of $\Delta_{\Sigma}$ is $R_{\Sigma,\Iw}$ and let $\phi_{\Sigma'}$ and $\phi_{\Sigma''}$ be likewise. As any arrow in the diagram
\begin{equation}\label{EqDiagCompXY}
\xymatrix{
&\Delta_{\Sigma'}\ar[r]^{\phi_{\Sigma'}}&R_{\Sigma',\Iw}\ar[dr]^{\mod J'}\\
\Delta_{\Sigma''}\ar[ur]^{\pi_{\Sigma'',\Sigma'}^{\Delta}}\ar[dr]_{\pi_{\Sigma'',\Sigma}^{\Delta}}\ar[r]^{\phi_{\Sigma''}}&R_{\Sigma'',\Iw}\ar[rd]\ar[ru]&&R\\
&\Delta_{\Sigma}\ar[r]^{\phi_{\Sigma}}&R_{\Sigma,\Iw}\ar[ru]_{\mod J}
}
\end{equation}
sends a basis to a basis, it commutes perhaps up to multiplication by a unit. Thus a choice of $x_{\Sigma''},y_{\Sigma''}$ such that
\begin{equation}\nonumber
y_{\Sigma''}\Delta_{\Sigma''}(T_{\Sigma'',\Iw})\overset{\psi_{\Sigma''}}{\simeq} x_{\Sigma''}R_{\Sigma'',\Iw}.
\end{equation} 
induces choices of $x_{\Sigma},y_{\Sigma}$ and $x_{\Sigma'},y_{\Sigma'}$ which are compatible after reduction modulo $J$ and $J'$.
\subsubsection{The Taylor-Wiles system of refined fundamental lines}
For $S$ a complete local $\Ocal$-algebra, a Taylor-Wiles system $\{(R_{Q},M_{Q})\}_{Q\in X}$ over $S$ consists of the following data.
\begin{enumerate}
\item The set $X$ is infinite. Its elements are the empty set and finite sets $Q$ of constant cardinality $r$ of rational primes congruent to 1 modulo $p$.
\suspend{enumerate}
For $Q\in X$ and $q\in Q$, we denote by $\Gamma_{q}$ the $p$-Sylow subgroup of $(\Z/q\Z)\croix$ and by $\Gamma_{Q}$ the product
\begin{equation}\nonumber
\Gamma_{Q}=\produit{q\in Q}{}\Gamma_{q}.
\end{equation}
\resume{enumerate}
\item For all $n\in\N$, the subset 
\begin{equation}\nonumber
X_{n}=\{Q\in X|\forall\ q\in Q,\ q\equiv 1\modulo p^{n}\}
\end{equation}
is infinite.
\item For $Q\in X$, $R_{Q}$ is a complete local noetherian $S[\Gamma_{Q}]$-algebra generated by at most $r$ elements and $M_{Q}$ is an $R_{Q}$-module which is a free $S[\Gamma_{Q}]$-module of finite rank independent of $Q$.
\suspend{enumerate}
We denote $(R_{\vide},M_{\vide})$ by $(R,M)$ and let $I_{Q}$ be the augmentation ideal of $S[\Gamma_{Q}]$.
\resume{enumerate}
\item For all $Q\in X$, there is a surjection of local $S$-algebras 
\begin{equation}\nonumber
R_{Q}/I_{Q}R_{Q}\surjection R
\end{equation}
equal to the identity if $Q=\vide$. 
\item The morphism 
\begin{equation}\nonumber
R_{Q}/I_{Q}R_{Q}\fleche \End_{S} M_{Q}/I_{Q}M_{Q}
\end{equation}
factors through $R$ and $M_{Q}/I_{Q}M_{Q}$ is isomorphic to $M$ as an $R$-module.
\end{enumerate}
The ring $R_{\Sigma}=R_{\Sigma,\Id}^{\fl}$ being minimal, there exists by \cite{WilesFermat,TaylorWiles} a well-chosen set $X$ such that the system $\{(R_{\Sigma\cup Q},\Delta_{\Sigma\cup Q})\}_{Q\in X}$ in which we identify $\Delta_{\Sigma\cup Q}$ with its underlying free $\Hecke_{\Sigma\cup Q}$-module is a Taylor-Wiles system over $\Ocal$. Taking the tensor product with $\Lambda$, this yields a Taylor-Wiles system $\{(R_{\Sigma\cup Q,\Iw},\Delta_{\Sigma\cup Q,\Iw})\}_{Q\in X}$

For $Q\in X$ non-empty and $n\in\N$, denote by $J_{Q,n}\subset\Lambda[\Gamma_{Q}]$ the ideal generated by
\begin{equation}\nonumber
\mgot_{\Lambda}^{n},\ \{\gamma^{p^{n}}-1|\gamma\in\Gamma_{q}\}
\end{equation}
and by $R_{Q,n}$ the quotient $R_{\Sigma\cup Q,\Iw}/J_{Q,n}R_{\Sigma\cup Q,\Iw}$. Then there exists a projective system $\{R_{Q(n),n}\}_{n\in\N}$ with surjective transition maps
\begin{equation}\nonumber
R_{Q(n+1),n+1}\surjection R_{Q(n+1),n}\simeq R_{Q(n),n}
\end{equation}
such that the inverse limit 
\begin{equation}\nonumber
R_{\infty}=\limproj{n}\ R_{Q(n),n}
\end{equation}
is isomorphic to the power-series ring $\Lambda[[X_{1},\cdots,X_{r}]]$ and is in particular local regular of dimension $2+r\geq 3$ (see \cite[Section 2.2]{FujiwaraDeformation}). This projective system induces a projective systems of refined fundamental lines $\Delta_{\Sigma\cup Q(n),m}=\Delta_{\Sigma\cup Q(n),\Iw}/J_{Q,m}$ satisfying
\begin{equation}\nonumber
\Delta_{R_{Q(n+1),n+1}}\surjection\Delta_{R_{Q(n+1),n}}\isocan \Delta_{R_{Q(n),n}}.
\end{equation}
Note that though this does not appear anymore in the notation, $R_{Q(n),n}$ and $\Delta_{R_{Q(n)},n}$ are $\Lambda$-modules. Let $(x,y)$ be a choice of elements as in \eqref{EqEquationXY}. For $n$ large enough, the ideal $J_{Q(n),n}$ is adequate with respect to $(x,y)$. Hence, for $n$ large enough, we can make appropriate choices as in diagram \eqref{EqDiagCompXY} to construct a projective system $(x_{n},y_{n})_{n\in\N}\in R_{Q(n)}/J_{Q(n),n}$ (which depends highly on the choices at each steps) such that 
\begin{equation}\nonumber
y_{n}\Delta_{R_{Q(n),n}}\isocan x_{n}R_{Q(n)}/J_{Q(n),n}
\end{equation} 
for $n$ large enough. We let $\Delta_{\infty}$ be the inverse limit of the $\Delta_{R_{Q(n),n}}$ and $(x_{\infty},y_{\infty})\in R_{\infty}^{2}$ be the inverse limit of the $(x_{n},y_{n})$. Then there is a specified (even canonical once all previous choices have been made) isomorphism 
\begin{equation}\nonumber
\psi_{\infty}:y_{\infty}\Delta_{\infty}\isocan x_{\infty}R_{\infty}
\end{equation}  
of invertible modules defined as the inverse limits of the $\psi_{\Sigma\cup Q(n)}\modulo J_{Q(n),n}$. If the image of $\Delta_{\infty}$ through $\psi_{\infty}$ is included in $R_{\infty}$, then the image of $\Delta_{\Sigma,\Iw}(T_{\Sigma,\Iw})$ through $\psi$ is included in $R_{\Sigma,\Iw}$ and hence conjecture \ref{ConjXcaliHeckeSigmaWeak} is true. Consequently, it is enough to show that the image of $\Delta_{\infty}$ through $\psi_{\infty}$ is included in $R_{\infty}$.
\subsubsection{Reduction to classical Iwasawa theory of modular forms}\label{SubReduction}
Assume by way of contradiction that this is not the case, \textit{i.e} that the image of $\Delta_{\infty}$ through $\psi_{\infty}$ is not included in $R_{\infty}$. Then there exists a prime ideal $\pid$ generated by a sub-system of parameters of $R_{\infty}$ which is adequate with respect to $(x_{\infty},y_{\infty})$ such that $A=R_{\infty}/\pid$ is a discrete valuation ring flat over $\Ocal$ and such that there exists a power $N$ of the principal maximal ideal of $A$ such that the image of $\Delta_{\infty}\tenseur_{R_{\infty}}A/\mgot_{A}^{N}$ through $\bar{\psi}_{\infty}=\psi_{\infty}\modulo(\pid,\mgot_{A}^{N})$ is equal to $A/\mgot_{A}^{N+n_{0}}$ for some $n_{0}>0$. If $A$ satisfies these properties, then any discrete valuation ring finite flat over $A$ as $A$-module also does, so that we can assume $A$ is large. By construction of $R_{\infty}$, there exists $n$ large enough so that $R_{Q(n),n}\tenseur_{\Ocal}A$ surjects onto $A/\mgot_{A}^{N}$. The diagram
\begin{equation}\nonumber
\xymatrix{
\Delta_{\infty}\ar[dd]_{-\tenseur_{R_{\infty}}A/\mgot_{A}^{N}}\ar[rr]^{-\tenseur_{R_{\infty}}R_{\Sigma\cup Q(n)}}&&\Delta_{\Sigma\cup Q(n)}\tenseur_{\Ocal}A\ar[dd]^{\psi_{\Sigma\cup Q(n)}\modulo\mgot_{A}^{N}}\\
\\
\Delta_{\infty}\tenseur_{R_{\infty}}A/\mgot_{A}^{N}\ar[rr]^{\bar{\psi}_{\infty}}&&A/\mgot_{A}^{N+n_{0}}
}
\end{equation}
is then commutative perhaps up to a unit. In particular, the image of $\Delta_{\Sigma\cup Q(n)}\tenseur_{\Ocal}B$ through $\psi_{\Sigma\cup Q(n)}$ is not included in $R_{\Sigma\cup Q(n)}\tenseur_{\Ocal}B$ for any finite flat discrete valuation ring extension $B$ of $A$. Choose $B$ large enough so that the ring $R_{\Sigma\cup Q(n)}\tenseur_{\Ocal}B$ is a product of integrally closed domains (for instance by taking $B$ the normalization of $A$). There then exists a minimal prime ideal $\aid\in\Spec(R_{\Sigma\cup Q(n)}\tenseur_{\Ocal}B)$ such that the image of $\Delta_{R_{\Sigma\cup Q(n)}}\tenseur_{\Ocal}B/\aid$ through $\psi_{\Sigma\cup Q_{n}}$ is not included in $R_{\Sigma\cup Q(n)}\tenseur_{\Ocal}B/\aid$. Let $(T,\rho,B/\aid)$ be the $G_{\Q,\Sigma\cup Q(n)}$-representation $T_{\Sigma\cup Q(n)}\tenseur_{R_{\Sigma\cup Q(n)}}B/\aid$ and let $Z$ be the image of $Z(\aid)$ in $\Hun_{\et}(\Z[1/p],T)$. The ring $R_{\Sigma\cup Q(n)}\tenseur_{\Ocal}B$ is equal to $S=\Lambda\tenseur_{\Ocal}\Ocal'$ for some discrete valuation ring $\Ocal'$ finite flat over $\Ocal$ and hence is regular of dimension 2. Moreover, $T$ is the $G_{\Q}$-representation with coefficients in $S$ attached to some newform $g\equiv f\modulo p$ of tame level $N(\rhobar_{f})\cup Q$ and of weight $k$. The statement that the image of $\Delta_{R_{Q(n)}}\tenseur_{R_{Q(n)}}B/\aid$ through $\psi_{\Sigma\cup Q_{n}}$ is not included in $B/\aid$ thus becomes the statement
\begin{equation}\label{EqContradiction}
\carac_{S}H^{2}_{\et}(\Z[1/p],T(g)_{\Iw})\nmid \carac_{S}\Hun_{\et}(\Z[1/p],T(g)_{\Iw})/Z.
\end{equation}
This contradicts proposition \ref{PropContradiction}.
\subsection{Proof of conjecture \ref{ConjXcaliHeckeSigmaWeak} under assumption \ref{HypNOf}}
In addition to our ongoing assumptions, we assume in this sub-section the following.
\begin{HypEnglish}\label{HypNOf}
The local representation $\rho_{f}|_{G_{\qp}}$ is reducible.
\end{HypEnglish}
In particular, $\rhobar_{f}$ satisfies assumption \ref{HypNO}. The proof of conjecture \ref{ConjXcaliHeckeSigmaWeak} under assumption \ref{HypNO} cannot imitate directly the proof in sub-section \ref{SubReduction} for two reasons. First, if $\rhobar_{f}$ is nearly ordinary finite, then $\rho_{f}$ might not correspond to a point on the minimal deformation ring of $\rhobar_{f}$. Second, the statement
\begin{equation}\nonumber
\carac_{\Lambda} H^{2}_{\et}(\Z[1/p],T(g)_{\Iw})\nmid \carac\Hun_{\et}(\Z[1/p],T(g)_{\Iw})/Z
\end{equation}
invoked in equation \eqref{EqContradiction} does not contradict \cite[Theorem 12.4]{KatoEuler}: when $\pi(g)_{p}$ is a Steinberg representation, the non-finiteness of $H^{2}(G_{\qp},T(g)_{\Iw})$ might contribute a non-trivial error term. Both difficulties disappear if we repeat the entire argument of this manuscript with $\Hecke^{\new}$ or $\Hecke^{\ord}$ replaced everywhere by the Hida-Hecke algebras $\Hecke^{\new,\ord}$ and $\Hecke^{\red,\ord}$. 
\subsubsection{Hida-theoretic conjectures}
We repeat the entirety of section \ref{SectionETNC} with the following modifications. The discrete valuation ring $\Ocal$ is replaced everywhere by $\Lambda_{\Hi}\simeq\Ocal[[Y]]$ and the Iwasawa algebra $\Lambda$ is replaced everywhere by $\Lambda_{\Hi,\Iw}=\Lambda_{\Hi}[[\Gamma]]\simeq\Ocal[[X,Y]]$. The Betti cohomology group $M_{B}\tenseur_{\Z}\Ocal$ is replaced everywhere by 
\begin{equation}\nonumber
M_{\Hi}=\limproj{s}\ e^{\ord}\Hun_{\et}(X_{1}(Np^{s})\times_{\Q}\Qbar,\Fcal_{k-2}\tenseur_{\Z}\Ocal).
\end{equation}
The $G_{\Q}$-representation $(T(f)_{\Iw},\rho_{f},\Lambda)$ is replaced everywhere by $(T(f)_{\Hi,\Iw},\rho_{f,\Hi},\Lambda_{\Hi,\Iw})$. Likewise, $T(\aid)_{\Iw}$ over $R(\aid)_{\Iw}$ is replaced by $T(\aid)_{\Hi,\Iw}$ over $R(\aid)_{\Hi,\Iw}$ and $T_{\Sigma,\Iw}$ over $R(\aid)_{\Iw}$ is replaced by $T_{\Sigma,\Hi,\Iw}$ over $R_{\Sigma,\Hi,\Iw}$
%

The elements $_{c,d}\z^{(p)}_{M,N}(k,r)$ are replaced by the elements $\limproj{s}\ _{c,d}\z^{(p)}_{Mp^{s},Np^{s}}(k,r)$. The existence of the morphisms
\begin{equation}\nonumber
Z(f)_{\Hi}:(M_{\Hi}\tenseur_{\Ocal}\Lambda)^{+}[-1]\fleche\RGamma_{\et}(\Z[1/p],T(f)_{\Hi,\Iw})
\end{equation}
and
\begin{equation}\nonumber
Z(\aid)_{\Hi}:M_{\aid,\Hi}^{+}[-1]\fleche\RGamma_{\et}(\Z[1/p],T(\aid)_{\Hi,\Iw})
\end{equation}
then follows by reduction to finite level using \eqref{EqControlHecke} as explained in \cite[Section 3]{FukayaKatoSharifi}, see especially \cite[Theorem 3.2.3]{FukayaKatoSharifi} (in fact, the results proved there are more general, as they incorporate the possibility of the $G_{\Q}$-action being of residual type; in our case, the torsion submodule which might arise in this way vanishes after localization at $\mgot_{f}$). Then $\Delta_{R(\aid)_{\Hi,\Iw}}(T(\aid)_{\Hi,\Iw})$ with it distinguished basis $\z(\aid)_{\Hi}$ are defined as in definition \ref{DefDeltaAid}. Proposition \ref{PropIhara} is replaced by \cite[Theorem 3.6.2]{EmertonPollackWeston}. In the proof of proposition \ref{PropIsoCompactBetti}, it might no longer be true that $\Hun_{c}(\Z[1/\Sigma],T_{\Sigma,\Hi,\Iw})$ is free of rank 1 over $\Lambda_{\Hi,\Iw}$. Nevertheless, the same proof as in the proof of proposition \ref{PropIsoCompactBetti} shows that after localization at any grade 1 prime of this regular ring, $\Hun_{c}(\Z[1/\Sigma],T_{\Sigma,\Hi,\Iw})$ becomes free of rank 1. Thus there is still an isomorphism
\begin{equation}\label{EqIsoCompactBettiHida}
M_{\Sigma,\Hi}^{+}\tenseur_{R_{\Sigma,\Hi\Iw}}Q(R_{\Sigma,\Hi\Iw})\simeq H^{1}_{c}(\Z[1/\Sigma],T_{\Sigma,\Hi,\Iw})\tenseur_{R_{\Sigma,\Iw}}Q(R_{\Sigma,\Hi,\Iw}).
\end{equation}
The definition of $\Delta_{\Sigma,\Hi}(T_{\Sigma,\Iw,\Hi})$, of its distinguished basis $\z_{\Sigma,\Hi}$ and of the map
\begin{equation}\nonumber
\pi^{\Delta}_{\Sigma,\aid,\Hi}:\Delta_{\Sigma,\Hi}(T_{\Sigma,\Iw,\Hi})\fleche\Delta_{R(\aid)_{\Hi,\Iw}}(T(\aid)_{\Hi,\Iw})
\end{equation}
sending $\z_{\Sigma,\Hi}$ to $\z(\aid)_{\Hi}$ is then as in definition \ref{DefDeltaSigma} and proposition-definition \ref{DefPropPiDeltaSigma} with similar proofs. Conjectures \ref{ConjXcaliRaid}, \ref{ConjXcaliHeckeSigma}, \ref{ConjXcaliRaidWeak} and \ref{ConjXcaliHeckeSigmaWeak} are then generalized as follows.
\begin{Conj}\label{ConjRaidHida}
There is an identity of $R(\aid)_{\Hi,\Iw}$-lattices
\begin{align}\label{EqDeltaRaidHida}
\Delta_{R(\aid)_{\Hi,\Iw}}(T(\aid)_{\Hi,\Iw})=\Det_{R(\aid)_{\Hi,\Iw}}(0)
\end{align}
inside $\Delta_{R(\aid)_{\Hi,\Iw}}(T(\aid)_{\Hi,\Iw})\tenseur_{R(\aid)_{\Hi,\Iw}}\Kcal(\aid)_{\Hi,\Iw}\overset{\can}{\simeq}\Kcal(\aid)_{\Hi,\Iw}$.
\end{Conj}
\begin{Conj}\label{ConjSigmaHida}
There is an identity of $R_{\Sigma,\Hi,\Iw}$-lattices
\begin{equation}\nonumber
\Delta_{R_{\Sigma,\Hi,\Iw}}(T_{\Sigma,\Hi,\Iw})=\Det_{R_{\Sigma,\Hi,\Iw}}(0)
\end{equation}
in $\Delta_{R_{\Sigma,\Hi,\Iw}}(T_{\Sigma,\Hi,\Iw})\tenseur_{R_{\Sigma,\Hi,\Iw}}Q(R_{\Sigma,\Hi,\Iw})\overset{\can}{\simeq}Q(R_{\Sigma,\Hi,\Iw})$.
\end{Conj}
\begin{Conj}\label{ConjAidWeak}
There is an inclusion of $R(\aid)_{\Hi,\Iw}$-lattices
\begin{align}\label{EqDeltaRaidWeak}
\Delta_{R(\aid)_{\Hi,\Iw}}(T(\aid)_{\Hi,\Iw})\subset\Det_{R(\aid)_{\Hi,\Iw}}(0)
\end{align}
inside $\Delta_{R(\aid)_{\Hi,\Iw}}(T(\aid)_{\Hi,\Iw})\tenseur_{R(\aid)_{\Hi,\Iw}}\Kcal(\aid)_{\Hi,\Iw}\overset{\can}{\simeq}\Kcal(\aid)_{\Hi,\Iw}$.
\end{Conj}
\begin{Conj}\label{ConjSigmaWeak}
There is an inclusion of $R_{\Sigma,\Hi,\Iw}$-lattices
\begin{equation}\nonumber
\Delta_{R_{\Sigma,\Hi,\Iw}}(T_{\Sigma,\Hi,\Iw})\subset\Det_{R_{\Sigma,\Hi,\Iw}}(0)
\end{equation}
in $\Delta_{R_{\Sigma,\Hi,\Iw}}(T_{\Sigma,\Hi,\Iw})\tenseur_{R_{\Sigma,\Hi,\Iw}}Q(R_{\Sigma,\Hi,\Iw})\overset{\can}{\simeq}Q(R_{\Sigma,\Hi,\Iw})$.
\end{Conj}
Analogues of propositions \ref{PropCompatibleSpec}, \ref{PropCompLevel} and \ref{PropCompWeak} then remain true with the same proofs. 
\begin{Prop}\label{PropCompatibleHida}
Conjecture \ref{ConjSigmaHida} for $R_{\Sigma,\Hi,\Iw}$ implies conjecture \ref{ConjXcaliHeckeSigmaWeak} for $R_{\Sigma,\Iw}$. Conjecture \ref{ConjSigmaHida} for $R_{\Sigma,\Hi,\Iw}$ is equivalent to conjecture \ref{ConjSigmaHida} for $R_{\Sigma',\Hi,\Iw}$ for all $\Sigma'\supset\Sigma$. Conjecture \ref{ConjSigmaHida} implies conjecture \ref{ConjRaidHida} for all $M|N(\Sigma)$ and all minimal prime ideals $\aid\in\Spec\Hecke^{\new}(U(M))$ as well as conjecture \ref{ConjXcaliGm} for all modular specializations $\lambda_{g}$ of $R_{\Sigma,\Hi}$ and for all morphisms $\phi:\Lambda\fleche S$ as in definition \ref{DefDelta}.
\end{Prop}
\begin{proof}
The ring $R_{\Sigma,\Hi,\Iw}$ is by proposition \ref{PropTW} the universal deformation ring parametrizing nearly ordinary deformations of $\rhobar_{f}$. There is thus a morphism
\begin{equation}\nonumber
\psi:R_{\Sigma,\Hi,\Iw}\fleche R_{\Sigma,\Iw}
\end{equation}
coming from the identification of $R_{\Sigma,\Iw}$ with the universal deformation ring parametrizing nearly ordinary deformations of fixed weight. The morphism $\psi$ induces isomorphisms
\begin{align}\nonumber
&T_{\Sigma,\Hi,\Iw}\tenseur_{R_{\Sigma,\Hi,\Iw},\psi}R_{\Sigma,\Iw}\simeq T_{\Sigma,\Iw}\\\nonumber
&M_{\Sigma,\Hi}\tenseur_{R_{\Sigma,\Hi,\Iw}}R_{\Sigma,\Iw}\simeq M_{\Sigma,\Hi}.
\end{align}
Hence $\psi$ induces induces a canonical isomorphism
\begin{equation}\label{EqHidaDelta}
\Delta_{R_{\Sigma,\Hi,\Iw}}(T_{{\Sigma,\Hi,\Iw}})\tenseur_{R_{\Sigma,\Hi,\Iw},\psi}R_{\Sigma,\Iw}\isocan\Delta_{R_{\Sigma,\Iw}}(T_{\Sigma,\Iw}).
\end{equation}
The statements of the proposition follow.
\end{proof}
In order to establish \eqref{EqHidaDelta}, it is enough to appeal to equation \eqref{EqControlHecke} and so this isomorphism holds under the only hypothesis that $f$ is nearly $p$-ordinary: it is only for the sake of concision that we appealed to comparatively much more sophisticated interpretation of $R_{\Sigma,\Hi,\Iw}$ as a universal deformation ring (which is true under our much more stringent ongoing hypotheses). 
\subsubsection{Reduction to classical Iwasawa theory for modular forms of weight $k>2$}\label{SubRedHida}
We assume by way of contradiction that conjecture \ref{ConjSigmaWeak} is false.

Nearly ordinary universal deformation rings give rise to Taylor-Wiles systems with coefficients in $\Lambda_{\Hi,\Iw}$, see \cite[Section 11]{FujiwaraDeformation}. Hence, there exists a well-chosen set $X$ such that the system $\{(R_{\Sigma\cup Q,\Hi,\Iw},\Delta_{\Sigma\cup Q,\Hi,\Iw})\}_{Q\in X}$ is a Taylor-Wiles system over $\Lambda_{\Hi,\Iw}$. We repeat the proof of subsection \ref{SubReduction}. Recall that this proof establishes successively the existence of the following objects.
\begin{enumerate}
\item A prime ideal $\pid$ adequate with respect to $(x_{\infty},y_{\infty})$ generated by a sub-system of parameters of $R_{\infty}$ such that $A=R_{\infty}/\pid$ is a discrete valuation ring flat over $\Ocal$.
\item An integer $N$ such that the image of $\Delta_{\infty}\tenseur_{R_{\infty}}A/\mgot_{A}^{N}$ through $\bar{\psi}_{\infty}=\psi_{\infty}\modulo(\pid,\mgot_{A}^{N})$ is equal to $A/\mgot_{A}^{N+n_{0}}$ for some $n_{0}>0$.
\item An integer $n$ such that $R_{Q(n),n}\tenseur_{\Ocal}A$ surjects onto $A/\mgot_{A}^{N}$.
\item A discrete valuation ring $B$ finite flat over $A$ such that $R_{Q(n)}\tenseur_{\Ocal}B$ is normal.
\item A minimal prime ideal $\aid\in\Spec(R_{Q(n)}\tenseur_{\Ocal}B)$ such that the image of $\Delta_{R_{Q(n)}}\tenseur_{\Ocal}B/\aid$ through $\psi_{\Sigma\cup Q_{n}}$ is not included in $B/\aid$.
\item A $G_{\Q}$-representation $T=T(g)_{\Iw}$ with coefficients in $S=R_{\Sigma\cup Q(n)}\tenseur_{\Ocal}B$ attached to a $p$-odinary modular form $g$ congruent to $f$ verifying
\begin{equation}\label{EqContradictionHida}
\carac_{S}H^{2}_{\et}(\Z[1/p],T(g)_{\Iw})\nmid \carac_{S}\Hun_{\et}(\Z[1/p],T(g)_{\Iw})/Z.
\end{equation}
\end{enumerate}
Among the set of primes $\pid$ allowing such a construction, the set of those such that the $G_{\Q}$-representation $T(g)_{\Iw}$ with coefficients in $B/\aid$ comes for $g$ of weight 2 modular and Steinberg at $p$ is of large codimension. Hence, we can choose $\pid$ such that $g$ is of weight $k>2$ in which case equation \eqref{EqContradictionHida} contradicts proposition \ref{PropContradiction}.
\subsubsection{Proof of corollary \ref{CorMain}}
We repeat the statement of corollary \ref{CorMain} from the introduction.
\begin{CorEnglish}\label{CorMainPreuve}
Assume that $f$ satisfies assumptions \ref{HypNO}, \ref{HypIrrTW}, \ref{HypTamagawa} and \ref{HypNOf} (hence $f$ is $p$-ordinary). Then
\begin{equation}\label{EqDivisibilityPreuve} 
\carac_{\Lambda}H^{2}_{\et}(\Spec\Z[1/p],T(f)_{\Iw})|\carac_{\Lambda}H^{1}_{\et}(\Spec\Z[1/p],T(f)_{\Iw})/\z(f).
\end{equation}
\end{CorEnglish}
\begin{proof}
The divisibility \eqref{EqDivisibilityPreuve} is equivalent to the statement of conjecture \ref{ConjXcaliLambdaWeak} for $f$. According to subsection \ref{SubRedHida}, under the hypotheses of the corollary, conjecture \ref{ConjSigmaWeak} is true and thus conjecture \ref{ConjXcaliLambdaWeak} for $f$ is true by proposition \ref{PropCompatibleHida}.
\end{proof}
If $f$ satisfies assumptions $\ref{HypIrrTW}$, \ref{HypTamagawa} and does not satisfy \ref{HypNOf}, then the divisibility 
\begin{equation}\nonumber 
\carac_{\Lambda}H^{2}_{\et}(\Spec\Z[1/p],T(f)_{\Iw})|\carac_{\Lambda}H^{1}_{\et}(\Spec\Z[1/p],T(f)_{\Iw})/\z(f).
\end{equation}
follows from proposition \ref{PropContradiction} and hence already follows from \cite{KatoEuler}. Hence, corollary \ref{CorMainPreuve} establishes the last remaining case in corollary \ref{CorMain}.  
\subsection{Proof of conjecture \ref{ConjXcaliLambdaWeak} in the remaining cases}
It remains to prove conjecture \ref{ConjXcaliLambdaWeak} when $f$ is not $p$-ordinary but its residual representation is reducible. Assume the conjecture to be false. We repeat the argument of subsection \ref{SubReduction} to obtain a modular form $g$ congruent to $f$ verifying 
\begin{equation}\label{EqContradictionFinal} 
\carac_{\Lambda}H^{2}_{\et}(\Spec\Z[1/p],T(g)_{\Iw})\nmid \carac_{\Lambda}H^{1}_{\et}(\Spec\Z[1/p],T(g)_{\Iw})/Z.
\end{equation}
Either $g$ is not $p$-ordinary with $\pi(g)_{p}$ a Steinberg representation and \eqref{EqContradictionFinal} contradicts proposition \ref{PropContradiction} or it is, in which case $\rho_{g}$ corresponds to a point of $R^{\ord}_{\Sigma,\Hi,\Iw}(\rhobar_{f})$, for which conjecture \ref{ConjSigmaWeak} is true by subsection \ref{SubRedHida}, and \eqref{EqContradictionFinal} contradicts corollary \ref{CorMainPreuve}.

Finally, we have shown the following theorem.
\begin{TheoEnglish}\label{TheoMainPreuve}
Let $f$ be a attached to a modular point of $R_{\Sigma}$ factoring through $R(\aid)$. Assume $\rhobar_{f}$ satisfies assumptions \ref{HypIrrTW}, \ref{HypFlat}, and \ref{HypTamagawa}. Then conjecture \ref{ConjXcaliHeckeSigmaWeak} is true for $R_{\Sigma,\Iw}$ and conjecture \ref{ConjXcaliRaidWeak} is true for $R(\aid)_{\Iw}$. Assume $\rhobar_{f}$ satisfies assumptions \ref{HypIrrTW}, \ref{HypNO}, and \ref{HypTamagawa}. Then conjecture \ref{ConjXcaliHeckeSigmaWeak} is true for $R_{\Sigma,\Iw}$, conjecture \ref{ConjSigmaWeak} is true for $R_{\Sigma,\Hi,\Iw}$, conjecture \ref{ConjXcaliRaidWeak} is true for $R(\aid)_{\Iw}$ and conjecture \ref{ConjAidWeak} is true for $R(\aid)_{\Hi,\Iw}$.
\end{TheoEnglish}
\subsection{Proof of corollary \ref{CorEPW} and theorem \ref{TheoSU}}
\subsubsection{Proof of corollary \ref{CorEPW}}
We repeat the statement of corollary \ref{CorEPW}
\begin{CorEnglish}\label{CorEPWPreuve}
Assume that $\rhobar_{f}$ satisfies the assumptions \ref{HypNO} or \ref{HypFlat}, \ref{HypIrrTW} and \ref{HypTamagawa}. The three following assertions are equivalent.
\begin{enumerate}
\item\label{AssertionFinal} Conjecture \ref{ConjXcaliHeckeSigma} is true for $R_{\Sigma,\Iw}$.
\item\label{AssertionAllMu} For all modular specializations $\lambda$ of $R_{\Sigma}$, conjecture \ref{ConjXcaliLambda} is true for $f_{\lambda}$.
\item\label{AssertionSingleMu} There exists a modular specialization $\lambda$ of $R_{\Sigma}$ such that conjecture \ref{ConjXcaliLambda} is true for $f_{\lambda}$.
\end{enumerate}
If moreover $\rhobar_{f}|_{G_{\qp}}$ is reducible, then $R_{\Sigma}$ may be replaced by $R_{\Sigma,\Hi}$ in assertions \ref{AssertionAllMu} and \ref{AssertionSingleMu} and $R_{\Sigma,\Iw}$ may be replaced by $R_{\Sigma,\Hi,\Iw}$ in assertion \ref{AssertionFinal}.
\end{CorEnglish}
\begin{proof}
According to proposition \ref{PropCompLevel}, the assertions are in decreasing order of logical strength so it is enough to prove that assertion \ref{AssertionSingleMu} implies assertion \ref{AssertionFinal}. Let $f$ be the eigencuspform attached to the modular specialization $\lambda$ for which conjecture \ref{ConjXcaliLambda} is true. Theorem \ref{TheoMainPreuve} states that the specified isomorphism 
\begin{equation}\nonumber
\Delta_{R_{\Sigma,\Iw}}(T_{\Sigma,\Iw})\tenseur Q(R_{\Sigma,\Iw})\isocan Q(R_{\Sigma,\Iw})
\end{equation}
sends $\Delta_{R_{\Sigma,\Iw}}(T_{\Sigma,\Iw})$ into $R_{\Sigma,\Iw}$. By proposition \ref{PropCompWeak}, it follows that conjecture \ref{ConjXcaliLambdaWeak} for $f$ is true and thus that the specified isomorphism 
\begin{equation}\nonumber
\Delta_{\Lambda}(T(f)_{\Iw})\tenseur\Frac(\Lambda)\isocan\Frac(\Lambda)
\end{equation}
sends $\Delta_{\Lambda}(T(f)_{\Iw})$ into $\Lambda$. Hence, there is a commutative diagram of local morphisms:
\begin{equation}\nonumber
\xymatrix{
\Delta_{R_{\Sigma,\Iw}}(T_{\Sigma,\Iw})\ar[r]\ar[d]&R_{\Sigma,\Iw}\ar[d]\\
\Delta_{\Lambda}(T(f)_{\Iw})\ar[r]&\Lambda
}
\end{equation}
Conjecture \ref{ConjXcaliLambda} for $f$, which is true by assertion \ref{AssertionSingleMu}, states that the image of the lowermost horizontal arrow is $\Lambda$. This implies that the image of the uppermost horizontal arrow is $R_{\Sigma,\Iw}$.

If moreover $\rhobar_{f}|G_{\qp}$ is reducible, then the same proof replacing everywhere $R_{\Sigma,\Iw}$ by $R_{\Sigma,\Hi,\Iw}$ proves the ultimate claim.
\end{proof}
\subsubsection{Proof of theorem \ref{TheoSU}}
We repeat the statement of theorem \ref{TheoSU}.
\begin{TheoEnglish}\label{TheoSUPreuve}
Let $p$ be an odd prime and $N$ such that $p\nmid N$. Let $f\in S_{k}(\Gamma_{1}(p^{r})\cap\Gamma_{0}(N))$ be an eigencuspform. Assume that $\rhobar_{f}$ satisfies assumptions \ref{HypIrrTW}, \ref{HypNO} and \ref{HypTamagawa}. Then conjecture \ref{ConjXcaliHeckeSigma} is true for $R_{\Sigma,\Iw}$ and conjecture \ref{ConjSigmaHida} is true for $R_{\Sigma,\Hi,\Iw}$.
\end{TheoEnglish}
\begin{proof}
In view of theorem \ref{TheoMainPreuve} and corollaries \ref{CorMainPreuve} and \ref{CorEPWPreuve}, it is enough to prove that there exists a modular specialization of $R_{\Sigma,\Hi,\Iw}$ of weight $k>2$ such that
\begin{equation}\nonumber
\carac_{\Lambda}H^{1}_{\et}(\Z[1/p],T(g)\tenseur\Lambda)/\z(g)|\carac_{\Lambda}H^{2}_{\et}(\Z[1/p],T(g)\tenseur\Lambda).
\end{equation}
By \cite[Section 17.13]{KatoEuler} (see especially the short exact sequence at the end of that section) and \cite[Theorem 3.14]{OchiaiColeman}, this is equivalent to the main conjecture in Iwasawa theory of modular forms of R.Greenberg and B.Mazur; see for instance \cite[Conjecture 7.4]{OchiaiMainConjecture} for a precise statement. Hence, it is true by \cite[Theorem 3.29]{SkinnerUrban} once we check that the hypotheses of this theorem are verified. Under the hypotheses of \ref{TheoSUPreuve}, the hypotheses \textbf{(dist)} and \textbf{(irr)} of \cite[Theorem 3.29]{SkinnerUrban} are true respectively by our assumption \ref{HypNO} and assumption \ref{HypIrrTW}. The third hypothesis of \cite[Theorem 3.29]{SkinnerUrban} is our assumption \ref{HypTamagawa}. The first, fourth and last hypotheses of \cite[Theorem 3.29]{SkinnerUrban} are imposed there in order to establish proposition \ref{PropContradiction}, but we have checked that this proposition remained true without any supplementary assumptions (the main reason for this difference is that \cite{SkinnerUrban} do not use the improvements of \cite{KatoEuler} contained in \cite{OchiaiColeman,OchiaiEuler,OchiaiMainConjecture}). 
\end{proof}
\paragraph{Acknowledgments}The author thanks L.Clozel, F.Jouve and J.Riou for helpful discussions and comments about this work. He is particularly grateful to T.Fukaya for providing him a copy of \cite{FukayaKatoSharifi}. 
\bibliographystyle{amsalpha}
\def\Dbar{\leavevmode\lower.6ex\hbox to 0pt{\hskip-.23ex \accent"16\hss}D}
  \def\cfac#1{\ifmmode\setbox7\hbox{$\accent"5E#1$}\else
  \setbox7\hbox{\accent"5E#1}\penalty 10000\relax\fi\raise 1\ht7
  \hbox{\lower1.15ex\hbox to 1\wd7{\hss\accent"13\hss}}\penalty 10000
  \hskip-1\wd7\penalty 10000\box7}
  \def\cftil#1{\ifmmode\setbox7\hbox{$\accent"5E#1$}\else
  \setbox7\hbox{\accent"5E#1}\penalty 10000\relax\fi\raise 1\ht7
  \hbox{\lower1.15ex\hbox to 1\wd7{\hss\accent"7E\hss}}\penalty 10000
  \hskip-1\wd7\penalty 10000\box7} \def\Dbar{\leavevmode\lower.6ex\hbox to
  0pt{\hskip-.23ex \accent"16\hss}D}
  \def\cfac#1{\ifmmode\setbox7\hbox{$\accent"5E#1$}\else
  \setbox7\hbox{\accent"5E#1}\penalty 10000\relax\fi\raise 1\ht7
  \hbox{\lower1.15ex\hbox to 1\wd7{\hss\accent"13\hss}}\penalty 10000
  \hskip-1\wd7\penalty 10000\box7}
  \def\cftil#1{\ifmmode\setbox7\hbox{$\accent"5E#1$}\else
  \setbox7\hbox{\accent"5E#1}\penalty 10000\relax\fi\raise 1\ht7
  \hbox{\lower1.15ex\hbox to 1\wd7{\hss\accent"7E\hss}}\penalty 10000
  \hskip-1\wd7\penalty 10000\box7} \def\Dbar{\leavevmode\lower.6ex\hbox to
  0pt{\hskip-.23ex \accent"16\hss}D}
  \def\cfac#1{\ifmmode\setbox7\hbox{$\accent"5E#1$}\else
  \setbox7\hbox{\accent"5E#1}\penalty 10000\relax\fi\raise 1\ht7
  \hbox{\lower1.15ex\hbox to 1\wd7{\hss\accent"13\hss}}\penalty 10000
  \hskip-1\wd7\penalty 10000\box7}
  \def\cftil#1{\ifmmode\setbox7\hbox{$\accent"5E#1$}\else
  \setbox7\hbox{\accent"5E#1}\penalty 10000\relax\fi\raise 1\ht7
  \hbox{\lower1.15ex\hbox to 1\wd7{\hss\accent"7E\hss}}\penalty 10000
  \hskip-1\wd7\penalty 10000\box7}
\providecommand{\bysame}{\leavevmode\hbox to3em{\hrulefill}\thinspace}
\providecommand{\MR}{\relax\ifhmode\unskip\space\fi MR }
\providecommand{\MRhref}[2]{%
  \href{http://www.ams.org/mathscinet-getitem?mr=#1}{#2}
}
\providecommand{\href}[2]{#2}

\end{document}